\documentclass[11pt]{article}

\RequirePackage{amsthm,amsmath,amsfonts,amssymb,stmaryrd}

\RequirePackage[colorlinks,citecolor=blue,urlcolor=blue,menucolor=black,linkcolor=black]{hyperref}
\RequirePackage{graphicx}

\usepackage[total={7in, 9in}]{geometry}

\theoremstyle{plain}
\newtheorem{lemma}{Lemma}[section]
\newtheorem{theorem}[lemma]{Theorem}
\newtheorem{cor}[lemma]{Corollary}

\newtheorem{proposition}[lemma]{Proposition}

\newtheorem{conj}{Conjecture}

\theoremstyle{remark}
\newtheorem{rem}[lemma]{Remark}

\newcommand\linf{\lambda_1^\infty}
\newcommand\vep{\varepsilon}

\newcommand{\sinc}{\operatorname{sinc}}

\usepackage{authblk}

\newlength{\affilskip}
\begin{document}

\title{A branching particle system as a model of semipushed fronts}
\author[*]{Julie Tourniaire}
\affil[*]{Faculty of Mathematics, University of Vienna, Oskar-Morgenstern-Platz 1, 1090 Wien, Austria\vspace{\affilskip}}

\maketitle

\begin{abstract}
    We consider a system of particles performing a one-dimensional dyadic branching
    Brownian motion with space-dependent branching rate, negative drift $-\mu$ and
    killed upon reaching $0$, starting with $N$ particles. More precisely, particles
    branch at rate $\rho/2$ in the interval $[0,1]$, for some $\rho>1$, and at rate
    $1/2$ in $(1,+\infty)$. The drift $\mu(\rho)$ is chosen in such a way that,
    heuristically, the system is critical in some sense: the number of particles
    stays roughly constant before it eventually dies out. This particle system can
    be seen as an analytically tractable model for fluctuating fronts, describing the
    internal mechanisms driving the invasion of a habitat by a cooperating population.
    Recent studies from Birzu, Hallatschek and Korolev suggest the existence of
    three classes of fluctuating fronts:  pulled, semipushed and fully pushed fronts.
    Here, we rigorously verify and make precise this classification and focus on
    the semipushed regime.  This complements previous results from Berestycki,
    Berestycki and Schweinsberg for the case $\rho=1$.
\end{abstract}

\tableofcontents

\section{Introduction}

In this article, we are interested in the underlying dynamics of travelling wavefronts arising from certain reaction diffusion equations. Formally, the front is represented by a branching Brownian motion (BBM) with absorption at zero and negative drift $-\mu$. This system can be seen as a co-moving frame following the particles located at the tip of the front. In this framework, the drift $\mu$ is interpreted as the speed of the wave.

In this introductory section, we first  motivate our analysis with the results of some recent studies and state an informal version of the main theorem  in Section \ref{sec:relatedmodel}. In Section \ref{sec;csbp}, we recall some well known facts on continuous-state branching processes. The model and the results are given in Section \ref{sec:model} and a sketch of the proof is outlined in Section \ref{sec:skproof}. In Section \ref{pt:Laplacian}, we explain the connection between the model defined in Section \ref{sec:relatedmodel} and the generalised principal eigenvalue of the perturbed Laplacian on the half-line. We then discuss the link between our model and previous work on pulled fronts and branching processes in Section \ref{relmod} and give a biological interpretation of the result in Section \ref{sec:allee}.

\subsection{Noisy FKPP-type equations and semipushed fronts}\label{sec:relatedmodel}

This work is motivated by the results of recent work by Birzu, Hallatschek and Korolev \cite{Birzu2018,Birzu:2020up} on the noisy FKPP-type equation
\begin{equation}
    u_t=\frac{1}{2}u_{xx}+r_{0} u(1-u)(1+Bu)+\frac{1}{\sqrt{N}}\Gamma(u)W(t,x). \label{eq:FKPPhall}
\end{equation}  From a biological standpoint, Equation \eqref{eq:FKPPhall} models the invasion of an uncolonised habitat by a species:  $u$ corresponds to the population density, $r_0$ to the per-capita growth rate at low densities, $B$ is a positive parameter scaling the strength of cooperation between the individuals, $N$ is the local number of particles at   equilibrium, $\Gamma$ stands for the strength of the demographic fluctuations and $W$ is a Gaussian white noise. The numerical experiments and analytical arguments from \cite{Birzu2018,Birzu:2020up} suggest the existence of three regimes in Equation \eqref{eq:FKPPhall}: the \textit{pulled} regime for $B\leqslant 2$, the \textit{semipushed} or \textit{weakly pushed} regime for $B\in(2,B_c)$, for some $B_c>2$, and the \textit{fully pushed} regime, for $B\geqslant B_c$.

The notion of pulled and pushed waves was first introduced by Stokes \cite{stokes1976two} in PDE theory. The distinction between the pulled and pushed regimes in  \eqref{eq:FKPPhall} is based on the asymptotic spreading speed $v$ of the solutions of the limiting reaction diffusion equation  ($N=\infty$),
\begin{equation}
    u_t=\frac{1}{2}u_{xx}+f(u), \label{limit:PDEh}
\end{equation}
with
\begin{equation}
    f(u):=r_0u(1-u)(1+Bu). \label{reactionterm}
\end{equation}
It is a known fact (see e.g.~\cite{Hadeler:1975uk}) that Equation (\ref{limit:PDEh}) has a one-parameter family of front solutions $u(t,x)=\varphi_c(x-ct)$ for  $c\geqslant c_{\min}$, for some $c_{\min}>0$. Moreover, it was shown \cite{stokes1977nonlinear} that the asymptotic spreading speed $v$ of any solution to Equation \eqref{limit:PDEh} with compactly supported initial data is equal to the minimal speed $c_{\min}$. We refer to  \cite{these}, Chapter 1, for further details on the convergence of such solutions.
An invasion is then said to be ``pulled'' if $c_{\min}$ coincides with the asymptotic speed $c_0$ of the linearised equation
\begin{equation*}
    u_t=\frac{1}{2}u_{xx}+f'(0)u,
\end{equation*}
and ``pushed'' if $c_{\min}>c_0$. In Equation (\ref{limit:PDEh}), the transition between pulled and pushed fronts occurs at $B=2$ \cite{Hadeler:1975uk}.
As observed in \cite{Birzu2018}, the addition of demographic fluctuations in \eqref{limit:PDEh} uncovers  a second phase transition within the pushed regime. This leads to the distinction of two classes of pushed fronts:  \textit{semipushed} (or \textit{weakly pushed}) fronts and  \textit{fully pushed} fronts. The effect of fluctuations on pulled fronts has already been widely studied in the literature. A rich theory based on the work of  Brunet, Derrida and co-authors \cite{Brunet1997,Brunet2006,Brunet2006a} describes the behaviour of the front solutions of (\ref{eq:FKPPhall}) for $B=0$. The spreading speed of these solutions admits a correction of order $\log(N)^{-2}$ compared to the one of the limiting PDE \eqref{limit:PDEh}. In this sense, fluctuations have a huge impact on pulled fronts (see Section \ref{relmod} for further details). Moreover, the genealogy at the tip of the front is expected to be described by a Bolthausen--Sznitman coalescent over a time scale of order $\log(N)^3$, which suggests that the particles located at the tip of the front evolve as a population undergoing natural selection.

On the other hand, pushed fronts are expected to be less sensitive. In \cite{Birzu2018}, it is numerically observed that for $B>B_c$, the fluctuations in the position of the front and in the genetic drift occur on a time scale of order $N$, which may indicate the presence of Kingman's coalescent (a coalescent with binary mergers). This is consistent with the fact that the population in the bulk behaves like a neutral population. However, for intermediate values of $B$, that is $B\in(2,B_c)$, the fluctuations appear on a shorter time scale, namely $N^\gamma$ with $\gamma\in(0,1)$.  This intermediate region is defined as the semipushed regime.

In this work, we propose an analytically tractable particle system to investigate the microscopic mechanisms leading to semipushed invasions. This model is an extension of the one studied by Berestycki, Berestycki and Schweinsberg \cite{Berestycki2010} to prove the conjecture on the genealogy of pulled fronts. Similarly, we are able to exhibit the time scale and the structure of the genealogy of our particle system. Based on the branching particle system analysed in \cite{Berestycki2010}, we consider a branching Brownian motion with absorption at $0$, negative drift $-\mu$ and a space-dependent branching rate $r(x)$ of the form
\begin{equation}
    \label{def:r}
    r(x)=\frac{1}{2}+\frac{\rho-1}{2}\mathbf{1}_{x\in[0,1]},
\end{equation}
for some $\rho\geqslant 1$. As mentioned above, this system is a toy model for what happens to the right of the front. Hence, the parameter $\rho$ plays the same role as $B$ in Equation \eqref{eq:FKPPhall} and thus scales the strength of the cooperation between the particles.

We assume that the system starts with $N$ particles located at $1$. We denote by $N_t$ the number of particles alive in the particle system at time $t$ and consider the rescaled number of particles $\bar{N}_t=N_t/N$.
Essentially, our result is the following:
\begin{theorem}[informal version]\label{informal:version} Let $\rho_1:=1+\frac{\pi^2}{4}$.
    There exists $\rho_2>\rho_1$ such that for all $\rho\in(\rho_1,\rho_2)$, there exists $\mu(\rho)>1$ and $\alpha=\alpha(\rho)\in(1,2)$ such that, if we consider the BBM with branching rate \eqref{def:r} and drift $-\mu(\rho)$, the process $(\bar{N}_{N^{\alpha-1}t})_{t>0}$ converges in law to an $\alpha$-stable continuous-state branching process as $N$ goes to infinity.
  Moreover, the exponent $\alpha$  is an increasing function of $\rho$ such that $\alpha(\rho)\to i$ as $\rho\to\rho_i$, $i=1,2$.

\end{theorem}

This result is consistent with the observations made on the fluctuations in \cite{Birzu2018} and with the genealogical structure proposed in \cite{Birzu:2020up} for semipushed fronts. Indeed, it is known that the genealogy corresponding to an $\alpha$-stable continuous-state branching process is given by a Beta$(2-\alpha,\alpha)$-coalescent \cite{Birkner2005}. Theorem \ref{informal:version} thus suggests that the genealogy of the particle system in the semipushed regime interpolates between Bolthausen--Sznitman ($\alpha=1$) and Kingman ($\alpha=2$) coalescents.

We refer to Section \ref{sec:model} for a precise statement of Theorem \ref{informal:version} and to Section \ref{sec;csbp} for a definition of continuous-state branching processes.

\subsection{Continuous-state branching processes}\label{sec;csbp}
We recall known facts about continuous-state branching processes (CSBP) and, more specifically, the family of $\alpha$-stable CSBP, for $\alpha\in[1,2]$ (see e.g.~\cite{berestycki2009recent,Birkner2005}). A continuous-state branching process is a $[0, \infty]$-valued Markov process $(\Xi(t), t \geq 0)$ whose transition functions satisfy the branching property $p_t(x + y, \: \cdot) = p_t(x, \: \cdot) * p_t(y, \: \cdot),$ which means that the sum of two independent copies of the process starting from $x$ and $y$ has the same finite-dimensional distributions as the process starting from $x + y$.  It is well-known that continuous-state branching processes can be characterised by their branching mechanism, which is a function $\Psi: [0, \infty) \rightarrow \mathbb{R}$.  If we exclude processes that can make an instantaneous jump to $\infty$, the function $\Psi$ is of the form $$\Psi(q) = \gamma q + \beta q^2 + \int_0^{\infty} (e^{-qx} - 1 + qx \mathbf{1}_{x \leq 1}) \: \nu(dx),$$ where $\gamma \in \mathbb{R}$, $\beta \geq 0$, and $\nu$ is a measure on $(0, \infty)$ satisfying $\int_0^{\infty} (1 \wedge x^2) \: \nu(dx) < \infty$.  If $(\Xi(t), t \geq 0)$ is a continuous-state branching process with branching mechanism $\Psi$, then for all $\lambda \geq 0$,
\begin{equation}\label{csbpLaplace}
    E[e^{-\lambda \Xi(t)} \,| \, \Xi_0 = x] = e^{-x u_t(\lambda)},
\end{equation}
where $u_t(\lambda)$ can be obtained as the solution to the differential equation
\begin{equation}\label{diffeq}
    \frac{\partial}{\partial t} u_t(\lambda) = -\Psi(u_t(\lambda)), \hspace{.5in} u_0(\lambda) = \lambda.
\end{equation}
We will be interested in $\alpha$-stable CSBP for $\alpha\in[1,2]$, for which the branching mechanism $\Psi$ is of the form
\begin{equation}
    \label{eq:Psi}
    \Psi(u) =
    \begin{cases}
        -au + b u^\alpha & \text{if } \alpha \in (1,2], \\
        -au + b u\log u  & \text{if } \alpha = 1.
    \end{cases}
\end{equation}
It is known that in this case, the CSBP does not explode in finite time, i.e.~Grey's condition is satisfied. The 2-stable CSBP is also known as the \emph{Feller diffusion} and the 1-stable CSBP as \emph{Neveu's CSBP}.

\subsection{The model: assumptions and main result} \label{sec:model}

We consider a dyadic branching Brownian motion with killing at zero, negative drift $-\mu$ and position-dependent branching rate $r:[0,\infty)\to\mathbb{R}$ given by
\[
    r(x) = \begin{cases}
        \rho/2 & x\in[0,1], \\
        1/2    & x>1,
    \end{cases}
\]
for some parameter $\rho\geqslant1$. We denote by $\mathcal N_t$ the set of particles in the system at time $t$ and for all $v\in\mathcal N_t$, we denote by $X_v(t)$ the position of the particle $v$ at time $t$. Furthermore, we write $N_t = |\mathcal N_t|$ for the number of particles in the system at time $t$. The drift $\mu$ is chosen with respect to $\rho$ in such a way that the number of particles in the system  stays roughly constant. Depending on the value of $\rho$, $\mu$ is equal to $1$ (\textit{pulled regime}) or $\mu$ is strictly larger than $1$ (\textit{pushed regime}).

In practice, $\mu=\mu(\rho)$ is a function of $\rho$ related to the generalised principal  eigenvalue $\linf$ of a certain differential operator (see Section \ref{sec:skproof} for further details). More precisely, we have

\begin{itemize}
    \item If $\rho\leq1+\frac{\pi^2}{4}$, then
          \begin{equation}
              \mu=1. \label{defmu1}
          \end{equation}
    \item If $\rho>1+\frac{\pi^2}{4}$, then  $\mu$ is the unique solution of
          \begin{equation}
              \frac{\tan(\sqrt{\rho-\mu^2})}{\sqrt{\rho-\mu^2}}=-\frac{1}{\sqrt{\mu^2-1}}, \quad \text{such that } \quad \rho-\mu^2\in\left[\frac{\pi^2}{4},\pi^2\right].\label{def:murho}
          \end{equation}
\end{itemize}
In terms of $\linf$, we have $\linf=0$ for $\rho<1+\frac{\pi^2}{4}$, $\linf>0$ for $\rho>1+\frac{\pi^2}{4}$ and the definition of $\mu$ given by Equations (\ref{defmu1}) and \eqref{def:murho} is equivalent to
\begin{equation}
    \label{def:mu} \mu=\sqrt{1+2\linf},
\end{equation}
so that
\begin{equation*}
    \mu>1 \quad \iff \quad \linf >0 \quad \iff \quad\rho>1+\frac{\pi^2}{4}.
\end{equation*}
The branching Brownian motion with absorption at $0$, branching rate $r(x)$ and drift $-\mu$ is now fully defined. Let us define the exponent $\alpha$: for $\mu>1$, we set
\begin{equation}
    \alpha=\frac{\mu+\sqrt{\mu^2-1}}{\mu-\sqrt{\mu^2-1}}.\label{defalpha:rho}
\end{equation}

We now define two regimes of interest for the parameter $\rho$. The first one corresponds to  the \textit{pushed} regime:
\begin{equation}\label{hpushed}
    \quad
    \rho>\rho_1,\tag{H$_{psh}$}
\end{equation}
where
\begin{equation*}
    \rho_1=1+\frac{\pi^2}{4}.
\end{equation*}
It turns out that the transition between the \textit{weakly pushed} and the \textit{fully pushed} regimes occurs when $\alpha=2$, which corresponds to the critical value of $\mu$,
\begin{equation}\label{def:muc}
    \mu_c=\frac{3}{4}\sqrt{2}.
\end{equation}
Therefore, the \textit{weakly pushed} regime corresponds to the following range of the parameter $\rho$:
\begin{equation}\tag{H$_{wp}$}
    \label{hwp}
    \rho_1<\rho<\rho_2,
\end{equation}
where $\rho_2$ is the unique solution of
\begin{equation*}
    \frac{\tan\left(\sqrt{\rho-\mu_c^2}\right)}{\sqrt{\rho-\mu_c^2}}=-\frac{1}{\sqrt{\mu_c^2-1}}\quad s.t. \quad \rho-\mu_c^2 \in\left[\frac{\pi^2}{4},\pi^2\right].
\end{equation*}
Numerically, we have $\rho_1\approx 3.467$ and $\rho_2\approx 4.286$.
In this regime, we prove the following convergence result, which is the main result of this article:
\begin{theorem}\label{semipushedfr}
    Assume that \eqref{hwp} holds and suppose that the system initially starts with $N$ particles located at $1$. Then there exists an explicit constant $\sigma(\rho)>0$ such that, if we define $\bar{N_t}= \sigma(\rho) N_t/N$,  as $N\to\infty$, the finite-dimensional distributions of the processes $(\overline N_{N^{\alpha-1}t})_{t>0}$ converge to the finite-dimensional distributions of an $\alpha$-stable CSBP starting from 1, where $\alpha$ is given by Equation \eqref{defalpha:rho}.
\end{theorem}

A more general version of Theorem~\ref{semipushedfr} is stated in Theorem \ref{ThNt}. In addition, an explicit formula for $\sigma(\rho)$ is given in Section \ref{nbrpart:CSBP} (see Equation \eqref{def:sigma}). We strongly believe that this result can be completed with the study of the cases $\rho\in[1,\rho_1)$ and $\rho\in(\rho_2,+\infty)$. The expected convergence results are summarised in the following conjectures. This will be the subject of future work.

\begin{conj} If $\rho<\rho_1$, under suitable assumptions on the initial configurations, the finite-dimensional distributions of the processes $(\overline N_{(\log N)^3t})_{t>0}$ converge to the finite-dimensional distributions of a 1-stable (Neveu's) CSBP starting from 1, as $N\to \infty$.
\end{conj}
\begin{conj}If $\rho > \rho_2$,  under suitable assumptions on the initial configurations, the finite-dimensional distributions of the processes $(\overline N_{Nt})_{t>0}$ converge  to the finite-dimensional distributions of a Feller diffusion starting from 1, as $N\to \infty$.
\end{conj}

The proof of Theorem \ref{semipushedfr} relies on  first and  second moment estimates for several processes. The assumptions (\ref{hpushed}) and (\ref{hwp}) are used to estimate these moments in the \textit{weakly pushed} regime. The first moment estimates (see Sections \ref{sec:fm} and \ref{sec:fmR}) will be established under assumption (\ref{hpushed}), so that they can also be used to investigate the \textit{fully pushed} regime, whereas the second moment calculations will  require the assumption (\ref{hwp}).

One can also investigate systems with more general branching rates of the form
\begin{equation*}
    r(x)=\frac{1}{2}+\frac{\rho-1}{2}f(x), \quad x\in[0,\infty),
\end{equation*}
for a function $f$ that is compactly supported (or even a function that converges quickly  to zero). In this case, the spectrum and eigenvectors are not necessarily explicit, but one can still analyse the system using spectral methods.

\subsection{Comparison with results on fluctuating fronts}

In the particle system, we say that the pulled regime corresponds to  $\rho\in\left[1,\rho_1\right)$, the weakly pushed regime to $\rho\in(\rho_1,\rho_2)$ and the fully pushed regime to $\rho>\rho_2$. From a biological standpoint, the process $N_t$ is related to \emph{the number of descendants left by the early founders} mentioned in \cite{Birzu:2020up}. Moreover, CSBPs can be seen as scaling limits of Galton-Watson processes, with  associated genealogical structures \cite{berestycki2009recent}. In this sense, the convergence results stated in Theorem \ref{semipushedfr} and in the two conjectures are consistent with the observations on the genealogical trees made in \cite{Birzu:2020up}. In the pulled regime, the genealogy of the particles at the tip of the front  is the one of a population undergoing selection, that is a Bolthausen--Sznitman coalescent. We know since the work of Bertoin and Le Gall \cite{bertoin2000bolthausen} that it is precisely the genealogy associated with Neveu's CSBP. Similarly, we know that the genealogy associated to the $\alpha$-stable CSBP and the Feller diffusion are respectively the Beta$(2-\alpha,\alpha)$-coalescent and Kingman's coalescent \cite{Birkner2005}. Again, this is exactly what is observed in \cite{Birzu:2020up}.

Moreover, note that the transitions between the three regimes occur at the same critical values of $\mu$ and $v$ (recall from Section \ref{sec:relatedmodel} that $v$ refers to the asymptotic spreading speed of the solutions of Equation (\ref{limit:PDEh})). Indeed, consider Equation \eqref{eq:FKPPhall} with $r_0=\frac{1}{2}$. Therefore $c_0=1$ and the invasion speed $v$ is given by \cite{Hadeler:1975uk}
\begin{equation}
    v=v(B)=\begin{cases}\label{eq:speedPDE}
        \sqrt{2r_0}= 1                                   & \text{if} \quad B\leqslant 2 \\
        \frac{1}{2}\sqrt{r_0B}\left(1+\frac{2}{B}\right) & \text{if} \quad B>2.
    \end{cases}
\end{equation}
In the particle system, note that the drift is also equal to $1$ in the pulled regime (see Equation \eqref{defmu1}). In both cases, the transition between the pushed and the pulled regimes happens when the propagation speed, $\mu$ or $v$, becomes larger than $1$, that is when $\rho>1+\frac{\pi^2}{4}$ in the particle system and $B> 2$ in the noisy FKPP Equation (\ref{limit:PDEh}). Similarly, the transition between weakly and fully pushed waves occurs for the same critical value of the invasion speed. Following \cite{Birzu2018},  consider $\tilde{\alpha}$ such that (see \cite{Birzu2018}, Equation (8))
\begin{equation}\label{defalphat}
    \tilde{\alpha}=\begin{cases}
        1+\frac{2\sqrt{1-c_0^2/v^2}}{1-\sqrt{1-c_0^2/v^2}} & \text{if} \quad\frac{v}{c_0}\in \left(1,\frac{3}{4}\sqrt{2}\right) \\
        2                                                  & \text{if}\quad\frac{v}{c_0} \geqslant \frac{3}{4}\sqrt{2}.
    \end{cases}
\end{equation}
Birzu, Hallatschek and Korolev observe that the fluctuations in the pushed regime appear on a time scale $N^{\tilde{\alpha}-1}$, so that the transition between the weakly and fully pushed regimes occurs at $v=\frac{3}{4}\sqrt{2}c_0$. This is consistent with Theorem \ref{semipushedfr}: if $r_0=\frac{1}{2}$, then $c_0=1$, so that the transition occurs at $v=\frac{3}{4}\sqrt{2}$, which corresponds to the critical value $\mu_c$ from Equation \eqref{def:muc}, delineating the semipushed and the pushed regimes. In addition, for $c_0=1$, we have
\begin{equation*}
    \tilde{\alpha}=\begin{cases}
        \frac{v+\sqrt{v^2-1}}{v-\sqrt{v^2-1}} & \text{if} \quad v\in \left(1,\frac{3}{4}\sqrt{2}\right) \\
        2                                     & \text{if}\quad v \geqslant \frac{3}{4}\sqrt{2}.
    \end{cases}
\end{equation*}
which seems to indicate the existence of a universality class given our definition of $\alpha$ (see Equation \eqref{defalpha:rho}). In particular, note that the exponent $\alpha$ (resp. $\tilde{\alpha}$) depends on $\rho$ (resp. $B$) only through the drift $\mu$ (resp. the speed $v$). This can be explained by the fact that the particles causing the jumps in the CSBP stay far away from the region in which the branching rate depends on $\rho$ (see below for further explanations).

We now investigate the asymptotic behaviour of $\mu$ and $v$ as the cooperation parameters $\rho$ and $B$ tend to their critical values. First, note that Equation \eqref{eq:speedPDE} implies that, for $r_0=\frac{1}{2}$,
\begin{equation*}
    v(B)\sim \frac{1}{2}\sqrt{\frac{B}{2}}\quad \text{as} \quad  B\to\infty.
\end{equation*}
On the other hand, by definition of $\mu$ (see Equation (\ref{def:murho})), we have $\frac{\pi^2}{4}\leqslant \rho-\mu^2\leqslant \pi^2$, so that
\begin{equation*}
    \mu\sim \sqrt{\rho}\quad \text{as} \quad\rho \to \infty.
\end{equation*}
When $B\to 2$, $B>2$, a second order Taylor expansion gives that
\begin{equation*}
    v(B)\sim 1+\frac{(B-2)^2}{16}.
\end{equation*}
Additionally, when $\rho\to\rho_1$, $\rho>\rho_1$, one can show that $\mu\to 1$ and the first order expansion of each term in Equation \eqref{def:murho} gives
\begin{equation*}
    \mu^2-1\sim\frac{1}{4}\left(\rho-1-\frac{\pi^2}{4}\right)^2,
\end{equation*}
so that we have
\begin{equation*}
    \mu\sim1+\frac{1}{8}\left(\rho-1-\frac{\pi^2}{4}\right)^2.
\end{equation*}
The similar asymptotic behaviours of $\mu$ and $\rho$, as well as the three regimes observed in the particle system support the hypothesis of the existence of a universality class. This is illustrated in Figure~\ref{fig:driftv} and Figure~\ref{fig:semipushed}.

\begin{figure}[t]
    \begin{center}
        \includegraphics[trim=25  0 25  10,scale=0.6]{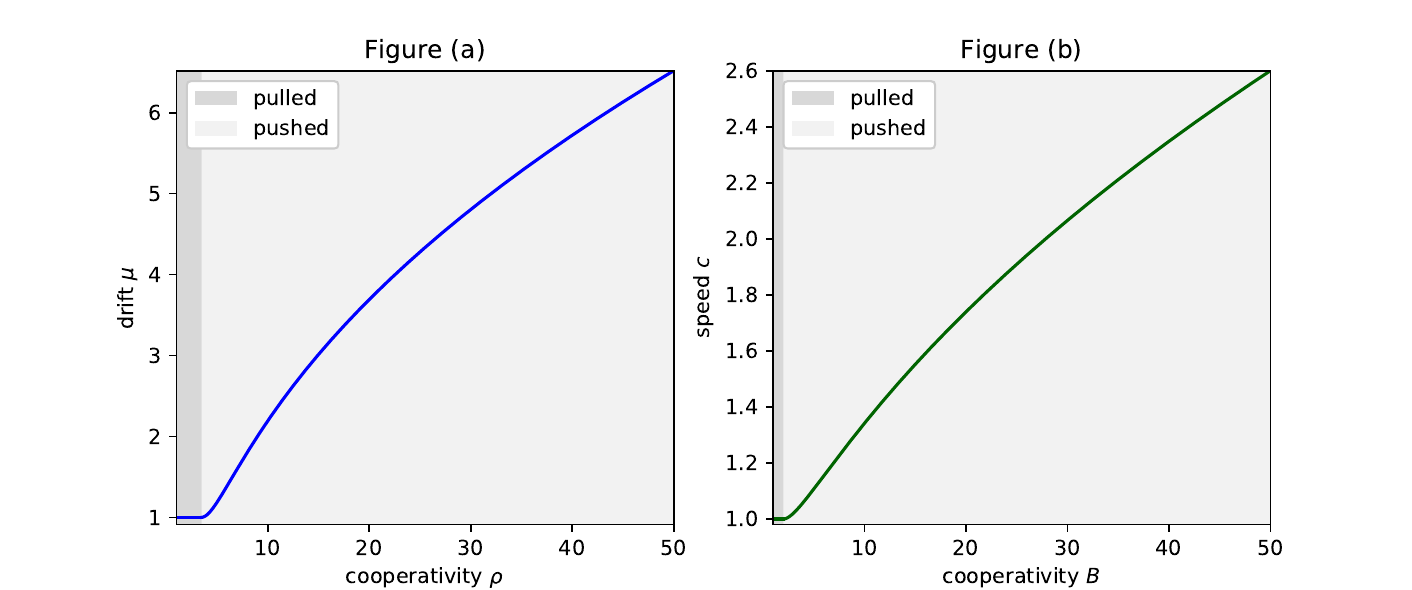}
        \caption{The expansion velocity as a function of cooperativity. Figure (a): in the particle system. Graph of $\mu$ as a function of $\rho$ (see Equations \eqref{defmu1} and \eqref{def:murho}).  The transition between the pulled and the pushed regimes occurs at $\rho_1=1+\frac{\pi^2}{4}\approx 3.47$.  Figure (b): in the PDE \eqref{limit:PDEh}. Graph of $v$ as a function of $B$ (see Equation \eqref{eq:speedPDE}) for $r_{0}=\frac{1}{2}$. The transition between the pulled and the pushed regimes occurs at $B=2$. Note that $\mu$ and $v$ have the same asymptotic behaviour when $\rho$ and $B$ tend to $+\infty$.}
        \label{fig:driftv}
    \end{center}
\end{figure}

\begin{figure}[t]
    \begin{center}
        \includegraphics[trim=25  0 25  10,scale=0.6]{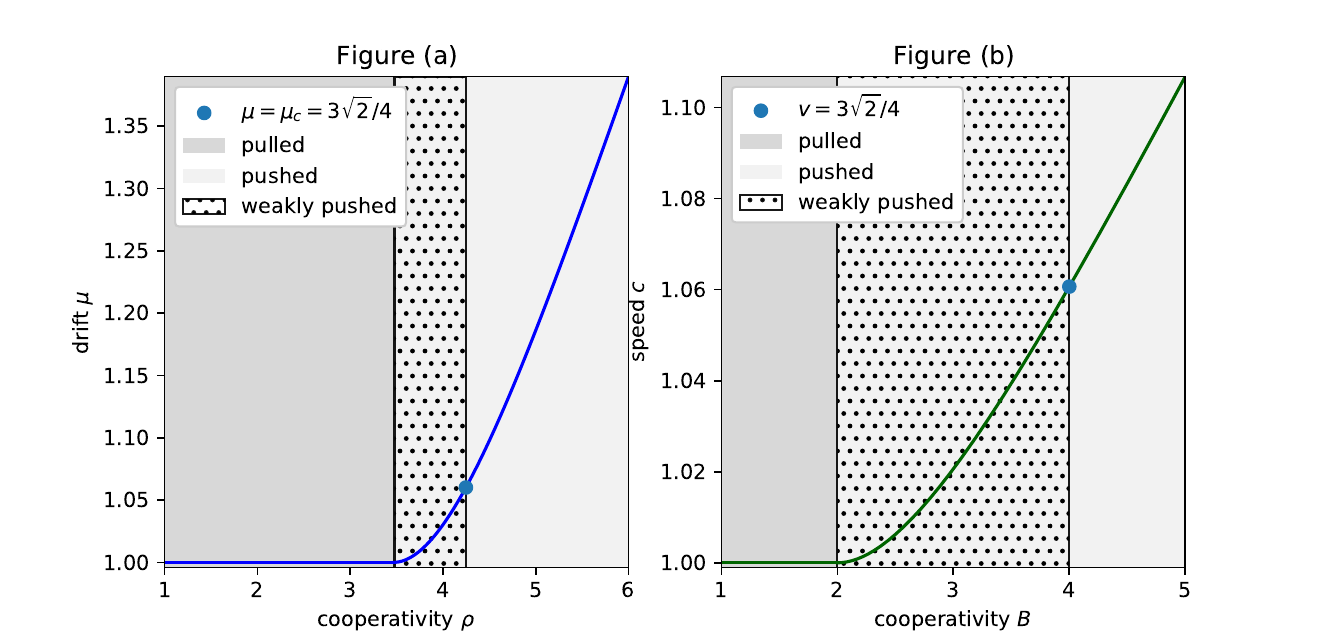}\\
        \caption[]{The expansion velocity as a function of cooperativity. Figure (a): in the particle system. Graph of $\mu$ as a function of $\rho$ (see \eqref{defmu1} and \eqref{def:murho}).  The weakly pushed regime  corresponds to $\mu\in(1,\mu_c)$. The transition between the weakly pushed and fully pushed regime occurs at $\rho=\rho_2$ (see \eqref{hwp}).  Figure (b): in the PDE. Graph of $v$ as a function of $B$ (see Equation \eqref{eq:speedPDE}) for $r_0=\frac{1}{2}$. In the noisy FKPP equation, the transition between weakly pushed and fully pushed waves occurs when $v=\mu_c$ (see \eqref{defalphat}), which corresponds to $B=4$. }
        \label{fig:semipushed}
    \end{center}
\end{figure}

\subsection{Overview of the proof}\label{sec:skproof}

The strategy of the proof is inspired by the work of Berestycki, Berestycki and Schweinsberg \cite{Berestycki2010}, who treated the case of a constant branching rate,   that is $\rho=1$. The main idea  is to introduce an additional barrier at a level $L$, depending on $N$, in such a way that the jumps of the limit of the rescaled process $\bar{N}$ are caused by particles that reach $L$.  In the case $\rho=1$, one chooses $L = \log N + 3\log \log N$, and it is reasonable to believe that this choice will also be suitable  for $\rho < \rho_1$. If $\rho \in (\rho_1,\rho_2)$, we instead choose a barrier at $L = C \log N $ for some $C>0$. In this section, we outline the main ideas used to choose this barrier and to prove the convergence to the $\alpha$-stable CSBP in the case where $\rho\in(\rho_1,\rho_2)$.

As explained in \cite{Berestycki2010}, the role of the barrier is to capture the particles that cause a jump in the CSBP or, equivallently, that will have a number of descendants of order $N$ at a later time. Hence, the level $L$ is chosen such that this number of descendants, at a later time, that is shorter than the time scale of the CSBP, is of order $N$. From this perspective, the behaviour of the particle system is the following:

\begin{enumerate}
    \item Most of the time, the particles stay in the interval $[0,L]$. Therefore the system is well-approximated by a BBM with drift $-\mu$ and branching rate $r(x)$, killed at $0$ and at the additional barrier $L$.
    \item From time to time, on the time scale of the CSBP (which we expect to be $N^{\alpha-1}$) a particle reaches $L$. The barrier $L$ is chosen in such a way that the number of descendants of a particle hitting $L$ is of order $N$ after a short time (compared to the time scale of the CSBP).
    \item In order to deal with these descendants, we let the particles reaching $L$ evolve freely during a time period which is large but of order 1. Following \cite{Berestycki2010}, one can, for example, fix some large constant $y$ and track the descendants when they first reach $L-y$. The number of such descendants will be a random quantity with tail $1/x^\alpha$. This random quantity will be proportional to an \emph{additive martingale} of the BBM rooted at the particle that reaches $L$.
    \item After this large (but independent of $L$) relaxation time, all particles are again in the interval $[0,L]$ and the system  evolves as before.
\end{enumerate}

Thanks to this sketch of proof, one can infer a suitable value of $L$ and justify the definition of the parameter $\mu$. Indeed, the first step implies that most of the time, the system can be approximated by a heat equation in the interval $[0,L]$ with Dirichlet boundary conditions. In other words, if we denote by $\mathcal N^L_t$ the set of particles in the BBM at time $t$ that have stayed in the interval $[0,L]$  until time $t$, the density of particles is given by the many-to-one lemma (see e.g.~\cite{Lawler:2018vn}, p.188):
\begin{lemma}[Many-to-one lemma]
    \label{lem:many-to-one}
    Let $p_t(x,y)$ be the fundamental solution to the PDE
    \begin{equation}
        \begin{cases}
            u_t(t,y) = \frac{1}{2}u_{yy}(t,y) + \mu u_y(t,y) + r(y)u(t,y) \\
            u(t,0) = u(t,L) = 0.
        \end{cases}\label{PDE:A}\tag{A}
    \end{equation}
    Then for every measurable positive  function $f:\mathbb{R}_+\to\mathbb{R}$, we have\footnote{The notation $\mathbb{E}_x$ means that we start with one particle at position $x$.}
    \[
        \mathbb{E}_x\left[\sum_{v\in \mathcal N^L_t} f(X_v(t))\right] = \int_0^L p_t(x,y) f(y)\,dy.
    \]
\end{lemma}
The function $p_t$ can be deduced from the Sturm--Liouville theory. Since (A) is not self-adjoint, we first define a function $\tilde p_t$ in such a way that
\begin{equation}
    p_t(x,y) = e^{\mu(x-y)+ \left(\frac{1}{2}-\frac{\mu^2}{2}\right)t}{\tilde p_t(x,y)}.\label{eq:ptqt}
\end{equation}
A direct computation shows that ${ \tilde p_t(x,y)}$ is the fundamental solution to the self-adjoint PDE
\begin{equation}
    \label{PDE:B}\tag{B}
    \begin{cases}
        u_t(t,y) = \frac{1}{2}u_{y y}(t,y) + \frac{\rho-1}{2} \mathbf{1}_{[0,1]}(y) u(t,y) \\
        u(t,0) = u(t,L) = 0.
    \end{cases}\end{equation}
By the Sturm--Liouville theory, the eigenvalues of the Sturm--Liouville problem
\begin{equation}
    \frac{1}{2}v''(x)+\frac{\rho-1}{2} v(x)\mathbf{1}_{x\leqslant 1}=\lambda v \quad \text{on} \quad  (0,L)\label{def:Tintro},
\end{equation}
with boundary conditions $v(0)=v(L)=0$, are simple and can be   {enumerated}
\[
    \lambda_1^L > \lambda_2^L > \cdots > \lambda_n^L >\cdots\to -\infty.
\]
  {It is also known that} each $\lambda_i^L$ is increasing with respect to $L$. If $v_1,v_2,\ldots$ denote the corresponding eigenfunctions of unit $\mathrm{L}^2$-norm, then $(v_i)$ is an orthonormal sequence,  complete in $\mathrm{L}^2([0,L])$, so that the function ${ \tilde p_t}$ is given by
\[
    { \tilde p_t(x,y)} = \sum_{n=1}^\infty e^{\lambda_n^L t} v_n(x)v_n(y),
\]
and hence,
\begin{equation}
    p_t(x,y) = \sum_{n=1}^\infty e^{\mu(x-y) + (\lambda_n^L + \frac{1}{2}-\mu^2/2) t} v_n(x)v_n(y)\label{def:pt}.
\end{equation}
We say that $p_t$ is the density of the BBM with branching rate $r(x)$ and drift $-\mu$, killed at $0$ and $L$, in the sense that, starting with a single particle at $x$, the expected number of particles in a Borel subset $B$ at time $t$ is given by $\int_B p_t(x,y)dy$.
Based on these observations, $\mu$ is chosen in such a way that the mass loss in $p_t$ stays controlled. Yet we will prove in Section \ref{sec:spec} that, for $\rho>\rho_1$, a positive and isolated generalised eigenvalue $\linf$ emerges as $L\to \infty$. Therefore we will choose $\mu$ such that
\begin{equation*}
    \mu=\sqrt{1+2\linf},
\end{equation*}
as stated in \eqref{def:mu}. We will prove in Section \ref{sec:spec} that this definition is equivalent to \eqref{def:murho}, see Lemma \ref{eigenvlocasymp}.
In the case where $\rho<\rho_1$, the sequence $(\lambda_i^L)$ converges to a non-positive continuous spectrum. In particular, $\linf=0$, so that $\mu=1$.
For $\rho>\rho_1$ and sufficiently large $t$, we show that
\begin{equation}\label{asymp:p}
    p_t(x,y) \approx e^{\mu(x-y) + (\lambda_1^L - \linf)t} v_1(x)v_1(y).
\end{equation}
Therefore, the time scale over which particles reach $L$ is of order $(\linf-\lambda_1^L)^{-1}.$ The spectral analysis of the system (B) provides the existence of a constant $C>0$ such that
\begin{equation*}
    \linf-\lambda_1^L\sim Ce^{-2\sqrt{2\linf} L}.
\end{equation*}
To simplify the notation we set
\begin{equation}
    \beta=\sqrt{2\linf}.\label{def:beta}
\end{equation}
As we expect the time scale of the CSBP to be given by $N^{\alpha-1}$ for some $\alpha\in(1,2)$, the asymptotic behaviour of $\lambda_1^L$ gives a first relation between $\alpha$, $N$ and $L$, that is
\begin{equation*}
    N^{\alpha-1}= e^{2\beta L}.
\end{equation*}
The eigenfunction associated to the principal eigenvalue $\lambda_1^L$  will play a crucial role in this analysis.  We  denote by $w_1$ the eigenfunction
\begin{equation*}
    w_1=\sinh\left(\sqrt{2\lambda_1^L}(L-1)\right)v_1.
\end{equation*}
This renormalisation will ensure that $w_1(L-1)$ remains of order $1$ as $L\to\infty$.
Then,  as in \cite{Berestycki2010}, we define  the process
\[
    Z_t = \sum_{v\in\mathcal N_t} e^{\mu (X_v(t)-L)} w_1(X_v(t))\mathbf{1}_{X_v(t)\in[0,L]}.
\]
As long as the particles stay in $[0,L]$, this process coincides with
$$Z_t' = \sum_{v\in\mathcal N_t^L} e^{\mu (X_v(t)-L)} w_1(X_v(t)).$$
The process $Z'_t$ is a supermartingale since, by  Lemma \ref{lem:many-to-one},
\begin{equation}
    \mathbb{E}_x\left[Z'_t\right]=e^{(\lambda_1^L-\linf)t}Z_0'.\label{eq:intr:fmZ}
\end{equation}
The process $Z_t$, and thus $Z'_t$, govern the long-time behaviour of the particle system. Indeed, for $t$ large enough, the expected number of particles in the system starting with a single particle at $x$ will be approximately  given  by
\begin{equation*}
    \mathbb{E}_x\left[N_t\right] \approx\int_0^L p_t(x,y)dy\approx e^{\mu x}v_1(x)e^{(\lambda_1^L-\linf)t}\int_0^Le^{-\mu y}v_1(y)dy.
\end{equation*}
This is a consequence of Lemma \ref{lem:many-to-one} and Equation \eqref{asymp:p}.
We will show that the second integral converges to a positive limit and that $v_1(x)\approx Ce^{-\beta L}w_1(x)$, so that
\begin{equation}
    \mathbb{E}_x\left[N_t\right] \approx Ce^{(\mu-\beta)L}Z_0', \label{eq:NZ}
\end{equation}
for $t\ll e^{2\beta L}$. Thus, we will first prove  Theorem \ref{semipushedfr} for $Z_t$ instead of $\bar{N}_t$ and then deduce the result for $\bar{N}_t$.

Moreover, we claim that the barrier $L$ has to be chosen so that
\begin{equation}
    \label{eq:NL*}N=e^{(\mu-\beta)L}.
\end{equation}
Indeed, $L$ is fixed in such a way that the particles that reach $L$ have a number of descendants of order $N$ after a short time, on the time scale $e^{2\beta L}$. Yet, if we consider the system starting with a single particle close to $L$, say at $x=L-1$, we get that $Z_0'$ is of order $1$. Thus, \eqref{eq:NL*} follows from Equation \eqref{eq:NZ}. In addition, we obtain that
\begin{equation}
    \alpha=\frac{\mu+\beta}{\mu-\beta}\label{eq:alpha},
\end{equation}
which is equivalent to the definition \eqref{defalpha:rho} (see \eqref{def:mu} and \eqref{def:beta}).

In light of Equations (\ref{eq:NZ}), (\ref{eq:NL*}) and \eqref{eq:alpha}, we claim that it is sufficient to prove that
\begin{equation}
    \label{eq:cvZ}
    Z_{e^{2\beta L}t}\Rightarrow \Xi(t), \quad \text{as} \quad \text L\to\infty,
\end{equation}
where $\Xi$ is an $\alpha$-stable CSBP,
starting with a suitable initial configuration.

As explained in \cite{Birzu:2020up}, the difference between the genealogical structures of the population for $\rho<\rho_1$, $\rho \in (\rho_1,\rho_2)$ and $\rho>\rho_2$ is explained by the \textit{fluctuations in the total number of descendants left by the early founders}. In our particle system, this number of descendants is related to the number of offspring of a particle hitting the barrier $L$. We  prove that the number $Z_y$ of these descendants reaching $L-y$ (for the first time) is such that
\begin{equation*}
    e^{-(\mu-\beta)y}Z_y\Rightarrow W\quad \text{as} \quad y\to \infty,
\end{equation*}
for some random variable $W$ satisfying
\begin{equation*}
    \mathbb{P}(W>x)\sim \frac{C}{x^\alpha}\quad \text{as} \quad x\to \infty.
\end{equation*}
The fact that $\alpha$ depends on $\rho$ only through the drift $\mu$ can be explained by this barrier at $L-y$:  it can be chosen in such a way that the particles are stopped before they reach $1$ so that they behave as  a BBM with drift $-\mu$ and constant branching rate $\frac{1}{2}$.

    {  The fluctuations of $Z'$ will be bounded using a second moment estimate.} We will make use of the many-to-two lemma.
\begin{lemma}[Many-to-two lemma, see \cite{Ikeda:1969tj}, Theorem 4.15]
    \label{lem:many-to-two}
    Let $f$ and  $p_t(x,y)$ be as in Lemma~\ref{lem:many-to-one}. Then
    \begin{align*}
         & \mathbb{E}_x\left[\left(\sum_{v\in \mathcal N^L_t} f(X_v(t))\right)^2\right] \\
         & = \int_0^L p_t(x,y) f(y)^2\,dy
        + \int_0^t \int_0^L p_s(x,y) 2r(y) \mathbb{E}_y\left[\sum_{v\in \mathcal N^L_{t-s}} f(X_v(t-s))\right]^2\,dy\,ds.
    \end{align*}
\end{lemma}

To prove \eqref{eq:cvZ}, we will follow the strategy developed in \cite{Maillard:2020aa} in the case $\rho=1$: we will show that the Laplace transform of $Z$ converges to that of $\Xi $ as $L\to\infty$. Once  \eqref{eq:cvZ} is proved, one can deduce the same convergence result for $\bar{N}_t$. It will be sufficient to prove that over a short time, on the time scale of the CSBP, $\bar N$ and $Z$ do not vary much and that $\bar{N}$ is well-approximated by $Z$ (see \eqref{eq:NZ}) as in \cite[Section 4.6]{Berestycki2010} in the case $\rho=1$.

  {We end this section with a reformulation of \eqref{hpushed} and \eqref{hwp} in terms of $\linf$, $\alpha$, $\mu$ and $\beta$ (the first assertion  will the object of Section \ref{sec:spec}):}
\begin{equation*}
    \eqref{hpushed} \; \Leftrightarrow \; \mu>1 \; \Leftrightarrow \; \linf>0 \; \Leftrightarrow \; \alpha >1,
\end{equation*}
and,
\begin{equation}\label{hwp2}
    \eqref{hwp} \; \Leftrightarrow \; \mu\in\left(1,\frac{3}{4}\sqrt{2}\right) \; \Leftrightarrow \; \linf\in\left(0,\frac{1}{16}\right) \; \Leftrightarrow \; \alpha\in(1,2).
\end{equation}
Finally, note that \eqref{eq:alpha} implies
\begin{equation}
    \alpha<2 \; \Leftrightarrow \; \mu>3\beta. \label{rem:alpha}
\end{equation}

\subsection{Perturbation of the Laplacian on the half-line}\label{pt:Laplacian}

A crucial role in the analysis will be played by the family of differential operators $T_\rho$, $\rho\in\mathbb{R}$, defined by
\[
    T_\rho u(x) = \begin{cases}
        \frac{1}{2} u''(x) + \frac{\rho}{2} \mathbf{1}_{[0,1]}(x)u(x), & x\in(0,1)\cup(1,\infty), \\
        \lim_{z\to 1} T_\rho u(z),                                     & x=1.
    \end{cases}
\]
with domain
\[
    \mathcal D_{T_\rho} = \{u\in C^1((0,\infty))\cap C^2((0,1)\cup(1,\infty)): \lim_{x\to0} u(x) = 0,\ \lim_{x\to 1} T_{\rho}u(x)\text{ exists}\}.
\]
The operator $T_\rho$ is a perturbation of the Laplacian on the positive half-line by a function of compact support.

In this section, we recall a few well-known facts about such operators, based on Section 4.6 in \cite{Pinsky1995}. These results are only given for continuous perturbations, but one can extend them to our particular perturbation by approximating the step function on $[0,1]$ by continuous functions. Actually, these facts will not be used in the following proofs, yet they provide a better understanding of the three regimes in the particle system.

Define the \emph{generalised principal eigenvalue} of the operator $T_\rho$ by
\[
    \lambda_c(\rho) = \inf\{\lambda\in\mathbb{R}: \exists u\in \mathcal D_{T_\rho}: u>0\text{ on }(0,\infty),\ T_\rho u = \lambda u\}.
\]
Theorem 4.4.3 in \cite{Pinsky1995} implies that $\lambda_c$ is a convex function of $\rho$ and Lipschitz-continuous with Lipschitz constant 1/2.
Let $(B_t)$ be a standard Brownian motion starting at $x>0$ and let $\tau=\inf\{t\in(0,\infty):B_t\notin(0,\infty)\}$. The Green function $G_\rho$ of the operator $T_\rho$ is the unique function such that, for all bounded measurable functions $f$, we have
\begin{equation*}
    \mathbb{E}\left[\int_0^\tau\exp\left(\int_0^t\frac{\rho}{2}\mathbf{1}_{[0,1]}(B_s)ds\;\right)f(B_t)dt\right]=\int_0^\infty G_\rho(x,y) f(y)dy.
\end{equation*}
Similarly, one can define the Green function of the operator $T_\rho-\lambda$, denoted by $G_\rho^\lambda$, such that
\begin{equation*}
    \mathbb{E}\left[\int_0^\tau\exp\left(\int_0^t\left(\frac{\rho}{2}\mathbf{1}_{[0,1]}(B_s)-\lambda\right)\, ds\right)f(B_t)dt\right]=\int_0^\infty G^\lambda_\rho(x,y) f(y)dy.
\end{equation*}
Recall from \cite[Section 4.3]{Pinsky1995}  that an operator is called
\begin{itemize}
    \item \emph{subcritical}, if its Green function is finite (and hence positive harmonic functions, i.e.~eigenfunctions of eigenvalue 0, exist),
    \item \emph{critical}, if its Green function is infinite, but positive harmonic functions exist,
    \item \emph{supercritical}, if no positive harmonic function exists.
\end{itemize}
It is well known that the Laplacian on the positive half-line, i.e.~the unperturbed operator $T_0$, is subcritical in the sense of \cite{Pinsky1995}: its Green function is finite and is given by $G_0(x,y) = 2x\wedge y$, $x,y>0$. Furthermore, its generalised principal eigenvalue is $\lambda_c(0) = 0$. It then follows from Theorem~4.6.4 in \cite{Pinsky1995} that there exists $\rho_c > 0$, such that $\lambda_c(\rho) = 0$ for all $\rho \le \rho_c$ and $\lambda_c(\rho) > 0$ for all $\rho > \rho_c$. Moreover, $T_\rho$ is subcritical for $\rho < \rho_c$, critical for $\rho = \rho_c$ and supercritical for $\rho > \rho_c$. In fact, Theorem~4.7.2 in \cite{Pinsky1995} implies that $T_\rho-\lambda_c(\rho)$ is critical for $\rho >\rho_c$.

These properties can be verified by elementary calculations, which also yield exact expressions for $\rho_c$ and $\lambda_c(\rho)$. We summarise these calculations in the following proposition:

\begin{proposition}\label{th:Trho}
    Define $\rho_c = \pi^2/4$. Define the function
    \[
          {h}(x) = \sinc(\sqrt{x})^{-2},\ x\in [\rho_c,\pi^2),
    \]
    where $\sinc(z) = \sin(z)/z$.
    Then $  {h}$ is an increasing and strictly convex function on $[\rho_c,\pi^2)$ with $  {h}(\rho_c) = \rho_c$, $  {h}'(\rho_c) = 1$ and $  {h}(x) \to \infty$ as $x\to \pi^2$. Denote by $  {h}^{-1}$ its inverse, defined on $[\rho_c,\infty)$. Then
    \[
        \lambda_c(\rho) = \begin{cases}
            0                                                  & \rho \le \rho_c \\
            \frac 1 2 (\rho -   {h}^{-1}(\rho)) & \rho > \rho_c.
        \end{cases}
    \]
\end{proposition}

The proof of this proposition can be found in Appendix \ref{proof:15}. One could go on calculating the positive eigenfunctions of the operator $T_\rho$ for all $\rho$. One would see that, for every $\rho\in \mathbb{R}$ and every $\lambda \ge \lambda_c(\rho)$, there exists a unique (up to a multiplicative constant) positive eigenfunction of eigenvalue $\lambda$. For $\lambda = \lambda_c(\rho)$, this function is affine on $[1,\infty)$ with positive slope for $\rho < \rho_c$, and exponentially decreasing, with exponent $-\sqrt{2\lambda_c(\rho)}$, on $[1,\infty)$, for $\rho > \rho_c$. In fact, in the latter case, an eigenfunction is
\[
    u(x) = \begin{cases}
        \sin(\sqrt{  {h}^{-1}(\rho)} x)                                  & x\in[0,1]       \\
        \sin(\sqrt{  {h} ^{-1}(\rho)}) e^{-\sqrt{2\lambda_c(\rho)}(x-1)} & x\in[1,\infty).
    \end{cases}
\] This function will play a crucial role in the system. Indeed, it corresponds to a harmonic function of the critical operator $T_\rho-\lambda_c(\rho)$. According to Theorem 8.6 in \cite{Pinsky1995}, this function is the unique (up to positive multiples) \textit{invariant function} for the transition measure associated to $T_\rho-\lambda_c(\rho)$. Roughly speaking, this means that $u$ is a stable configuration in the particle system. On the other hand, for $\lambda > \lambda_c(\rho)$, the function grows exponentially on $[1,\infty)$ with exponent $\sqrt{2\lambda}$.

Let us now go back to the differential operator $\frac{1}{2}\partial_{xx}+\mu \partial_x+r(x)$ from Equation \eqref{PDE:A}. Thanks to Equation \eqref{eq:ptqt}, the Green function $G$ of this operator  can be expressed thanks to $G^\lambda_{\rho-1}$,
\begin{equation*}
    G(x,y)=e^{\mu(x-y)}G_{\rho-1}^{\lambda}(x,y), \quad\text{for} \; \lambda=\frac{\mu^2-1}{2} .
\end{equation*}
The value of $\mu$ will be then chosen in such a way that the differential operator associated to \eqref{PDE:A} has a harmonic function. Then, for $\rho-1<\rho_c$ it is sufficient to choose $\mu=1$ since $T_{\rho-1}$ is subcritical. For $\rho-1>\rho_c$, we know that $T_{\rho-1}-\lambda_c(\rho-1)$ is critical. Therefore the corresponding Green function is infinite but the operator has harmonic functions. Hence we will choose the drift $\mu$ such that
\begin{equation*}
    \mu(\rho)=\sqrt{1+2\lambda_c(\rho-1)}.
\end{equation*}
Note that the limit $\linf$ of the maximal eigenvalues $\lambda_1^L$ and the generalised principal eigenvalue $\lambda_c(\rho-1)$ coincide. This is a consequence of Theorem 4.1 in \cite{Pinsky1995}.

\subsection{Related models}\label{relmod}
A rich theory has been developed in the case where $B=0$ in \eqref{eq:FKPPhall}, which corresponds to a special case of the pulled regime. First, the equation
\begin{equation}
    \label{eq:FKPPn}
    u_t=\frac{1}{2}u_{xx}+u(1-u)+\sqrt{\frac{u(1-u)}{N}}W(t,x)
\end{equation}
was studied in \cite{Brunet2006} to investigate the effect of demographic fluctuations on the FKPP equation. Indeed, if one removes the noise term in \eqref{eq:FKPPn}, one  obtains the FKPP equation, introduced by Fisher \cite{Fisher:1937wr} and independently by Kolmogorov, Petrovskii and Piskounov \cite{Kolmogorov:1937tc}, to describe the invasion of a stable phase ($u\approx1$) in an unstable phase ($u\approx0$). In this case, it is well-known \cite{Kolmogorov:1937tc} that $c_{\min}=c_0=\sqrt{2r_0}$ so that the invasion is pulled.

As explained in \cite{Panja2004}, the FKPP equation can be seen as the hydrodynamic limit of many particle systems. However, the finite nature of these physical or biological systems induces fluctuations, which can be modelled by adding multiplicative square root noise to the FKPP equation. Heuristically, this correction corresponds to the rescaled \textit{difference between the limiting PDE and the particle system in the style of a central limit theorem}   \cite{Mueller:1994wa}. The addition of this noise term in Equation (\ref{eq:FKPPn}) makes the shape and position of the front fluctuate.

In \cite{Brunet2006}, the authors explain how to infer the first order of the correction to the speed of the noisy fronts (compared to the deterministic fronts) thanks to a particle system. Since the fluctuations emerge at the leading edge of the front, they do not need to introduce a saturation rule in the particle system to deduce the correction to the velocity of the wave. Analysing the mechanisms driving the invasion, they conjecture that  the fluctuations appear over a time scale of order $\log(N)^3$. They deduce from this fact that the correction of the speed $c_0$ is of order $\log(N)^{-2}$. This statement was then rigorously proved in \cite{Mueller2010} for the SPDE \eqref{eq:FKPPn}. This correction, that is much greater than expected ($1/\sqrt{N}$), underscores the large fluctuations in the pulled regime.

In \cite{Brunet2006a,Brunet_2007}, the authors analyse a particle system with a fixed population size to
investigate the genealogy at the tip of the invasion front in the pulled regime. The particles evolve in discrete
time and, at each generation, independently give birth to exactly $k$ children, scattered around the parental
location. At the end of each generation, only the $N$ rightmost individuals survive. This set of particles forms
a cloud that does not diffuse and can be described by a front governed by \eqref{eq:FKPPn} \cite{Brunet_2007}.
In this framework, they conjecture \cite{Brunet2006a} that the genealogy of the particles in the cloud is
described by a Bolthausen--Sznitman coalescent. The fact that the correction to the speed of this system
is the same as the one for solutions of \eqref{eq:FKPPn} was rigorously proved in \cite{Berard2010},
in the case $k=2$. This result was then extended to random offspring distributions in \cite{Mallein2015}.

The conjecture on the genealogy stated in \cite{Brunet2006a,Brunet_2007} was proved under slightly different assumptions in \cite{Berestycki2010}.
Berestycki, Berestycki and Schweinsberg \cite{Berestycki2010}  considered a branching Brownian motion with absorption for a suitable choice of drift $-\mu$.
It is the branching property of the BBM that makes this system analytically tractable.   The drift is then chosen to be supercritical, matching the correction to the speed of the noisy front conjectured in \cite{Brunet2006}: for each integer $N$, they consider a dyadic BBM, with drift $-\mu_N$, with
\begin{equation}
    \mu_N=\sqrt{1-\frac{\pi^2}{(\log(N)+3\log\log(N))^2}},\label{driftBBS}
\end{equation}
starting, for instance, with $N\log(N)^3$ particles at $x=1$. With the notation of Theorem \ref{semipushedfr}, they obtain that, as $N$ goes to $\infty$, the processes $\left(\bar{N}_{\log(N)^3t},\;t\geqslant 0\right)$ converge in law to Neveu's continuous-state branching process. Using the results from \cite{bertoin2000bolthausen}, they deduce from this fact that the genealogy of the system is given by the Bolthausen--Sznitman coalescent. It was then shown in \cite{Maillard:2016uw} that many ideas developed in \cite{Berestycki2010} also hold in the case of a BBM with constant population size $N$.

In this work, we are interested in the genealogy of the particles at the tip of the front for a more general form of the reaction term in the limiting PDE. While the study in \cite{Brunet2006,Brunet2006a} concerns FKPP fronts, that are classified as \textit{pulled}, we focus on reaction terms of the form (\ref{reactionterm}). In this case, the deterministic front in the limiting PDE can be either \textit{pulled} ($B\leqslant 2$) or \textit{pushed} ($B>2$).

In the semipushed regime,  an $\alpha$-stable CSBP emerges in the limit in the particle system. This suggests that the genealogy of the particles  is given by a Beta($2-\alpha,\alpha$)-coalescent \cite{Birkner2005}. While Beta($2-\alpha,\alpha$)-coalescents  are known to interpolate between Bolthausen--Sznitman and Kingman coalescents in population models \cite{schweinsberg2003coalescent}, simple  systems exhibiting such a continuous phase transition are not  so common in the literature. Another particle system showing a similar interpolation regime can be found in \cite{cortines2018genealogy}.
A transition between the Bolthausen--Sznitman and Kingman coalescents also appears in \cite{Brunet2012} but the genealogical structure emerging in the interpolation regime is given by a $\Lambda-$coalescent.

Stochastic models for population genetics have received quite a lot of attention recently.
In \cite{fitnesswave} and \cite{liu2021particle}, the authors considered an inhomogeneous BBM, in which the difference between the branching rate and the death rate is linear,  to model a population undergoing natural selection.  Powell \cite{powell2019invariance} studied a critical branching diffusion in a bounded domain (in $\mathbb{R}^d$) and proved that the genealogical tree of the particles converges to Aldous’ Continuum Random Tree. In discrete space, Etheridge and Penington \cite{etheridge2020genealogies} examined a structured Moran model to describe the genealogy of an advantageous allele in a diploid population under selection.

In this work, the system is nearly critical as in \cite{Berestycki2010} but, unlike \eqref{driftBBS}, the drift $\mu$ is chosen as a function of $\rho$ and does not depend on $N$. This difference is due to different behaviours of the spectrum of the differential operators associated with the BBM. This relation between the generalised eigenvalues and local extinction/exponential growth has already been discussed, see e.g.~\cite{Englander:2004aa,powell2019invariance}.

\subsection{Biological motivations: the Allee effects}\label{sec:allee}

In biology, spatial invasions are often described by the minimal front solutions of \eqref{limit:PDEh}.
In terms of population models, a front is pushed, for instance, in the presence of a sufficiently strong \emph{Allee effect}, meaning that the particles near the front have a competitive advantage over particles far away from the front. The strength of the Allee effect is scaled by the parameter $B$ in the reaction term \eqref{reactionterm}.

Allee effects are well-explained in \cite{Hallatschek2008}: ``The presence of conspecifics can be beneficial due to numerous factors, such as predator dilution, anti predator vigilance, reduction of inbreeding and many others. Then, the individuals in the very tip of the front do not count so much, because the rate of reproduction decreases when the number density becomes too small. Consequently, the front is pushed in the sense, that its time-evolution is determined by the behaviour of an ensemble of individuals in the boundary region''. In sharp contrast, pulled invasions are the ones for which the growth is maximal at low densities so that the individuals located at the leading edge pull the invasion. As explained in  \cite{stokes1976two}, the consequence of this fact is that ``the speed of the wave is determined by the fecundity of their pioneers'', or, in other words, it only depends on $f'(0)$ (see  \eqref{reactionterm}). Pushed waves are faster and \textit{pushed}, or driven, by the nonlinear dynamic of the bulk (see Section \ref{sec:relatedmodel}). Consequently, the speed of the waves depends on the functional form of the reaction term $f$.

This shift in the invasion speed is not the only consequence of Allee effects. Indeed, one can investigate the genealogies of a particle system governed by Equation \eqref{eq:FKPPhall}. One expects them to evolve over larger time-scales for pushed fronts than for pulled fronts. In biological terms, this translates into a larger genetic diversity \cite{Hallatschek2008}. For pulled fronts, the time-scale is logarithmic in $N$ and the genealogy is described by the Bolthausen--Sznitman coalescent \cite{Brunet2008}. If the Allee effect is sufficiently strong, it is natural to assume that the genealogy evolves over the timescale $N$ and is described by Kingman's coalescent \cite{Birzu2018}. This was proved in the case of strong Allee effects in the context of population genetics \cite{etheridge2020genealogies}. Strong Allee effects are often modelled by bistable reaction diffusion equations, which can not be considered with reaction terms of the form \eqref{reactionterm} (heuristically, it corresponds to $B\to \infty$). See  \cite[Chapter 1]{these} for further details on the classification of Allee effects. The simulations in \cite{Birzu2018} and the analysis conducted here describe the intermediate regime between these two extremes: the genealogy is observed on a time scale $N^{\alpha-1}$ for some $\alpha\in(1,2)$ and its structure is given by a Beta-coalescent.

According to \cite{Birzu2018}, pulled and pushed fronts can also be distinguished by the spatial position of the ancestors of the particles. Taking a particle at random and looking at its ancestor at a time far in the past, this ancestor will sit at the leading edge of the front (i.e.~far to the right of the front) in pulled fronts, whereas it will be at the middle of the front (i.e.~in the bulk) in pushed fronts, where most particles lie \cite[Fig.~2]{Birzu2018}. One can consider the trajectory described by the ancestors of this particle as the path of an \emph{immortal particle}, and thus conjecture the following two distinct behaviours: in pulled fronts, the path of an immortal particle typically spends most of its time far away from the bulk, whereas in pushed fronts, it spends most of its time in the bulk, in the vicinity of the other particles. Indeed, in the model studied in \cite{Berestycki2010}, which can be seen as a simplification of the noisy FKPP equation, the prime example of a pulled front, the path of the immortal particle resembles in the co-moving frame a Brownian motion constrained to stay in an interval of size of order $\log N$, and is thus typically a distance $\log N$ away from the bulk. On the other hand, for pushed fronts, one should expect that the path of an immortal particle is described in the co-moving frame by a positive recurrent Markov process independent of the population size.

Another distinction arises when one considers the events that drive the evolutionary dynamics, i.e. those that cause mergers in the ancestral lines of individuals randomly sampled from the population. The authors of \cite{Birzu2018} conjecture that the distinction does not take place between pulled and pushed, but between pulled and semipushed on the one side and fully pushed on the other \cite[SI, p36]{Birzu2018}. In fully pushed fronts, the population can be approximated by a neutral population, with \textit{all the organisms at the front}. In contrast, the particles located at the tip of the front  drive the evolutionary dynamics in semipushed and pulled waves.
This is consistent with the genealogical structures introduced above. Indeed, in pulled and semipushed fronts we expect the genealogies to be described by coalescents with multiple mergers. In these coalescents, single individuals replace a fraction of the population during coalescence events. It is reasonable to think that, for this to happen, a particle has to move far away from the front in order to have time to produce a large number of descendants before being incorporated in the front again. On the other hand, in fully pushed fronts, we expect the genealogy to be described by Kingman's coalescent, indicating that the population behaves like a neutral population where particles are indistinguishable. Thus, typical particles, i.e. those which are in the bulk, should drive the evolutionary dynamics. Of course, it is still possible for particles to move far away from the front and replace a fraction of the population. But since Kingman's coalescent only consists of binary mergers, these events are not visible in the limit and thus have to happen on a longer time-scale than the time-scale $N$ at which the genealogy evolves. The characteristics of the three types of fronts are summarised in Table~\ref{tab:pulled_pushed}.

\begin{table}[ht]
    \begin{center}
        \begin{tabular}{|p{5cm}|c|c|c|}
            \hline
                                                                                  & \textbf{pulled}                                   & \multicolumn{2}{c|}{\textbf{pushed}}                                            \\
                                                                                  &                                                   & semipushed                                                & fully pushed        \\
            \hline
            \textit{cooperativity $B$}                                            & $B\in(0,2]$                                       & $B\in(2,B_c)$                                             & $B\in(B_c,+\infty)$ \\
            \hline
            \textit{Allee effect}                                                 & \multicolumn{3}{c|}{\textbf{weak Allee effect}}                                                                                     \\
            \textit{}                                                             & \multicolumn{1}{l}{$\leftarrow$ no Allee effect}  & \multicolumn{2}{r|}{  strong Allee effect $\rightarrow$}                        \\
            \textit{}                                                             & \multicolumn{1}{c}{$(B=0)$}                       & \multicolumn{2}{c|}{  $\qquad \qquad \quad (B\to\infty)$}                       \\
            \hline
            \textit{speed of front compared to linearised equation}               & same                                              & \multicolumn{2}{c|}{faster}                                                     \\
            \hline
            \textit{path of an immortal particle}                                 & far to right of front                             & \multicolumn{2}{c|}{close to front}                                             \\
            \hline
            \textit{time-scale of genealogy}                                      & polylog($N$)                                      & $N^{\alpha-1}$, $\alpha\in(1,2)$                          & $N$                 \\
            \hline
            \textit{evolutionary dynamics driven by particles at positions\ldots} & \multicolumn{2}{c|}{\ldots far to right of front} & \ldots close to front                                                           \\
            \hline
        \end{tabular}
    \end{center}
    \caption{\label{tab:pulled_pushed} Summary of the characteristics of pulled, semipushed and fully pushed fronts. }
\end{table}

\subsection{Structure of the article}

The proof of the result follows the steps detailed in Section \ref{sec:skproof}.
In Section \ref{sec:density}, we examine the density of particles $p_t$: we fully characterise the spectrum of \eqref{def:Tintro} and show that the particles stabilise at a stationary configuration after a long time.
In Section \ref{sec:moment:est}, we bound the first and second moments of several quantities (including $Z_t'$) which rule the long-time behaviour of the system.
In Section \ref{sec:R}, we control the number of particles that hit the level $L$. In Section, \ref{sec:W}, we estimate the number of descendants of these particles after a large time of order~1. In Section \ref{sec:smallts} and \ref{sec:cvcsbp}, we put all these estimates together to prove the convergence result.

\subsection{Some notation}
We recall the definition of several quantities  depending  on the parameter $\rho$ of the model, as well as their dependences. In the remainder of the paper, we denote by $\lambda_1$ the maximal eigenvalue of the Sturm-Liouville problem \eqref{def:Tintro} with boundary condition $v(0)=v(L)=0$. Hence, $\lambda_1$ depends on $L$ and we prove (this is the object of Section \ref{sec:spec}) that, for $\rho>\rho_1$, $\lambda_1$ increases with $L$ and converges to a positive limit $\linf$ as $L$ goes to $\infty$.

In this case, we write $\alpha, \beta, \gamma$ and $\mu$ to refer to the following quantities:
\begin{equation}
    \label{defnot}
    \mu=\sqrt{1+2\linf}, \quad \beta=\sqrt{2\linf}, \quad  \alpha=\frac{\mu+\sqrt{\mu^2-1}}{\mu-\sqrt{\mu^2-1}}=\frac{\mu+\beta}{\mu-\beta}, \; \text{and} \; \; \gamma=\sqrt{\rho-1-2\linf},
\end{equation}
to emphasise that they do not depend on $L$, but only on $\rho$.

Throughout the paper, $C$ denotes a positive constant whose value may change from line to line. Unless otherwise specified, these constants only depend on $\rho$. Numbered constants keep the same value throughout the text.

\section[Branching Brownian motion in an interval: the density of particles]{BBM in an interval: the density of particles}\label{sec:density}

The goal of this section is to estimate the density of particles $p_t$ in the BBM with absorption at $0$ and at an additional barrier $L>0$. Recall from Lemma \ref{lem:many-to-one} and Equation \eqref{eq:ptqt} that the density  of particles in the dyadic BBM with branching rate $r(x)$ and drift  $-\mu$, killed upon exiting the interval $(0,L)$ can be calculated using the fundamental solution ${ \tilde p_t}$ of the self-adjoint partial differential equation \eqref{PDE:B}. According to Sturm--Liouville theory, this fundamental solution can be expressed in term of the eigenvalues and the eigenfunctions of  \eqref{def:Tintro}. Section \ref{sec:spec} is  devoted to this spectral decomposition.

We then prove that the particles stabilise at a \textit{stationary configuration} after a time of order $L$ in Section
\ref{sec:heatkernel}. In Section \ref{sec:green}, we give a bound on the Green function associated to \eqref{PDE:B}.

\subsection{Spectral analysis}\label{sec:spec}

Let $L>1, \, \rho\in (1,\infty)$ and consider the Sturm-Liouville problem (SLP) consisting of the equation
\begin{equation}
    \frac{1}{2}v''(x)+\frac{\rho-1}{2} v(x)\mathbf{1}_{x\leqslant 1}=\lambda v \quad \text{on} \quad  (0,L),\tag{E}
    \label{eq:defT}
\end{equation}
together with the boundary conditions
\begin{equation}\tag{BC}\label{bc}
    v(0)=v(L)=0.
\end{equation}
Let us first recall well-know facts about Sturm--Liouville theory following \cite[Section 4.6]{Zettl:2010aa}:
\begin{itemize}
    \item[(i)] A solution of \eqref{eq:defT} is defined as a function $v:[0,L]\to \mathbb{R}$ such that $v$ and $v'$ are absolutely continuous on $[0,L]$ and satisfies \eqref{eq:defT} a.e.~on $(0,L)$. In particular, any solution $v$ is continuously differentiable on $[0,L]$ and since $x\mapsto\mathbf{1}_{[0,1]}(x)$ is continuous on $(0,1)$ and $(1,L)$, the solutions are also twice differentiable on $(0,1)\cup(1,L)$ and (\ref{eq:defT}) holds for all $x\in(0,1)\cup(1,L)$.
    \item[(ii)] A complex number $\lambda$ is an eigenvalue of the Sturm--Liouville problem \eqref{eq:defT} with boundary conditions \eqref{bc}  if Equation \eqref{eq:defT} has a solution $v$ which is not identically zero on $[0,L]$ and that satisfies \eqref{bc}. This set of eigenvalues will be referred to as the spectrum.
    \item[(iii)] The spectrum of the SLP \eqref{eq:defT} with boundary conditions \eqref{bc} is infinite, countable and it has no finite accumulation point. Besides, it is upper bounded and all the eigenvalues are simple and real so that they can be numbered
          \begin{equation*}
              \lambda_1>\lambda_2>...> \lambda_n>...
          \end{equation*}
          where
          \begin{equation*}
              \lambda_n\rightarrow -\infty \quad \textnormal{ as } \quad  n\rightarrow+\infty.
          \end{equation*}
    \item[(iv)] As a consequence, the eigenvector $v_i$ associated to $\lambda_i$ is unique up to multiplicative constants. Furthermore, the sequence of eigenfunctions can be normalised to be an orthonormal sequence of $\mathrm{L}^2([0,L])$. This orthonormal sequence is complete in $\mathrm{L}^2([0,L])$ so that the fundamental solution of PDE \eqref{PDE:B} can be written as (see e.g.~\cite[p.188]{Lawler:2018vn})
          \begin{equation}
              { \tilde p_t(x,y)}=\sum_{k=1}^\infty e^{\lambda_k t}\frac{v_k(x)v_k(y)}{\|v_k\|^2}.\label{def:qt1}
          \end{equation}
    \item[(v)]{  The eigenvector $v_1$ does not change sign on $(0,L)$. For $k\geq 2$, the eigenvector $v_k$ has exactly $k-1$ zeros in $(0,L)$.}
    \item[(vi)] For fixed $i\in\mathbb{N}$, the eigenfunction $\lambda_i$ is an increasing function of $L$ (see \cite[Theorem 4.4.4]{Zettl:2010aa}).
\end{itemize}

In Lemma \ref{eigenvloc} and Lemma \ref{ev:sc}, we give a characterisation of the eigenvalues $\lambda_i$ and of the corresponding eigenvectors $v_i$ for large $L$.  The remainder of the subsection (Lemma \ref{eigenvlocasymp}  to Lemma \ref{lem:est:evk})  is devoted to the study of the asymptotic behaviour of the $\lambda_i$ and the $v_i$ as $L$ tends to $\infty$.

We now introduce some notation that will be used throughout this section.   {Let $Q^+:=\{z\in\mathbb{C}:z=\rho e^{i\theta}, \ \rho\geq 0,\ \theta\in[0,\frac{\pi}{2}]\}$. Let  $\tilde s:Q^+\to \tilde s (Q^+),\ z\mapsto z^2$. Note that $\tilde s$ is one-to-one and that $\tilde s (Q^+)$ is the upper half-plane. We denote by $\sqrt{\cdot} \ : \ \tilde s (Q^+)\to Q^+$ its inverse function.
For all $(x,\lambda)\in [0,+\infty)\times \mathbb{R}$, define
\begin{equation}
    \mathcal{S}(x,\lambda)=\frac{\sinh(\sqrt{\lambda}x)}{\sqrt{\lambda}}=\begin{cases}
        \frac{\sinh\left(\sqrt{\lambda}x\right)}{\sqrt{\lambda}}  & (x,\lambda)\in [0,+\infty)\times(0,+\infty) \\
        \frac{\sin\left(\sqrt{-\lambda}x\right)}{\sqrt{-\lambda}} & (x,\lambda)\in [0,+\infty)\times(-\infty,0) \\
        x                                                         & (x,\lambda)\in[0,+\infty)\times \{0\}
    \end{cases}.
\end{equation}
Similarly, let $\mathcal{C}(x,\lambda)=\cosh\left(\sqrt{\lambda}x\right)/\sqrt{\lambda}$ and $\mathcal{T}(x,\lambda)=\frac{\mathcal{S}(x,\lambda)}{\mathcal{C}(x,\lambda)}$  for $(x,\lambda)\in [0,+\infty)\times \mathbb{R}$.}

\begin{lemma} \label{eigenvloc} Assume $\rho\notin\left\{1+\left(n-\frac{1}{2}\right)^2\pi^2, \; n\in\mathbb{N}\right\}$.
    There exists $L_0 = L_0(\rho)$, such that the following holds for all $L\ge L_0$: Let $K\in\mathbb{N}$ be the largest positive integer such that
    \begin{equation}
        \rho-1>\left(K-\frac{1}{2}\right)^2\pi^2,
        \label{ik}
    \end{equation} and $K=0$ otherwise.
    Then, for all $1\leqslant k \leqslant K$, $\lambda_k$ is the unique solution of
    \begin{equation}
        \mathcal{T}\left(1,2\lambda+1-\rho\right)=\mathcal{T}\left(L-1,2\lambda\right),
        \tag{$\square$}
        \label{square}
    \end{equation}
    such that
    \begin{equation}
        (\rho-1-k^2\pi^2\vee 0)<2\lambda_{k}< \rho - 1-\left(k-\frac{1}{2}\right)^2\pi^2.
        \label{enc}
    \end{equation}
    Furthermore, $\lambda_k<0$ for all $k > K$. More precisely, set for all $i\ge0$:
    \begin{align}
        A_i & = \frac{1}{2}\left(\left(K+\frac{1}{2}+i\right)^2\pi^2+1-\rho\right), \label{def:A}         \\
        N_i & = \left\lfloor\frac{(L-1)}{\pi}\sqrt{2 A_i} \; +\frac{1}{2}\right\rfloor + i, \label{def:N}
    \end{align}
    and $A_{-1} = N_{-1} = 0$. Also, set $a_0=0$ and
    \begin{equation}
        a_j=\frac{\left(j-\frac{1}{2}\right)^2}{2(L-1)^2}\pi^2,\quad j\ge 1.
        \label{def:aj}
    \end{equation}
    Then, for every $i\ge0$ and every $j\in \mathbb{N}$ such that $N_{i-1} < j \le N_{i}$,
    $\lambda_{K+j}$ is the unique solution of \eqref{square} in the interval
    \begin{equation}
        (-A_i,-A_{i-1}) \cap (-a_{j-i+1},-a_{j-i}).
        \label{encK}
    \end{equation}
    Finally, for all $k\in \mathbb{N}$, the eigenvector $v_k$ associated with $\lambda_k$ is unique up to multiplicative constants and is given by
    \begin{equation}
        v_k(x)=
        \begin{cases}
            \mathcal{S}(x,2\lambda_k-\rho-1)/\mathcal{S}(1,2\lambda_k-\rho-1) & x\in [0,1], \\
            \mathcal{S}(L-x,2\lambda_k)/\mathcal{S}(L-1,2\lambda_k)           & x\in[1,L].  \\
        \end{cases}
        \label{vecp1}
    \end{equation}
\end{lemma}

\begin{proof}
    Let $\lambda$ be an eigenvalue of the SLP consisting of \eqref{eq:defT} with boundary condition \eqref{bc} and consider $v$ an eigenvector associated to $\lambda$. As mentioned above, the eigenvalue $\lambda$ is real and $v$ is unique up to multiplicative constants. In addition, the function $v$ is twice differentiable on $(0,1)\cup(1,L)$ and solves the following system
    \begin{equation}
        \begin{cases}
            v''(x)=(2\lambda+1-\rho) v(x) & x \in (0,1) , \\
            v''(x)=2\lambda v(x)          & x\in(1,L),    \\
            v(0)=v(L)=0.
        \end{cases}
        \label{syst}
        \tag{$\mathcal{C}_\lambda$}
    \end{equation}
    First, let us prove by contradiction that the spectrum is bounded above by $\frac{\rho-1}{2}$. Suppose that \eqref{syst} has a solution for some   $\lambda>(\rho-1)/2.$ Then, $\lambda_1>\frac{\rho-1}{2}$ and the function $v_1$ can be written as \begin{equation} 
        v_1(x)=\begin{cases}
            A\sinh\left(\sqrt{2\lambda_1+1-\rho}\,x\right) & x\in (0,1), \\
            B\sinh\left(\sqrt{2\lambda_1}\,(L-x)\right)    & x\in(1,L),
        \end{cases}
        \label{evlpo}
    \end{equation}{}for some $(A,B)\neq (0,0)$. Recalling that $v_1$ is positive on $(0,L)$, we see that both $A$ and $B$ are positive. Moreover, the derivative  $v_1'$ is continuous at $1$ so that $A$ and $B$ satisfy
    \begin{equation*} 
        A\sqrt{2\lambda_1+1-\rho}\,\cosh(\sqrt{2\lambda_1}(L-1))=-B\sqrt{2\lambda_1}\, \cosh(\sqrt{2\lambda_1+1-\rho}).
    \end{equation*}
    This implies that $(A,B)=(0,0)$, which contradicts the fact that $\lambda_1$ is an eigenvalue. Similarly, one can prove that $\lambda_1\neq\frac{\rho-1}{2}$.

    Let us now characterise the positive part of the spectrum. Let $v$ be a solution of  (\ref{syst}) for some $0<\lambda<\frac{\rho-1}{2}$. Then there exists $(A,B)\neq(0,0)$ such that
    \begin{equation*} 
        v(x)=\begin{cases}
            A\sin\left(\sqrt{\rho-1-2\lambda}\,x\right) & x\in (0,1), \\
            B\sinh\left(\sqrt{2\lambda}\,(L-x)\right)   & x\in(1,L).
        \end{cases}
        \label{ev}
    \end{equation*}
    { 
    In fact, we have $A\neq 0$ and $B\neq 0$: point (v) implies that $v$ cannot be constant equal to $0$ on $(0,1)$ nor on $(1,L)$.
    Since $v$ is continuous and differentiable at $1$, we see that $A,B$ and $\lambda$ solve the system}
    \begin{equation}\label{sys:lambda}
        \begin{cases}
            A\sin(\sqrt{\rho-1-2\lambda}) =B\sinh(\sqrt{2\lambda}(L-1)),                                      \\
            A\sqrt{\rho-1-2\lambda}\cos(\sqrt{\rho-1-2\lambda})=-B\sqrt{2\lambda}\cosh(\sqrt{2\lambda}(L-1)). \\
        \end{cases}
    \end{equation}
    In particular, this implies  that $\sqrt {\rho-1-2\lambda}\notin\{\left(k-\frac{1}{2}\right)\pi, \; k\in\mathbb{N}\}$.  In addition, we get that $\lambda$ is solution to
    \begin{equation}
        -\frac{\tan(\sqrt{\rho-1-2\lambda})}{\sqrt{\rho-1-2\lambda}}=\frac{\tanh(\sqrt{2\lambda}(L-1))}{\sqrt{2\lambda}}.
        \label{eq1}
    \end{equation}
    We now prove that Equation (\ref{eq1}) has exactly $K$ solutions in $\left(0,\frac{\rho-1}{2}\right)$ for $L$ large enough. Let
    \begin{equation*}\label{deffg}
        \begin{array}{l rcl}
            {  f_1} : & (0,\frac{\pi}{2})\cup\left(\cup_{k\in\mathbb{N}}\left(\left(k-\frac{1}{2}\right)\pi,\left(k+\frac{1}{2}\right)\pi\right)\right) & \rightarrow & \mathbb{R}        \\
                                  & x                                                                                                                               & \mapsto     & \frac{\tan(x)}{x}\end{array}
        \quad \text{and}  \quad \begin{array}{l rcl}{  f_2}: & (0,\infty) & \rightarrow & (0,\infty)         \\
                                         & x          & \mapsto     & \frac{\tanh(x)}{x}\end{array}.
    \end{equation*}
    For $x\in(0,\frac{\pi}{2})\cup\left(\cup_{k\in\mathbb{N}}\left(\left(k-\frac{1}{2}\right)\pi,\left(k+\frac{1}{2}\right)\pi\right)\right),$ $${  f_1}'(x)=\frac{2x-\sin(2x)}{2x^2\cos(x)^2}>0.$$ Besides, ${  f_1}(x)<0$ if and only if $x\in\cup_{k\in\mathbb{N}}\left(\left(k-\frac{1}{2}\right)\pi,k\pi\right)$ and ${  f_1}(x)\rightarrow 0$ as $x\rightarrow k\pi$ and ${  f_1}(x)\rightarrow +\infty$ as $x\rightarrow \left(k-\frac{1}{2}\right)\pi^-$. Similarly, by a convexity argument, we get that for $x\in(0,\infty)$ $${  f_2}'(x)=\frac{2x-\sinh(2x)}{x^2\cosh(x)^2}<0.$$ Note that for all $x>0$, ${  f_2}(x)$ is positive and that  ${  f_2}(x)\to 0$ as $x\to\infty$ and ${  f_2}(x)\to1$ as $x\to 0$.
    As a consequence, on each interval $\left(\frac{1}{2}(\rho-1-(k+\frac{1}{2})^2\pi^2),\frac{1}{2}(\rho-1-(k-\frac{1}{2})^2\pi^2)\right),$ $k\in\{ 1,...,K-1\}$,
    the function $\lambda\mapsto - {  f_1}(\sqrt{\rho-1-2\lambda})$ is increasing and
    \begin{eqnarray*}
        -{  f_1}(\sqrt{\rho-1-2\lambda})=0  &&  \text{for } \quad\lambda=\frac{1}{2}(\rho-1-k^2\pi^2),\\
        -{  f_1}(\sqrt{\rho-1-2\lambda})\to\infty  &&\text{as } \quad \lambda\to \frac{1}{2}\left(\rho-1-\left(k-\frac{1}{2}\right)^2\pi^2\right)^-.
    \end{eqnarray*}
    On the other hand, the function  $\lambda\mapsto (L-1){  f_2}(\sqrt{2\lambda}(L-1))$ is positive and decreasing on $(0,\infty)$. Hence, Equation (\ref{eq1}) has a unique solution in each interval
    $(\frac{1}{2}(\rho-1-k^2\pi^2),\allowbreak \frac{1}{2}(\rho-1-(k-\frac{1}{2})^2\pi^2)),$ $k\in\{ 1,...,K-1\}$. It has no solution in  $\cup_{k=1}^{K-1}(\frac{1}{2}(\rho-1-(k+\frac{1}{2})^2\pi^2),\frac{1}{2}(\rho-1-k^2\pi^2))$ since $\lambda\mapsto {  f_2}(\sqrt{2\lambda}(L-1))$ is positive  and $\lambda\mapsto -{  f_1}(\sqrt{\rho-1-2\lambda})$ is negative on this set. Then, note that for sufficiently large $L$, Equation (\ref{eq1}) has a unique solution in the interval $(0\vee\frac{1}{2}(\rho-1-K^2\pi^2),\frac{1}{2}(\rho-1-(K-\frac{1}{2})^2\pi^2)$. Indeed, the function $\lambda\mapsto - {  f_1}(\sqrt{\rho-1-2\lambda})$ is positive, increasing and
    \begin{align*}
        -{  f_1}(\sqrt{\rho-1-2\lambda})\to\infty                                      & \quad\text{as} \quad \lambda\to \frac{1}{2}\left(\rho-1-\left(K-\frac{1}{2}\right)^2\pi^2\right), \\
        -{  f_1}(\sqrt{\rho-1-2\lambda})\to 0                                          & \quad \text{as} \quad \lambda\to \frac{1}{2}(\rho-1-K^2\pi^2),                                    \\
        -{  f_1}(\sqrt{\rho-1-2\lambda})\to -\frac{\tan{\sqrt{\rho-1}}}{\sqrt{\rho-1}} & \quad\text{as} \quad \lambda\to 0.
    \end{align*}
    Besides, $\lambda\mapsto {  f_2}(\sqrt{2\lambda}(L-1))$ is positive, decreasing and ${  f_2}(\sqrt{2\lambda}(L-1))\rightarrow L-1$ as $\lambda\to 0$. Therefore, if $L>1-\frac{\tan\sqrt{\rho-1}}{\sqrt{\rho-1}}$, Equation  (\ref{eq1}) has one solution in $(0\vee\frac{1}{2}(\rho-1-K^2\pi^2),\frac{1}{2}(\rho-1-(K-\frac{1}{2})^2\pi^2))$. If it exists, this solution is unique.
    If $\rho-1-K^2\pi^2>0$, there is no solution of (\ref{eq1}) in $(0,\frac{1}{2}(\rho-1-K^2\pi^2)]$ since the LHS of (\ref{eq1}) is negative on this set.
    Therefore, for $L>1-\frac{\tan\sqrt{\rho-1}}{\sqrt{\rho-1}}$, we found exactly $K$ solutions of (\ref{eq1}) in $(0,\frac{\rho-1}{2})$. Conversely, one can check that these solutions are eigenvalues, corresponding to eigenvectors defined by (\ref{vecp1}).

    We now prove that $\lambda=0$  does not belong to the spectrum. Assume that \eqref{syst} has a solution for $\lambda=0$. Then, this solution is of the form
    \begin{equation*} 
        v(x)=\begin{cases}
            C(L-1)\sin\left(\sqrt{\rho-1}\ x\right) & x\in[0,1], \\
            C\sin(\sqrt{\rho-1})(L-x)               & x\in[1,L],
        \end{cases}
    \end{equation*}
    for some $C\neq 0$. Here we use that $v$ is continuous at $1$ and that $\sin(\sqrt{\rho-1})\neq0$. A direct calculation shows that this function is not differentiable at $1$ as soon as  $L>1-\frac{\tan(\sqrt{\rho-1})}{\sqrt{\rho-1}}$. Thus $0$ is not an eigenvalue.

    We now move to the negative part of the spectrum. In this case, a solution $v$ of (\ref{syst}) can be written as $$  v(x)= \begin{cases}
            A\sin\left(\sqrt{\rho-1-2\lambda}\,x\right) & x\in [0,1], \\
            B\sin\left(\sqrt{-2\lambda}\,(L-x)\right)   & x\in[1,L],
        \end{cases}$$
    for some $A\neq 0$, $B\neq0$ (see (v)). Using that $v$ in continuously differentiable at $1$ shows that $A$ and $B$ solve
    \begin{equation} \label{sys:lambda1}
        \begin{cases}
            A\sin(\sqrt{\rho-1-2\lambda})  =B\sin(\sqrt{2\lambda}(L-1)),                                     \\
            A\sqrt{\rho-1-2\lambda}\cos(\sqrt{\rho-1-2\lambda})=-B\sqrt{2\lambda}\cos(\sqrt{2\lambda}(L-1)). \\
        \end{cases}
    \end{equation}
    Again, this shows that $\sqrt{\rho-1-2\lambda}\notin \left\{\left(n-\frac{1}{2}\right),\ n\in \mathbb{N}\right\}$. Moreover, $\lambda$ solves the equation
    \begin{equation}
        -\frac{\tan(\sqrt{\rho-1-2\lambda})}{\sqrt{\rho-1-2\lambda}}=\frac{\tan(\sqrt{-2\lambda}(L-1))}{\sqrt{-2\lambda}}.
        \label{eq2}
    \end{equation}
    Consider the sequences $(A_i)$ and $(a_j)$ defined in (\ref{def:A}) and (\ref{def:aj}). The function $\lambda\mapsto (L-1){  f_1}(\sqrt{-2\lambda}(L-1))$ is defined on $\cup_{j=0}^\infty(-a_{j+1},-a_j)$. In view of the above, $\lambda\mapsto (L-1){  f_1}(\sqrt{-2\lambda}(L-1))$ is decreasing on each interval $(-a_{j+1},-a_j)$. Similarly, the function $\lambda\mapsto -{  f_1}(\sqrt{\rho-1-2\lambda})$ is defined on $\cup_{i=0}^\infty(-A_i,-A_{i-1})$ and is increasing on each interval $(-A_i,-A_{i-1})$. Besides, for all $i\geqslant 0$ and $j\geqslant1$,
    \begin{eqnarray}\label{lim1}
        \lim\limits_{\substack{\lambda \rightarrow -a_{j} \\ x>-a_{j}}} (L-1){  f_1}(\sqrt{-2\lambda}(L-1))&=& +\infty,  \quad
        \lim\limits_{\substack{\lambda \rightarrow -a_{j} \\ x<-a_{j}}} (L-1){  f_1}(\sqrt{-2\lambda}(L-1))= -\infty,\\
        \lim\limits_{\substack{\lambda \rightarrow -A_{i} \\ x<-A_{i}}} -{  f_1}(\sqrt{\rho-1-2\lambda})&=& +\infty,
        \quad \lim\limits_{\substack{\lambda \rightarrow -A_{i} \\ x>-A_{i}}} -{  f_1}(\sqrt{\rho-1-2\lambda})= -\infty. \label{lim2}
    \end{eqnarray}
    Therefore,
    Equation (\ref{eq2}) has a unique solution in each non-empty interval of the form $(-A_i,-A_{i-1})\cap(-a_{j+1},-a_{j})$, $i\geqslant 0$ and  $j\geqslant 1$. On the other hand, for $j=0$ we have
    \begin{eqnarray}
        \lim\limits_{\substack{\lambda \rightarrow 0\\ \lambda<0}} (L-1){  f_1}(\sqrt{-2\lambda}(L-1)) &=&L-1, \;
        \lim\limits_{\substack{\lambda \rightarrow 0\\ \lambda<0}} -{  f_1}(\sqrt{\rho-1-2\lambda}) =-\frac{\tan\sqrt{\rho-1}}{\sqrt{\rho-1}}, \label{lim3}
    \end{eqnarray}
    so that (\ref{eq2}) has no solution in $(-a_1,-a_0)$ as long as $L>1-\frac{\tan{\sqrt{\rho-1}}}{\sqrt{\rho-1}}$. Hence the negative eigenvalues are distributed as follows: there is no eigenvalue in $(-a_1,-a_0)$,  two eigenvalues in each interval $(-a_{j+1},-a_j)$ such that $-A_i\in (-a_{j+1},-a_j)$ for some $i\geqslant 0$, one smaller than $-A_i$ and one larger than $-A_i$, and a unique eigenvalue in each interval $(-a_{j+1},-a_j)$ which does not satisfy the two above conditions. Let us prove that this is equivalent to \eqref{encK}.

    For  $i\geqslant 0$, denote by $n_i$ the largest integer such that $a_{n_i}<A_i$ . One can prove that\begin{equation*}
        n_i=\left\lfloor \frac{(L-1)}{\pi}\sqrt{2A_i}+\frac{1}{2}\right\rfloor.
    \end{equation*}
    Note that $N_0=n_0$. According to \eqref{lim1}, \eqref{lim2} and \eqref{lim3}, for all $0< j\leqslant n_0$, the eigenvalue $\lambda_{K+j}$ is the unique solution of \eqref{eq2} located in the interval $(-a_{j+1},-a_j)$,  which coincides with Equation \eqref{encK}. Assume that \eqref{encK} holds until some $i\geqslant 0$. Then, the eigenvalue $\lambda_{K+N_i+1}$ is the unique solution of \eqref{eq2} located in
    \begin{equation*}
        (-A_{i+1},-A_i)\cap(-a_{n_i+1},-a_{n_i}).
    \end{equation*}
    If we set $j=N_i+1$, then $n_i=N_i-i=j-1-i$ and we get that
    \begin{equation*}
        \lambda_{K+j}\in(-A_{i+1},-A_i)\cap(-a_{j-i},-a_{j-i-1}).
    \end{equation*}
    Similarly, there is a  unique solution of \eqref{eq2} in each interval $$(-A_{i+1},-A_i)\cap(-a_{n_i+k},-a_{n_i+k-1})$$ for all $2\leqslant k\leqslant n_{i+1}-n_i+1$. Hence, for $j=N_i+k=n_i+i+k$, $1\leqslant k\leqslant n_{i+1}-n_i+1$
    \begin{equation}\label{eq:denomb}
        \lambda_{K+j}\in(-A_{i+1},-A_i)\cap (-a_{n_i+k}, -a_{n_i+k-1})= (-A_{i+1},-A_i)\cap (-a_{j-i}, -a_{j-i-1}).
    \end{equation}
    Finally, note that $n_{i+1}-n_i=N_{i+1}-N_i-1$ so that \eqref{eq:denomb} holds for all $N_i+1 \leqslant j\leqslant N_{i+1}.$ This concludes the proof of the lemma.
\end{proof}
It remains to characterise the spectrum in the case $\rho=1+\left(n-\frac{1}{2}\right)^2\pi^2$ for some $n\in\mathbb{N}$. In Lemma \ref{ev:sc}, we prove that the distribution of the eigenvalues is similar to the previous case.
\begin{figure}[t]
    \begin{center}
        \includegraphics[trim=50 30 50 40,clip,scale=0.5]{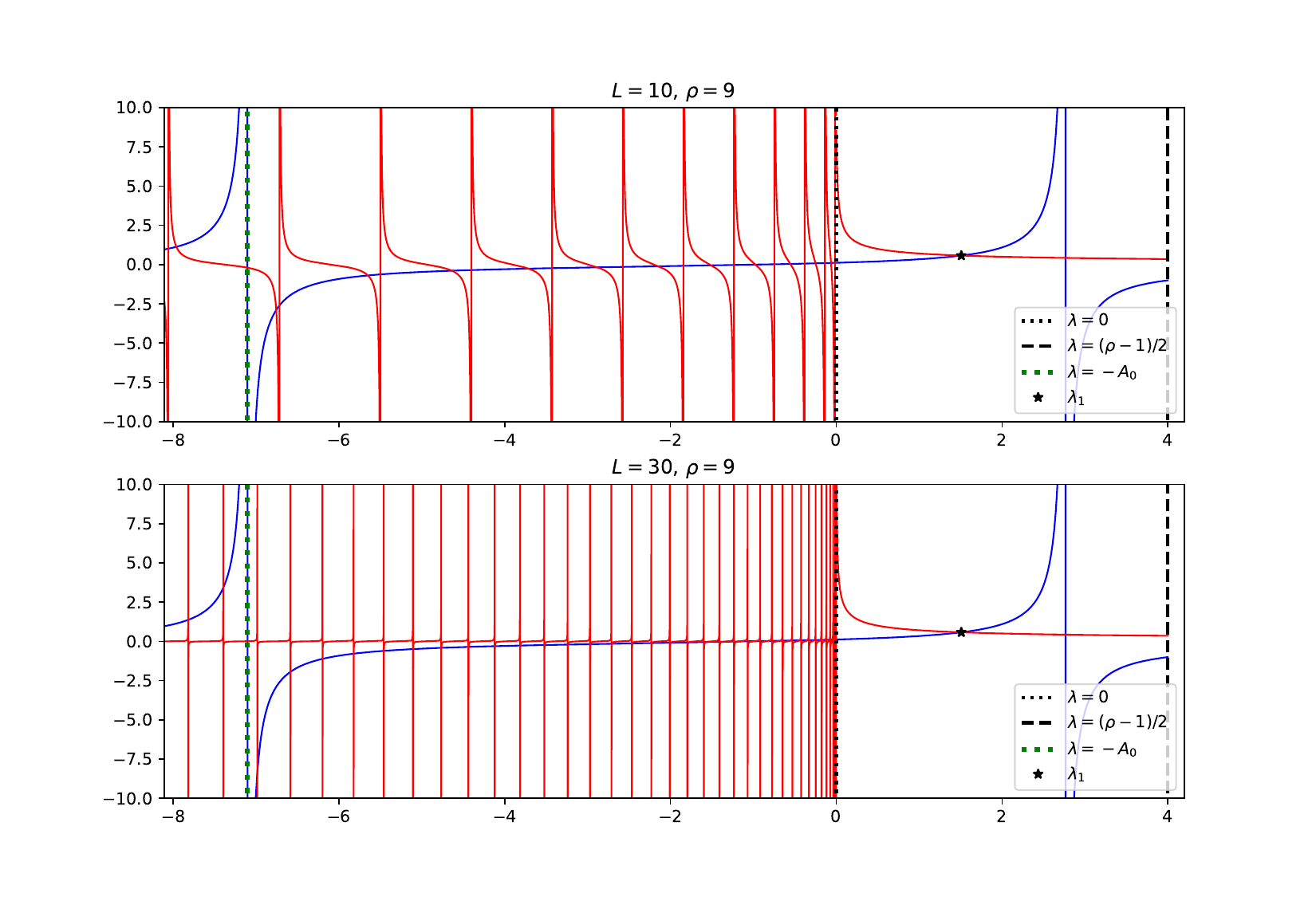}
        \caption{Location of the eigenvalues of the SLP \eqref{eq:defT} for $\rho=4$ and different values of $L$. The blue line represents the LHS of \eqref{square}.  The red line corresponds to the RHS of \eqref{square}. The eigenvalues are located at the intersections of the blue and red solid lines. Note that the negative eigenvalues tend to a continuous spectrum as $L\to\infty$. For $\rho=9$, we have $K=1$.}
        \label{fig:1}
    \end{center}
\end{figure}

\begin{figure}[t]
    \begin{center}
        \includegraphics[trim=50 30 50 40,clip,scale=0.5]{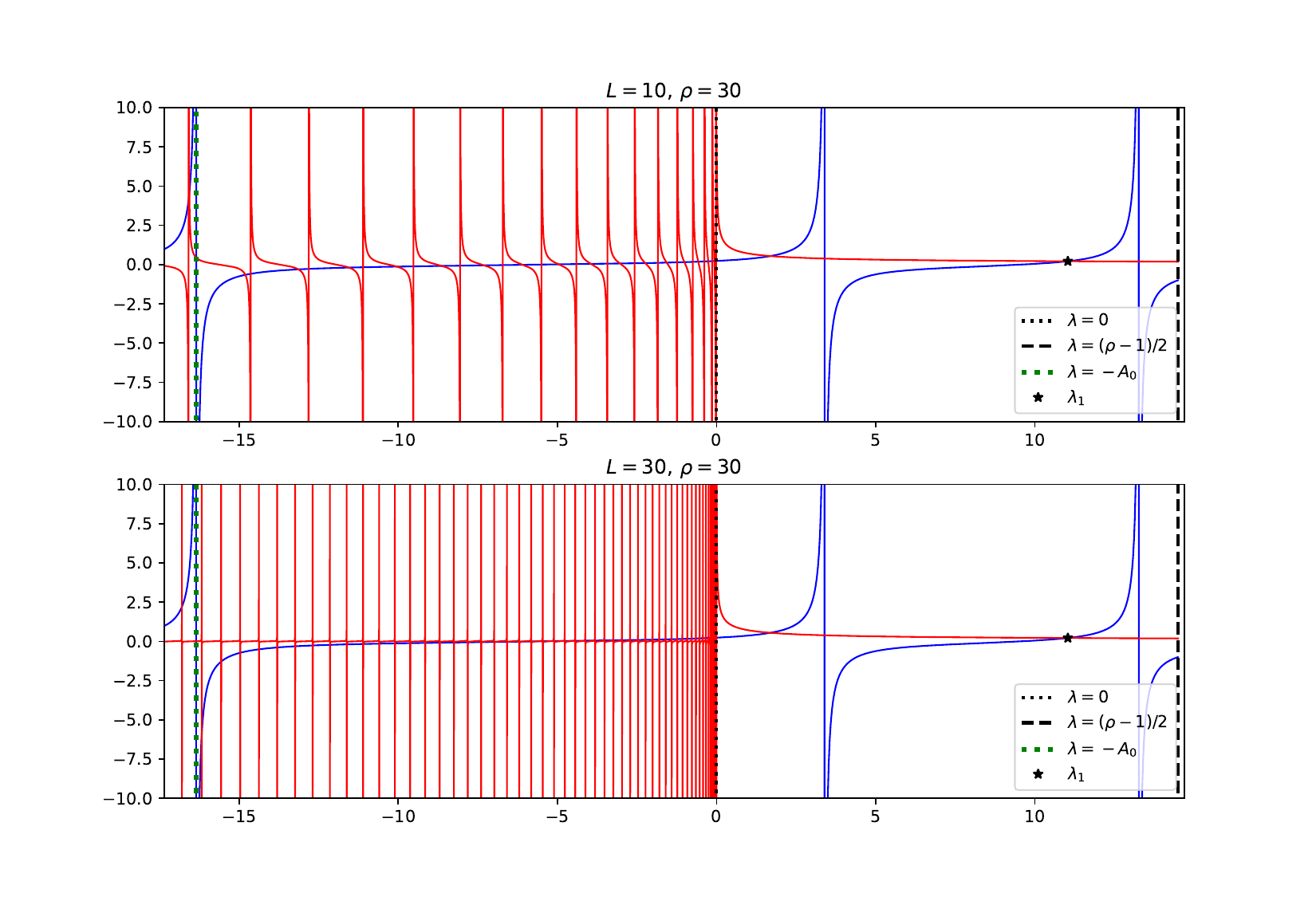} \caption{Location of the eigenvalues of the SLP \eqref{eq:defT} for $\rho=30$ and different values of $L$.  For $\rho=30$, we have $K=2$.}
        \label{fig:2}
    \end{center}
\end{figure}

\begin{lemma}\label{ev:sc}
    Assume $\rho=1+\left(n-\frac{1}{2}\right)^2\pi^2$ for some $n\in\mathbb{N}$ and let $K=n-1$. Then, for all $1\leqslant k \leqslant K$, $\lambda_k$ is the unique solution of \eqref{square}
    such that,
    \begin{equation*}
        (\rho-1-k^2\pi^2\vee 0)<2\lambda_{k}< \rho - 1-\left(k-\frac{1}{2}\right)^2\pi^2.
    \end{equation*}
    Furthermore, $\lambda_k<0$ for all $k > K$. More precisely, for every $i\in\mathbb{N}$ and every $j\in \mathbb{N}$ such that $N_{i-1} < j \le N_{i}$,
    $\lambda_{K+j}$ is the unique solution of  \eqref{square} located  in the interval
    \begin{equation}
        (-A_{i},-A_{i-1}) \cap (-a_{j-i+1},-a_{j-i}), \label{eigenvlocsc}
    \end{equation}
    where the sequences $(N_i)$, $(A_i)$ and $(a_j)$ are the same as the ones defined in Lemma \ref{eigenvloc}.
    In addition, there exists a constant $C>0$ such that for $L$ large enough, we have
    \begin{equation}
        \lambda_{K+1}<-\frac{C}{L^2}. \label{asymp:a1}
    \end{equation}
\end{lemma}
\begin{proof}
    The proof of Lemma \ref{ev:sc} is similar to the one of Lemma \ref{eigenvloc}. Again, one can prove that the spectrum is bounded above by $\frac{\rho-1}{2}$ and that $0$ is not an eigenvalue by proving that the corresponding solutions of \eqref{syst} are not differentiable at $1$.
    The positive eigenvalues can then be defined as above, remarking that in the case of the smallest positive eigenvalue $\lambda_K$ we have $\rho-1-K^2\pi^2>0,
    $ so that the same argument holds.
    For negative eigenvalues, first note that $N_0=A_0=0$ and that Equation \eqref{square} has a solution located in the interval $(-a_0,-a_1)$. As a consequence, all the indices are shifted as stated in the lemma.

    We now prove Equation \eqref{asymp:a1}, which provides an upper bound on the negative part of the spectrum. Let $\lambda<0$ be a negative  eigenvalue of the SLP \eqref{eq:defT} with boundary conditions \eqref{bc}. Hence, as in \eqref{sys:lambda1}, there exist two constants $ A\neq 0, B\neq 0$ such that
    \begin{align*}
        \begin{cases}A\sin(\sqrt{\rho-1-2\lambda} )  =B\sin(\sqrt{-2\lambda}(L-1)) \\
            A\sqrt{\rho-1-2\lambda}\, \cos(\sqrt{\rho-1-2\lambda})  = -B\sqrt{-2\lambda} \cos(\sqrt{-2\lambda}(L-1)).\end{cases}
    \end{align*}
    Let us now assume that for all $C>0$, then there exists $L$ large  such that $\lambda_{K+1}\geqslant -C/L^2$. Then, up to extraction, $\lambda_{K+1}L^2\to0$ as $L\to \infty$. Letting $L$ tend to $+\infty$ in the first line of the system gives that $A=0$, which contradicts the fact that $A,B\in\mathbb{R}^*$. This concludes the proof of the lemma.
\end{proof}

The positions of the eigenvalues for different values of $L$ and $\rho$ are illustrated in Figures \ref{fig:1} and \ref{fig:2}.

\begin{rem} 
    Recall from \eqref{hwp2} that  $\linf\in(0,1/16)$ for all $\rho\in(\rho_1,\rho_2)$. This combined with point (vi), Lemma \ref{eigenvloc} and Lemma \ref{ev:sc} shows that
    $$\rho-1<\pi^2+2\lambda_1<\pi^2+2\linf<\frac{1}{8}+\pi^2 < \left(2-\frac{1}{2}\right)^2\pi^2=\frac{9}{4}\pi^2$$
    for all $\rho\in(\rho_1,\rho_2)$.
    As a consequence, $K=1$ under (\ref{hwp}). Moreover, for $n=1$, $\rho=1+\left(n-\frac{1}{2}\right)^2\pi^2=\rho_1$ and, for $n\geq 2$, $\rho=1+\left(n-\frac{1}{2}\right)^2\pi^2>\rho_2.$

    Although  Lemma \ref{ev:sc} is not necessary to study the semipushed regime, it will ensure that our estimates on the fundamental solution $p_t$ are valid for all $\rho>\rho_1$.
    \label{rm:K1}
\end{rem}

\begin{lemma}[Asymptotic expansions of the positive eigenvalues]\label{eigenvlocasymp} Assume that (\ref{hpushed}) holds. Let $K\in\mathbb{N}$ be the largest positive integer such that
    $
        \rho-1>\left(K-\frac{1}{2}\right)^2\pi^2.
    $ Then, for all $1\leqslant k\leqslant K$, $\lambda_k$ is increasing and tends to the unique solution $\lambda_k^\infty$ of
    \begin{equation}
        -\frac{\tan(\sqrt{\rho-1-2\lambda})}{\sqrt{\rho-1-2\lambda}}=\frac{1}{\sqrt{2\lambda}}
        \label{eqlim}
    \end{equation}{}
    located in the interval
    \begin{equation*}
        \left(\frac{\rho-1-k^2\pi^2}{2}\vee 0,\frac{1}{2}\left(\rho-1 - \left(k-\frac{1}{2}\right)^2\pi^2\right)\right),
    \end{equation*} as $L\to\infty$. Moreover, for all $1\leqslant k\leqslant K$, there exists a constant $C_k(\rho)>0$ such that
    \begin{equation*}
        \lambda_k=\lambda_k^\infty-C_k(\rho)e^{-2\sqrt{2\lambda_k^\infty} L}+o\left(e^{-2\sqrt{2\lambda_k^\infty}L}\right).
    \end{equation*}

\end{lemma}
\begin{proof}
    Recall from Lemma \ref{eigenvloc} and Lemma \ref{ev:sc} that for $L\geqslant L_0(\rho)$ and $1\leqslant k\leqslant K$, $\lambda_k$ is the unique solution of Equation (\ref{square}):
    \begin{equation*}
        \frac{\tan(\sqrt{\rho-1-2\lambda})}{\sqrt{\rho-1-2\lambda}}=-\frac{\tanh(\sqrt{2\lambda}(L-1))}{\sqrt{2\lambda}}
    \end{equation*}
    located in the interval $(\frac{1}{2}(\rho-1-k^2\pi^2\vee 0),\frac{1}{2}(\rho -1- \left(k-\frac{1}{2}\right)^2\pi^2))$.
    Recall from the proof of Lemma \ref{eigenvloc} that the function $\lambda\mapsto -{  f_1}(\sqrt{\rho-1-2\lambda})$ is increasing on each interval $(\frac{1}{2}(\rho-1-k^2\pi^2\vee 0),\frac{1}{2}(\rho -1- \left(k-\frac{1}{2}\right)^2\pi^2))$ and does not depend on $L$. On the other hand,  the function $\lambda\mapsto (L-1){  f_2}(\sqrt{2\lambda}(L-1))$ is decreasing on $(1,\infty)$ for all $L>1$. In addition, the function $L\mapsto (L-1){  f_2}(\sqrt{2\lambda}(L-1))$ is increasing on $[1,+\infty)$ for all $\lambda>0$. As a consequence, $\lambda_k$ is an increasing function of $L$ for all $1\leqslant k\leqslant K$. Since it is upperbounded by $\frac{1}{2}(\rho -1- \left(k-\frac{1}{2}\right)^2\pi^2)$, it converges to some limit $\lambda_k^\infty\in\left(\frac{1}{2}(\rho-1-k^2\pi^2\vee 0),\frac{1}{2}(\rho -1- \left(k-\frac{1}{2}\right)^2\pi^2)\right]$.

    If $\lambda_k^\infty =\frac{1}{2}(\rho -1- \left(k-\frac{1}{2}\right)^2\pi^2),$ the LHS of (\ref{square}) tends to $+\infty$ as $L\to \infty$ whereas the RHS tends to $(2\lambda_k^\infty)^{-\frac{1}{2}}$. Thus, $\lambda_k^\infty\in\left(\frac{1}{2}(\rho-1-k^2\pi^2\vee 0),\frac{1}{2}(\rho -1- \left(k-\frac{1}{2}\right)^2\pi^2)\right).$
    Then, since the RHS and LHS of (\ref{square}) are continuous on each interval $(\frac{1}{2}(\rho-1-k^2\pi^2\vee 0),\frac{1}{2}(\rho -1- \left(k-\frac{1}{2}\right)^2\pi^2))$, we obtain that $\lambda_k^\infty$ is a solution of Equation (\ref{eqlim}). Moreover one can show that this solution is unique.

    Let us now compute an asymptotic expansion of $\lambda_k$ as $L\to\infty$. From now, we assume that $k=1$ but similar calculations can be made for $k\in\llbracket 2,K\rrbracket.$ Let us first recall the definitions of $\beta$ and $\gamma$ from Equation (\ref{defnot}). Note that $\beta>0$ since $\linf>0$ under \eqref{hpushed}. Then, remark that $\cos(\gamma)<0$. Indeed, we know from the first part of the lemma that $\gamma \in \left(\frac{\pi}{2},\pi\right)$ and that
    $\sin(\gamma)=-\frac{\gamma}{\beta}\cos(\gamma).$
    Hence,  $\cos(\gamma)=-\frac{\beta}{\gamma}\sin(\gamma)\leqslant-\frac{1}{2}\frac{\beta}{\gamma}<0$ if $\gamma\in\left(\frac{\pi}{2},\frac{3\pi}{4}\right]$ and $\cos(\gamma)\leqslant -\frac{1}{2}$ if $\gamma\in\left(\frac{3\pi}{4},\pi\right)$.

    Let us now rewrite (\ref{square}) as
    \begin{equation}
        \sqrt{\frac{2\lambda_1}{\rho-1-2\lambda_1}}\tan\left(\sqrt{\rho-1-2\lambda_1}\right)=-\tanh\left(\sqrt{2\lambda_1}(L-1)\right),\label{eq:sqr:2}
    \end{equation}
    and define $h=\lambda_1-\linf.$ As $L\to\infty$, $h\to0$ and
    \begin{eqnarray*}
        \sqrt{\frac{2\lambda_1}{\rho-1-2\lambda_1}}&=&\sqrt{\frac{2(\linf+h)}{\rho-1-2(\linf+h)}}=\left(\frac{2\linf}{\rho-1-2\linf}\right)^{1/2}\left(\frac{1+\frac{h}{\linf}}{1-\frac{2h}{\rho-1-2\linf}}\right)^{1/2}\\
        &=&\frac{\beta}{\gamma}\left(1+\frac{2h}{\beta^2}\right)^{1/2}\left(1-\frac{2h}{\gamma^2}\right)^{-1/2}=\frac{\beta}{\gamma}\left(1+\frac{h}{\beta^2}+o(h)\right)\left(1+\frac{h}{\gamma^2}+o(h)\right)\\
        &=& \frac{\beta}{\gamma}\left(1+\left(\frac{\gamma^2+\beta^2}{\gamma^2\beta^2}\right)h+o(h)\right),
    \end{eqnarray*} and
    \begin{eqnarray*}
        \tan(\sqrt{\rho-1-2\lambda_1})&=&\tan\left(\gamma-\frac{h}{\gamma}+o(h)\right)=\tan(\gamma)-\frac{h}{\gamma\cos(\gamma)^2}+o(h).
    \end{eqnarray*} Then, since $\tan(\gamma)=-\frac{\gamma}{\beta}$, we have
      {\begin{align*}
             & \sqrt{\frac{2\lambda_1}{\rho-1-2\lambda_1}}\tan(\sqrt{\rho-1-2\lambda_1})                                                                             \\
             & =\frac{\beta}{\gamma}\left(\tan(\gamma)+\left(\frac{\gamma^2+\beta^2}{\gamma^2\beta^2}\tan(\gamma)-\frac{1}{\gamma\cos(\gamma)^2}\right)h+o(h)\right) \\
             & =-1-\left(\frac{\gamma^2+\beta^2}{\gamma^2\beta^2}+\frac{\beta}{\gamma^2\cos(\gamma)^2}\right)h+o(h)                                                  \\
             & =-1-\frac{1}{\gamma^2\beta^2\cos(\gamma)^2}\left((\gamma^2+\beta^2)\cos(\gamma)^2+\beta^3\right)h+o(h)                                                \\
             & =-1-\frac{(\rho-1)\cos(\sqrt{\rho-1-2\linf})^2+(2\linf)^{3/2}}{2\linf(\rho-1-2\linf)\cos(\sqrt{\rho-1-2\linf})^2}h+o(h).
        \end{align*}}
    Besides, since $\sqrt{2\lambda_1}L\to \infty$ as $L\to\infty,$
    \begin{equation*}
        \tanh(\sqrt{2\lambda_1}(L-1))=1-2e^{-2\sqrt{2\lambda_1}(L-1)}+o(e^{-2\sqrt{2\lambda_1}(L-1)}).
    \end{equation*}
    Combined with Equation (\ref{eqlim}), this implies that
    \begin{equation}\label{eq:dl:1}
        2e^{-2\sqrt{2\lambda_1}(L-1)}+o(e^{-2\sqrt{2\lambda_1}(L-1)})=\frac{(\rho-1)\cos(\sqrt{\rho-1-2\linf})^2+(2\linf)^{3/2}}{2\linf(\rho-1-2\linf)\cos(\sqrt{\rho-1-2\linf})^2}h+o(h).
    \end{equation}
    In addition, we obtain that $Lh\rightarrow 0$ as $L\to\infty$ and that
    \begin{equation*}
        e^{-2\sqrt{2\lambda_1}(L-1)}= e^{-2\beta\left(1+\frac{1}{2}\frac{h}{\beta^2}+o(h)\right)(L-1)}=e^{-2\beta(L-1)}e^{o(1)}=e^{-2\beta(L-1)}+o(e^{-2\beta(L)}).
    \end{equation*}
    Finally, according to Equation (\ref{eq:dl:1}), we have
    \begin{equation*}
        \lambda_1-\linf= -2\frac{2\linf(\rho-1-2\linf)\cos(\sqrt{\rho-1-2\linf})^2}{(\rho-1)\cos(\sqrt{\rho-1-2\linf})^2+(2\linf)^{3/2}}e^{-2\beta (L-1))}+o(e^{-2\beta L}),
    \end{equation*}
    with $$\frac{2\linf(\rho-1-2\linf)\cos(\sqrt{\rho-1-2\linf})^2}{(\rho-1)\cos(\sqrt{\rho-1-2\linf})^2+(2\linf)^{3/2}}>0,$$ since $\cos(\sqrt{\rho-1-2\linf})^2=\cos(\gamma)^2>0.$
\end{proof}
\begin{rem}\label{lb:sin} The equation used to prove that $\cos(\gamma)>0$ in the above lemma  also implies that for sufficiently large $L$
    \begin{equation}\label{ub:sin}
        \frac{\sin(\sqrt{\rho-1-2\lambda_1})}{\sqrt{\rho-1-2\lambda_1}}>\frac{1}{4}\frac{1}{\gamma\wedge \beta}.
    \end{equation}
\end{rem}
\begin{rem}\label{r:vw} The asymptotic expansion of $\lambda_1$ gives that
    \begin{equation*}
        \sinh(\sqrt{2\lambda_1}(L-1))= \frac{1}{2}e^{\beta (L-1)}+o\left(e^{\beta L}\right).
    \end{equation*}
\end{rem}

\begin{lemma}[$\mathrm{L}^2$-norm of the first eigenvector] \label{lem:l2norm}  Assume that (\ref{hpushed}) holds. As $L\to\infty$,
    \begin{equation*}
        \|v_1\|^2\rightarrow \frac{1}{2}\frac{(\rho-1)\cos(\sqrt{\rho-1-2\linf})^2+(2\linf)^{3/2}}{\sqrt {2\linf}(\rho-1-2\linf)\cos(\sqrt{\rho-1-2\linf})}.
    \end{equation*}
\end{lemma}
\begin{proof}
    The $\mathrm{L}^2$-norm of the function $v_1$ is given by
    \begin{equation*}
        \|v_1\|^2=\int_0^Lv_1(x)^2dx=\frac{1-\frac{\sin(2\sqrt{\rho-1-2\lambda_1})}{2\sqrt{\rho-1-2\lambda_1}}}{2\sin(\sqrt{\rho-1-2\lambda_1})^2}+\frac{\frac{\sinh(2\sqrt{2\lambda_1}(L-1))}{2\sqrt{2\lambda_1}}-(L-1)}{2\sinh(\sqrt{2\lambda_1}(L-1))^2}.
    \end{equation*}
    The first term of the RHS tends to $\frac{1}{2\sin(\gamma)^2}\left(1-\frac{\sin(2\gamma)}{2\gamma}\right)$ and the second one to $\frac{1}{2\beta}$ as $L\to\infty$. Besides, we know thanks to Equation (\ref{eqlim}) that  $\sin(\gamma)=-\frac{\gamma}{\beta}\cos(\gamma)$. Therefore,
    \begin{align*}
         & \frac{1}{2\sin(\gamma)^2}\left(1-\frac{\sin(2\gamma)}{2\gamma}\right)+\frac{1}{2\beta}                                                                                                                  \\&=\frac{1}{2}\frac{\beta^2}{\gamma^2}\left(\frac{1-\frac{\sin(\gamma)\cos(\gamma)}{\gamma}}{\cos(\gamma)^2}+\frac{1}{2\beta}\right)=\frac{1}{2}\frac{\beta^2}{\gamma^2}\left(\frac{1+\frac{\cos(\gamma)^2}{\beta}}{\cos(\gamma)^2}\right)+\frac{1}{2\beta}\\
         & = \frac{1}{2}\frac{\beta}{\gamma^2}\left(\frac{\beta+\cos(\gamma)^2}{\cos(\gamma)^2}\right)+\frac{1}{2\beta}\frac{1}{2\beta\gamma^2\cos(\gamma)^2}\left((\gamma^2+\beta^2)\cos(\gamma)^2+\beta^3\right) \\
         & =\frac{1}{2}\frac{(\rho-1)\cos(\sqrt{\rho-1-2\linf})^2+(2\linf)^{3/2}}{\sqrt {2\linf}(\rho-1-2\linf)\cos(\sqrt{\rho-1-2\linf})}.
    \end{align*}
\end{proof}
\begin{cor}[Asymptotic expansion of the maximal eigenvalue] \label{exp:lambda1}  Assume that (\ref{hpushed}) holds. Then, as $L\to \infty$,
    \begin{equation}
        \lambda_1=\linf-\frac{\beta}{\lim\limits_{L\rightarrow\infty}\|v_1\|^2}e^{-2\beta(L-1)}+o(e^{-2\beta L}).
    \end{equation}
\end{cor}

In Lemma \ref{lem:est:evk} we give several bounds on the eigenvectors $(v_k)$  under \eqref{hpushed}. The proof of this lemma, which relies on explicit calculations, is given in Appendix~\ref{appendix}.
\begin{lemma}\label{lem:est:evk} Assume that (\ref{hpushed}) holds and let $K$ be the largest integer such that $
        \rho-1>\left(K-\frac{1}{2}\right)^2\pi^2.$
    There exist some constants $C_1,C_2,C_3,C_4,C_5$ (that only depend on $\rho$) such that for $L$ large enough
    \begin{itemize}
        \item[(i)]the norms of the vectors $v_k$ are bounded below by\begin{equation*}
                  \|v_k\|^2\geqslant C_1, \quad  \forall  k\in\llbracket2,K\rrbracket,
              \end{equation*}
              and
              \begin{equation*}
                  \|v_k\|^2\geqslant \frac{C_2}{\sin(\sqrt{\rho-1-2\lambda_k})^2\wedge\sin(\sqrt{-2\lambda_k}(L-1))^2}, \quad  \forall k> K.
              \end{equation*}
        \item[(ii)] the ratio $v_k/v_1$ is  bounded above by\begin{equation*}
                  \frac{|v_k(x)|}{\|v_k\|}\leqslant C_3 e^{\beta L}\;\frac{v_1(x)}{\|v_1\|},\quad \forall x\in[0,L], \quad  \forall  k\in\llbracket2,K\rrbracket,
              \end{equation*}
              and
              \begin{equation*}
                  \frac{|v_k(x)|}{\|v_k\|}\leqslant C_4 \sqrt{\rho-1-2\lambda_k}\;e^{\beta L}\;\frac{v_1(x)}{\|v_1\|},\quad \forall x\in[0,L],  \quad \forall k> K.
              \end{equation*}
        \item[(iii)] the ration $v_k/\|v_k\|$ is bounded above by
              \begin{equation*}
                  \frac{v_k(x)}{\|v_k\|}\leqslant C_5, \quad \forall x\in[0,L], \quad \forall k>K.
              \end{equation*}
    \end{itemize}
\end{lemma}

\subsection{Heat kernel estimates} \label{sec:heatkernel}
Recall from Equation \eqref{def:mu} that
\begin{equation*}
    \mu=\sqrt{1+2\linf}.
\end{equation*}
For the remainder of Section \ref{sec:density}, we consider a dyadic BBM with space-dependent branching rate $r(x)$ and drift $-\mu$, killed upon reaching $0$ and $L$.
Recall from Lemma \ref{lem:many-to-one} that the density of particles in this BBM is given by the fundamental solution of \eqref{PDE:A}. By definition of $\mu$, Equation \eqref{def:pt} can be written as
\begin{equation}
    p_t(x,y)=e^{\mu (x-y)}e^{-\linf t} { \tilde p_t(x,y)}= e^{\mu (x-y)}\sum_{k=1}^\infty e^{(\lambda_k-\linf)t}\frac{v_k(x)v_k(y)}{\|v_k\|^2}, \label{def:pt1}
\end{equation}
where the eigenvalues $\lambda_k$ and the eigenvectors $v_k$ are the ones defined in Lemma \ref{eigenvloc}.

In Lemma \ref{th1}, we prove that, under \eqref{hpushed}, after a time of order $L$, the density $p_t$ is well approximated by its first term. In Lemma \ref{lem:212}, we bound $p_t$ for all $t>1$ in the case $K=1$ (this technical lemma will be required at the end of the article).

\begin{lemma}\label{th1} Assume that (\ref{hpushed}) holds. There exists $c_{\ref{th1}}>0$ (that only depends on $\rho$) such that if $L$ is sufficiently large and $t>c_{\ref{th1}} L,$ then, for all $x,y\in[0,L]$
    \begin{equation}
        \left|p_t(x,y)-e^{\mu(x-y)}e^{(\lambda_1-\linf)t}\frac{v_1(x)v_1(y)}{\|v_1\|^2}\right|\leqslant e^{-\beta L}e^{\mu(x-y)}e^{(\lambda_1-\linf)t}\frac{v_1(x)v_1(y)}{\|v_1\|^2}. \label{rem:asymptM}
    \end{equation}
\end{lemma}

\begin{proof} We divide the sum (\ref{def:pt1}) into two parts, according to the sign of $\lambda_k$:
    \begin{equation}    \label{eq:sum}
        { \tilde p_t(x,y)}- e^{\lambda_1 t}\frac{v_1(x)v_1(y)}{\|v_1\|^2} = \underbrace{\sum_{k=2}^K\frac{1}{\|v_k\|^2}e^{\lambda_k t}v_k(x)v_k(y)}_{%
        \let\scriptstyle\textstyle
        \substack{=:S_1}}+\underbrace{\sum_{k=K+1}^\infty\frac{1}{\|v_k\|^2}e^{\lambda_k t}v_k(x)v_k(y)}_{%
        \let\scriptstyle\textstyle
        \substack{=:S_2}}.
    \end{equation}
    We know from Lemma \ref{eigenvloc} that for $L$ large enough and $k\geqslant 2$,
    \begin{equation}\label{ub:l1lk}
        \lambda_1-\lambda_k\geqslant \begin{cases}\frac{5}{8}\pi^2 & \text{if} \quad K\geqslant 2, \\
             \frac{\linf}{2}  & \text{if} \quad  K=1.\end{cases}
    \end{equation}
    We first bound $S_1$. According to Lemma \ref{lem:est:evk}, for $L$ large enough and $k\in\llbracket2,K\rrbracket$,
    \begin{equation*}
        \frac{1}{\|v_k\|^2}\left|v_k(x)v_k(y)\right|\leqslant  C e^{2\beta L-\lambda_1 t}\left(e^{\lambda_1 t}\frac{v_1(x)v_1(y)}{\|v_1\|^2}\right).
    \end{equation*}
    Combining this with Equation (\ref{ub:l1lk}), we get that
    \begin{equation}
        \left|S_1\right|\leqslant c_1e^{2\beta L-\frac{5}{8}\pi^2 t}\left(e^{\lambda_1 t}\frac{v_1(x)v_1(y)}{\|v_1\|^2}\right),
        \label{est:sum1}
    \end{equation}
      {for some $c_1=c_1(\rho)>0$}.
    We now bound $S_2$. We know from Lemma \ref{lem:est:evk} that for $L$ large enough  and  $k>K$,
    \begin{equation}
        \frac{1}{\|v_k\|^2}e^{\lambda_k t}\left|v_k(x)v_k(y)\right|
        \leqslant C e^{2\beta L-\lambda_1 t}\left(e^{\lambda_1 t}
        \frac{v_1(x)v_1(y)}{\|v_1\|^2}\right)\left[(\rho-1-2\lambda_k)e^{\lambda_k}\right].
        \label{est:sum2}
    \end{equation}
    Moreover, recall from
    Lemma \ref{eigenvloc} that for $i\geq 0$ and
    $N_{i-1}<j\leqslant N_i$,  $$-A_i<\lambda_{K+j}<-A_{i-1},$$
    so that we can group the terms when summing the third factor on the RHS of  \eqref{est:sum2} over $k>K$:
    \begin{align}
        \label{est:sum3}
        S_3 & :=e^{-\lambda_{K+1}t}\sum_{k=K+1}^\infty(\rho-1-2\lambda_k)e^{\lambda_k t}                            \\
        \nonumber
            & =\sum_{i=0}^\infty\sum_{j=N_{i-1}+1}^{N_i}(\rho-1-2\lambda_{K+j})e^{(\lambda_{K+j}-\lambda_{K+1})t}   \\
            & \leqslant  \sum_{i=0}^\infty\sum_{j=N_{i-1}+1}^{N_i}(\rho-1+2A_i)e^{(0\wedge(A_0-A_{i-1}))t}\nonumber \\
            & \leqslant\sum_{i=0}^\infty(N_i-N_{i-1})(\rho-1+2A_i)e^{(0\wedge(A_0-A_{i-1}))t}.\nonumber
    \end{align}
    By definition of $(N_i)$ and $(A_i)$, we get that
    \begin{equation}
        \rho-1+2A_i=\left(K+\frac{1}{2}+i\right)^2, \label{est:rho}
    \end{equation}
    and
    \begin{eqnarray*}
        N_i-N_{i-1}
        &\leqslant&\frac{\sqrt{2}}{\pi}(L-1)\left(\sqrt{2A_i}-\sqrt{2A_{i-1}}\right)+2, \\
        \sqrt{A_i}-\sqrt{A_{i-1}}&=&\frac{A_i-A_{i-1}}{\sqrt{A_i}+\sqrt{A_{i-1}}}\leqslant \frac{A_i-A_{i-1}}{\sqrt{A_0}},\\
        A_i-A_{i-1}&\leqslant& C(i+1),
    \end{eqnarray*}
    for all $i\geqslant 1$. For $i=0$,  $N_0-N_{-1}=\frac{L-1}{\pi}\sqrt{A_0}$. Hence, we obtain that  \begin{equation}
        N_i-N_{i-1}\leqslant C L (i+1) \quad \forall i\geqslant 0.
        \label{est:N}
    \end{equation}
    The exponential factor in $S_3$ is then bounded above using the definition of $A_i$. For $i\in\mathbb{N}$, we have
    \begin{eqnarray}
        2(A_0-A_{i-1})&=&\left(K+\frac{1}{2}\right)^2-\left(K+\frac{1}{2}+i-1\right)^2=-2(i-1)\left(K+\frac{1}{2}\right)-(i-1)^2\nonumber\\
        &\leqslant&-(i-1)^2.
        \label{est:A0}
    \end{eqnarray}
    Combining Equations (\ref{est:sum3}), \eqref{est:rho}, (\ref{est:N}) and (\ref{est:A0}), we obtain that for sufficiently large $L$ and $t>1$,
    \begin{equation}
        S_3\leqslant CL\left(1+\sum_{i=1}^\infty(i+1)\left(K+\frac{1}{2}+i\right)^2e^{-\frac{1}{2}(i-1)^2t}\right)\leqslant CL. \label{ub:S3}
    \end{equation}
    Therefore, this estimate combined with Equations (\ref{ub:l1lk}) and (\ref{est:sum2}) implies that for $L$ large enough and $t>1$ , we have
    \begin{equation}
        \left|S_2\right|\leqslant Ce^{2\beta L}e^{(\lambda_{K+1}-\lambda_1)t}S_3\leqslant c_2Le^{2\beta L-\frac{\linf}{2}t}\left(e^{\lambda_1 t}\frac{v_1(x)v_1(y)}{\|v_1\|^2}\right),
        \label{est:S2}
    \end{equation}
      {for some $c_2=c_2(\rho)>0.$}   {Finally, we see from  Equations \eqref{eq:sum}, (\ref{est:sum1}) and (\ref{est:S2}) that it is sufficient to choose $c_{\ref{th1}}>0$ such that
        \begin{equation*}
            \frac{5}{8}\pi^2 c_{\ref{th1}}>3\beta+1 \quad \text{and} \quad  \frac{\linf}{2} c_{\ref{th1}}>3\beta +1,
        \end{equation*}so that
        \begin{align*}
            \left|{ \tilde p_t(x,y)}- e^{\lambda_1 t}\frac{v_1(x)v_1(y)}{\|v_1\|^2}\right| & \leq \left(c_1e^{-L}+c_2Le^{-L}\right)e^{-\beta L}e^{\lambda_1 t}\frac{v_1(x)v_1(y)}{\|v_1\|^2} \\
                                                                                                       & \leq e^{-\beta L}e^{\lambda_1 t}\frac{v_1(x)v_1(y)}{\|v_1\|^2},
        \end{align*}
        for $L$ large enough. The result then follows from \eqref{def:pt1}.}
\end{proof}

\begin{lemma}\label{lem:212}
    Assume that (\ref{hwp}) holds. There exists a positive constant $C>0$ (that only depends on $\rho$) such that  the following holds: for $L$ large enough, $t>1$ and $x,y\in[0,L]$,
    \begin{equation*}
        p_t(x,y)\leqslant Ce^{\mu(x-y)}\left(v_1(x)v_1(y)+Le^{-\linf t}\right).
    \end{equation*}
\end{lemma}
\begin{proof} First, recall from Remark \ref{rm:K1} that $K=1$ under (\ref{hwp}).
      {Hence, we see from Lemma \ref{lem:est:evk} (iii) that for $L$ large enough
        \begin{equation*}
            \sup_{\substack{x\in[0,L] \\ k\geq 2}} \frac{v_k(x)}{\|v_k\|}\leq C_5.
        \end{equation*}
        Putting this together with Lemma \ref{lem:l2norm} and the fact that $\lambda_1<\linf$ (see point (vi)), we get that for $L$ large enough}
    \begin{equation*}
        p_t(x,y)\leqslant Ce^{\mu(x-y)}\left(v_1(x)v_1(y)+\sum_{k=2}^\infty e^{(\lambda_k-\linf)t}\right).
    \end{equation*}
    The sum on the RHS can be bounded using the estimates established in the proof of Lemma \ref{th1}. Indeed, since $\lambda_k<0$ for $k\geq 2$,
    \begin{equation*}
        \sum_{k=2}^\infty e^{(\lambda_k-\linf)t}=e^{(\lambda_2-\linf)t}\sum_{k=2}^\infty e^{(\lambda_k-\lambda_2)t}\leqslant (\rho-1)^{-1}e^{(\lambda_2-\linf)t}S_3\leqslant  (\rho-1)^{-1}e^{-\linf t}S_3,
    \end{equation*}
    where $S_3$ is the sum defined in \eqref{est:sum3}. Yet, we know (see \eqref{ub:S3}) that $S_3\leqslant CL$ for $t>1$. This concludes the proof of the lemma.
\end{proof}

\subsection{The Green function}\label{sec:green}

In this section, we control the integral of the density $p_t$  with respect to the time variable, that is $
    \int_0^t p_s(x,y)ds. $
This quantity will play a central role in the second moment calculations. To bound this quantity, we will need to introduce the Green function $G$ associated to the PDE \eqref{PDE:A}

Let $B_t$ be a standard one-dimensional Brownian motion. For $\lambda\geq0$, define the Green function $H_\lambda$ such that, if $(B_t,t\geq0)$  starts from $B_0=x$ and if $\tau=\inf\{t:B_t\notin (0,L)\}$, then for all bounded measurable functions $f$, we have
\begin{equation*}
    \mathbb{E}\left[\int_0^\tau  \exp\left(\int_0^t\left[r(B_u)-\frac{1}{2}-\lambda\right]\;du\right)f(B_t)\right]=\int_0^LH_\lambda(x,y)f(y)dy.
\end{equation*}
The many-to-one lemma  yields
\begin{equation*}
    H_\lambda(x,y)=\int_0^\infty e^{-\lambda t} { \tilde p_t(x,y)}\,dt=e^{\mu(y-x)}\int_0^\infty e^{(\linf-\lambda) t}p_t(x,y)dt.
\end{equation*}
For $\xi\geq 0$, define
\begin{eqnarray}\label{def:GH}
    \quad G_\xi(x,y) = e^{\mu(x-y)}H_{\linf+\xi}(x,y).
\end{eqnarray}
A first idea to estimate  $
    \int_0^t p_s(x,y)ds $
would be to bound it by $G_0(x,y)$ as in \cite{Berestycki2010}.
However, to get a sharper bound, that depends on $t$, we will rather consider the function $G_\xi$ for some $\xi>0$ and point out that
\begin{equation}\label{rk:green}
    \int_0^t p_s(x,y)ds=\int_0^\infty p_s(x,y)\mathbf{1}_{s\in[0,t]}ds\leqslant \int_0^\infty e^{\frac{t-s}{t}}p_s(x,y)ds\leqslant e G_{\frac{1}{t}} (x,y).
\end{equation}
We will first give an explicit formula for $G_\xi$ following   \cite[Chapter II]{Borodin:2012aa} and then  bound $G_\xi$ for different functions $\xi(L)$ such that $\xi(L)\to 0$ as $L\to \infty$ (see Lemma \ref{lem:alphat} and Lemma \ref{lemma:wro}). We now introduce some notation that will be used in the two following lemmas.  For $\lambda>0$, set
\begin{eqnarray}
    {  \tilde f_1}(\lambda)&=&\sqrt{2\lambda}\sin(\sqrt{\rho-1-2\lambda})+\sqrt{\rho-1-2\lambda}\cos(\sqrt{\rho-1-2\lambda})\nonumber,\\
    {  \tilde f_2}(\lambda)&=&\sqrt{2\lambda}\sin(\sqrt{\rho-1-2\lambda})-\sqrt{\rho-1-2\lambda}\cos(\sqrt{\rho-1-2\lambda})\label{def:walpha},\\
    \omega_\lambda&=&{  \tilde f_1}(\lambda)e^{\sqrt{2\lambda}(L-1)}+{  \tilde f_2}(\lambda)e^{-\sqrt{2\lambda}(L-1)}. \nonumber
\end{eqnarray}
We recall from Lemma \ref{eigenvlocasymp} that $\linf$ is the unique solution of Equation (\ref{eqlim}) such that $$\gamma=\sqrt{\rho-1-2\linf}\in\left(\frac{\pi}{2},\pi\right).$$ Therefore, ${  \tilde f_1}(\linf)=0$ and ${  \tilde f_2}(\linf)>0$. Furthermore,
\begin{equation*}
    {  \tilde f_1}'(\linf)=-\frac{(1+\sqrt{2\linf})(\rho-1)}{2\linf\sqrt{\rho-1-2\linf}}\cos(\sqrt{\rho-1-2\linf})> 0.
\end{equation*}
Let $\varphi_\lambda$ and $\psi_\lambda$ be solutions of
$$\frac{1}{2}u''(x)+\frac{\rho-1}{2}u(x)\mathbf{1}_{x\leqslant 1}=\lambda u(x),$$
such that $\varphi_\lambda(0)=0$ and $\psi_\lambda(L)=0$.
If $0<\lambda<\frac{\rho-1}{2},$ up to multiplication by a constant, we can assume the existence of constants $A,B,C,D$ such that
\begin{equation*}
    \varphi_\lambda(x)=\begin{cases}
        \sin(\sqrt{\rho-1-2\lambda}\,x)              & x\in [0,1], \\
        Ae^{\sqrt{2\lambda}x}+Be^{-\sqrt{2\lambda}x} & x\in[1,L],  \\
    \end{cases}
\end{equation*}
and
\begin{equation*}
    \psi_\lambda(x)=\begin{cases}
        C\cos(\sqrt{\rho-1-2\lambda}x)+D\sin(\sqrt{\rho-1-2\lambda}x) & x\in [0,1], \\
        \sinh(\sqrt{2\lambda}(L-x))                                   & x\in[1,L].  \\
    \end{cases}
\end{equation*}{}
Since $\psi_\lambda$ and $\varphi_\lambda$ are continuous and differentiable at $1$, the constants $A$ and $B$ satisfy
\begin{equation*}
    \begin{cases}
        \sin(\sqrt{\rho-1-2\lambda})=Ae^{\sqrt{2\lambda}}+Be^{-\sqrt{2\lambda}}                                         \\
        \sqrt{\rho-1-2\lambda}\cos(\sqrt{\rho-1-2\lambda})=\sqrt{2\lambda}(Ae^{\sqrt{2\lambda}}-Be^{-\sqrt{2\lambda}}), \\
    \end{cases}
\end{equation*}
and the constants $C$ and $D$ solve
\begin{equation*}
    \begin{cases}
        C\cos(\sqrt{\rho-1-2\lambda})+D\sin(\sqrt{\rho-1-2\lambda})=\sinh(\sqrt{2\lambda}(L-1))                                         \\
        \sqrt{\rho-1-2\lambda}(-C\sin(\sqrt{\rho-1-2\lambda}+D\cos(\sqrt{\rho-1-2\lambda})=-\sqrt{2\lambda}\cosh(\sqrt{2\lambda}(L-1)). \\
    \end{cases}
\end{equation*}
Hence,
$$\begin{cases}
        A=\frac{e^{-\sqrt{2\lambda}}}{2}\left(\sin(\sqrt{\rho-1-2\lambda})+\frac{\sqrt{\rho-1-2\lambda}}{\sqrt{2\lambda}}\cos(\sqrt{\rho-1-2\lambda})\right)=\frac{1}{2\sqrt{\lambda}}{  \tilde f_1}(\lambda)e^{-\sqrt{2\lambda}} \\
        B=\frac{e^{\sqrt{2\lambda}}}{2}\left(\sin(\sqrt{\rho-1-2\lambda})-\frac{\sqrt{\rho-1-2\lambda}}{\sqrt{2\lambda}}\cos(\sqrt{\rho-1-2\lambda})\right)=\frac{1}{2\sqrt{\lambda}}{  \tilde f_2}(\lambda)e^{\sqrt{2\lambda}}   \\
        C=\cos(\sqrt{\rho-1-2\lambda})\sinh(\sqrt{2\lambda}(L-1))+\frac{\sqrt{2\lambda}}{\sqrt{\rho-1-2\lambda}}\sin(\sqrt{\rho-1-2\lambda})\cosh(\sqrt{2\lambda})                                                                            \\
        D=\sin(\sqrt{\rho-1-2\lambda})\sinh(\sqrt{2\lambda}(L-1))-\frac{\sqrt{2\lambda}}{\sqrt{\rho-1-2\lambda}}\cos(\sqrt{\rho-1-2\lambda})\cosh(\sqrt{2\lambda}).
    \end{cases}$$
Therefore, the functions $\varphi_\lambda$ and $\psi_\lambda$ are given  (up to a multiplicative factor) by
\begin{equation}
    \varphi_\lambda (x)=\begin{cases}
        \sin(\sqrt{\rho-1-2\lambda}\,x)                                                                                                                                 & x\in [0,1], \\
        \frac{1}{2\sqrt{2\lambda}}\left({  \tilde f_1}(\lambda)e^{\sqrt{2\lambda}(x-1)}+{  \tilde f_2}(\lambda)e^{-\sqrt{2\lambda}(x-1)}\right) & x\in[1,L],
    \end{cases}
    \label{nt:phi}
\end{equation}
and
\begin{equation}
    \psi_\lambda(x)=\begin{cases}
        \sinh(\sqrt{2\lambda}(L-1))\cos(\sqrt{\rho-1-2\lambda}(x-1))                                                &             \\
        -\frac{\sqrt{2\lambda}}{\sqrt{\rho-1-2\lambda}}\cosh(\sqrt{2\lambda}(L-1))\sin(\sqrt{\rho-1-2\lambda}(x-1)) & x\in [0,1], \\
        \sinh(\sqrt{2\lambda}(L-x))                                                                                 & x\in[1,L].  \\
    \end{cases}
    \label{nt:psi}
\end{equation}
Note that $\omega_\lambda$ defined in \eqref{def:walpha} corresponds to the Wronskian of $\varphi_\lambda$ and $\psi_\lambda$: it satisfies $\omega_\lambda=\psi_{\lambda}(x)\varphi'_{\lambda}(x)-\psi'_{\lambda}(x)\varphi_{\lambda}(x)$ for all $x\in[0,L]$. It is well-known (see \cite[Chapter II]{Borodin:2012aa}) that  $G_\xi$ can be written as
\begin{equation}
    G_{\xi}(x,y)=\begin{cases}
        (\omega_{{\linf+\xi}})^{-1}e^{\mu(x-y)}\psi_{{\linf+\xi}}(x)\varphi_{{\linf+\xi}}(y) & y\leqslant x, \\
        (\omega_{{\linf+\xi}})^{-1}e^{\mu(x-y)}\psi_{{\linf+\xi}}(y)\varphi_{{\linf+\xi}}(x) & y\geqslant x. \\
    \end{cases}
    \label{greenfun}
\end{equation}

\begin{lemma} Assume that (\ref{hpushed}) holds and let $\xi:(1,\infty)\rightarrow(0,\infty)$ \label{lem:alphat} be a function such that   {$\xi(L)=o(1/L)$} as $L\to\infty$.
    There exists a constant $C>1$ (that only depends on $\rho$) such that if $L$ is sufficiently large, then
    \begin{equation*}
        \omega_{\linf+\xi} \geqslant C^{-1}\xi(L)e^{\beta L},
    \end{equation*}
    and for all $x\in[0,L]$,
    \begin{equation*}
        \varphi_{\linf+\xi}(x)\leqslant C(1\wedge x)\left(\xi(L) e^{\beta x}+e^{-\beta x}\right),
    \end{equation*}
    and
    \begin{equation*}
        \psi_{\linf+\xi}(x)\leqslant C(1\wedge(L-x))e^{\beta (L-x)}.
    \end{equation*}
\end{lemma}

\begin{lemma} \label{lemma:wro} Assume that (\ref{hpushed}) holds. Let $h>0$ and $\xi=\frac{h}{L}$. There exists $C>1$ (that only depends on $\rho$ and $h$) such that if $L$ is sufficiently large, then
    \begin{equation*}
        \omega_{\linf+\xi}\geqslant
          {C^{-1}\frac{1}{L}e^{\beta L},}
    \end{equation*}
    and for all $x\in[0,L]$,
    \begin{equation*}
        \varphi_{\linf+\xi}(x)\leqslant C(1\wedge x)\left(\frac{  {1}}{L} e^{\beta x}+e^{-\beta x}\right),
    \end{equation*}
    and
    \begin{equation*}
        \psi_{\linf+\xi}(x)\leqslant C(1\wedge(L-x))e^{\beta (L-x)}.
    \end{equation*}
\end{lemma}
The proofs of Lemma \ref{lem:alphat} and Lemma \ref{lemma:wro} can be found in Appendix \ref{appendix:greenf}.

\section[Branching Brownian motion in an interval: moment estimates]{BBM in an interval: moment estimates}\label{sec:moment:est}

In Section \ref{sec:moment:est}, we assume that (\ref{hpushed}) holds and we consider the  BBM defined in Section \ref{sec:heatkernel}, that is a dyadic BBM with branching rate $r(x)$ and drift $-\mu$, killed upon reaching $0$ and $L$. This section is aimed at controlling the first and second moments of the supermartingale $Z'_t$ introduced in Section \ref{sec:skproof}. Let us first give a precise definition of this process. Denote by $\mathcal{N}_t^L$ the set of particles alive in the BBM at time $t$ and for each particle $v\in\mathcal{N}^L_t$, denote by $X_v(t)$ its position at time $t$. Recall the definition of $v_1$ from Lemma \ref{eigenvloc} and consider the eigenvector
\begin{equation}
    w_1(x)=\sinh\left(\sqrt{2\lambda_1}(L-1)\right)v_1(x), \label{def:w1}
\end{equation}
and the function
\begin{equation}
    z(x)=e^{\mu(x-L)}w_1(x) \label{def:z}.
\end{equation}
We also define
\begin{align}
    Z'_t        & =\sum_{v\in\mathcal{N}_t^L}e^{\mu(X_v(t)-L)}w_1(X_v(t))=\sum_{v\in\mathcal{N}_t^L}z(X_v(t))\label{def:Z}, \\
    Y_t         & =\sum_{v\in\mathcal{N}_t^L}(X_v(t)\wedge 1)e^{\mu(X_v(t)-L)}\label{def:Y},                                \\
    \tilde{Y}_t & =\sum_{v\in\mathcal{N}_t^L}e^{\mu(X_v(t)-L)}.\label{def:tY}
\end{align}
This section is  divided into two parts. In Section \ref{sec:fm}, we estimate the first moments of the processes $(Z'_t)_{t>0}$, $(Y_t)_{t>0}$ and $(\tilde{Y}_t)_{t>0}$ under (\ref{hpushed}).  In Section \ref{sec:sm}, we bound the second moment of $(Z'_t)_{t>0}$ under (\ref{hwp}). The key idea to calculate these moments is to approximate the density $p_t$ by the \textit{stationary configuration} from Lemma \ref{th1} and to control the fluctuations using the Green function. We will also use the following consequence of Lemma \ref{th1}: there exists a positive constant $C>0$ (that only depends on $\rho$) such that if $L$ is sufficiently large and $t>c_{\ref{th1}}L$, then
\begin{equation}
    p_t(x,y)\leqslant C e^{\mu(x-y)} v_1(x)v_1(y).\label{ub:pt2}
\end{equation}
Before getting to the moment calculations, we state a preliminary result that will be extensively used in the remainder of the article.

\begin{lemma}\label{ubv1}
    Assume that (\ref{hpushed}) holds. There exists a constant $C>0$ (that only depends on $\rho$) such that for $L$ large enough
    \begin{equation*}
        C^{-1}(x\wedge 1\wedge(L-x)) e^{\beta(L-x)} \leqslant  w_1(x)\leqslant C(x\wedge 1\wedge(L-x)) e^{\beta(L-x)}.
    \end{equation*}
    As a consequence, there exists $C'>0$ such that, for $L$ large enough, we have
    \begin{equation*}
        (C')^{-1}(x\wedge 1\wedge (L-x))e^{-\beta x}\leqslant v_1(x)\leqslant C'(x\wedge 1\wedge (L-x))e^{-\beta x}.
    \end{equation*}
\end{lemma}

\begin{proof}
    We distinguish $3$ cases to prove the first part of the lemma. Once this is proved, the second part is a direct consequence of Remark \ref{r:vw}.

    Suppose $x\in[0,1]$. Recall from Lemma \ref{eigenvlocasymp} that for $L$ large enough, $\sqrt{\rho-1-2\lambda_1}\in\left(\frac{\pi}{2},\pi\right)$. Hence, a concavity argument combined with Remark \ref{r:vw} yields the inequality
    \begin{equation*}
        w_1(x)\geqslant x \sinh(\sqrt{2\lambda_1}(L-1)) \geqslant x e^{\beta L}\geqslant x e^{\beta (L-x)}.
    \end{equation*}
    In addition, using that $|\sin(y)|\leqslant y$ for all $y\in \mathbb{R}$, we get that
    \begin{align*}
        w_1(x) & \leqslant \frac{\sqrt{\rho-1-2\lambda_1}}{\sin(\sqrt{\rho-1-2\lambda_1})} x \sinh(\sqrt{2\lambda_1}(L-1)) \\ &\leqslant  \frac{\sqrt{\rho-1-2\lambda_1}}{\sin(\sqrt{\rho-1-2\lambda_1})} x \sinh(\sqrt{2\lambda_1}(L-x))
        \leqslant C x e^{\beta (L-x)},
    \end{align*}
    where the last line follows from Remark \ref{lb:sin}.
    Suppose $x\in[1,L-1]$. Then, $w_1(x)=\sinh(\sqrt{2\lambda_1}(L-x))$ and
    \begin{eqnarray*}
        w_1(x)e^{-\beta(L-x)}&=&\frac{1}{2}\left(e^{(\sqrt{2\lambda_1}-\beta)(L-x)}-e^{-(\sqrt{2\lambda_1}+\beta)(L-x)}\right)\\&\geqslant&\frac{1}{2}\left(e^{(\sqrt{2\lambda_1}-\beta)L}-e^{-(\sqrt{2\lambda_1}+\beta)}\right).
    \end{eqnarray*} Recall from Lemma \ref{eigenvlocasymp} that $(\sqrt{2\lambda_1}-\beta)L\to0$ as $L\to \infty$. Besides $e^{-(\sqrt{2\lambda_1}+\beta)}<e^{-\beta }<1$ so that $e^{(\sqrt{2\lambda_1}-\beta)L}-e^{-(\sqrt{2\lambda_1}+\beta)}>1/2$ for $L$ large enough. On the other hand, it follows directly from the definition of $w_1$ that
    \begin{equation*}
        w_1(x)\leqslant e^{\sqrt{2\lambda_1}(L-x)}\leqslant e^{\beta (L-x)}.
    \end{equation*}
    Suppose $x\in[L-1,L]$. Since the function $w_1$ is convexe on this interval, we have
    \begin{eqnarray*}
        w_1(x)\geqslant \sqrt{2\lambda_1}(L-x)\geqslant (\beta/2)(L-x)\geqslant (\beta/2)(L-x)e^{\beta}\geqslant (\beta/2)(L-x)e^{\beta(L-x)},
    \end{eqnarray*}
    for $L$ large enough, and
    \begin{equation*}
        w_1(x)\leqslant \sinh(\sqrt{2\lambda_1})(L-x)\leqslant (L-x)e^{\beta}\leqslant C(L-x)e^{\beta (L-x)}.
    \end{equation*}
    This concludes the proof of the lemma.
\end{proof}

\subsection{First moment estimates} \label{sec:fm}

\begin{lemma}[First moment of $Z_t'$]\label{lem:fmZ}
    Assume that (\ref{hpushed}) holds and let $t>0$. We have
    \begin{equation*}
        \mathbb{E}[Z_t']=e^{(\lambda_1-\linf)t} Z_0'.
    \end{equation*}
\end{lemma}
\begin{proof}
    The many-to-one lemma (see Lemma \ref{lem:many-to-one}) yields
    \begin{equation*}
        \mathbb{E}_x[Z_t']=\int_0^Lp_t(x,y)e^{\mu(y-L)}w_1(y)dy= e^{\mu(x-L)}e^{-\linf t}\int_0^L w_1(y)\sum_{k=1}^\infty e^{\lambda_k t}\frac{v_k(x)v_k(y)}{\|v_k\|^2}\,dy.
    \end{equation*}
    The second equality comes from Equation \eqref{def:pt1}.
    Yet, $w_1$ is a multiple of $v_1$ and  $\left(\frac{v_k}{\|v_k\|}\right)_{k\geq1}$ is an orthonormal sequence of $\mathrm{L}^2([0,L])$ (see Section \ref{sec:spec}).  Hence
    $$\int_0^L w_1(y)\left(\sum_{k=1}^\infty e^{\lambda_k t}\frac{v_k(x)v_k(y)}{\|v_k\|^2}\right)dy=e^{\lambda_1 t}w_1(x).$$ The result follows by summing over the particles at time $0$.
\end{proof}

\begin{lemma}[First moment of $\tilde{Y}_t$]\label{lemmatild} Assume that (\ref{hpushed}) holds. There exists a constant $C>0$ (that only depends on $\rho$) such that for $L$ large enough  and $t>c_{\ref{th1}} L$
    \begin{equation*}
        \mathbb{E}[\tilde{Y}_t]\leqslant Ce^{-\beta L}Z_0'.
    \end{equation*}
\end{lemma}
\begin{cor}[First moment of $Y_t$] \label{lem:fmY}
    There exists a constant $C>0$ (that only depends on $\rho$) such that for $L$ large enough and $t>c_{\ref{th1}} L$
    \begin{equation*}
        \mathbb{E}[Y_t]\leqslant Ce^{-\beta L}Z_0'.
    \end{equation*}
\end{cor}
\begin{proof}[Proof of Lemma \ref{lemmatild}]

    Combining \eqref{ub:pt2} with the many-to-one lemma, we get that, for $L$ large enough and  $t>c_{\ref{th1}}L$,
    \begin{eqnarray*}
        \mathbb{E}_x[\tilde{Y}_t]=\int_0^Le^{\mu(y-L)}p_t(x,y)dy
        \leqslant Ce^{\mu(x-L)}v_1(x) \int_0^Lv_1(y)dy.\end{eqnarray*}
    Recalling from Lemma \ref{ubv1} that  $v_1(y)\leqslant Ce^{-\beta y}$, we see that  for $L$ large enough and $t>c_{\ref{th1}}L$,
    \begin{equation*}
        \mathbb{E}_x[\tilde{Y}_t]\leqslant Ce^{\mu(x-L)}v_1(x).
    \end{equation*}
    Remark \ref{r:vw} then yields the lemma.
\end{proof}

\subsection{Second moment estimates}\label{sec:sm}

  {In this section, we bound the second moment of $Z'$ under \eqref{hwp}. In particular, we will make heavy use of the fact that $\mu>3\beta$ in the semipushed regime (see \eqref{rem:alpha}). }
\begin{lemma}[Second moment of $Z_t'$]\label{lem:smZ}
    Assume (\ref{hwp}) holds and let $u:(1,\infty)\rightarrow (0,\infty)$ be a function such that $u(L)\rightarrow \infty$  and $u(L)/L\rightarrow \infty$ as $L\rightarrow \infty$. There exists a constant $C>0$ (that only depends on $\rho$) such that for $L$ large enough
    \begin{equation*}
        \mathbb{E}_{  {x}}[(Z_u')^2]\leqslant C\left(u e^{-2\beta L}Z'_0+Y_0\right).
    \end{equation*}
\end{lemma}
\begin{proof}

    The many-to-two lemma (see Lemma \ref{lem:many-to-two})) combined with the formula for the first moment of $Z'$ calculated in Lemma \ref{lem:fmZ} yields
    \begin{eqnarray}
        \mathbb{E}_x[(Z_u')^2]&=&\mathbb{E}_x\left[\sum_{v\in\mathcal{N}_u^L}z(X_v (u))^2\right]+2\int_0^u\int_0^L r(y)p_s(x,y)\mathbb{E}_y[Z'_{t-s}]^2dy\, ds\nonumber\\
        &\leqslant&\underbrace{\mathbb{E}_x\left[\sum_{v\in\mathcal{N}_u^L}z(X_v (u))^2\right]}_{\textstyle=:T_1}+2\rho\underbrace{\int_0^u\int_0^L p_s(x,y)e^{2\mu(y-L)}w_1(y)^2dy \;ds}_{\textstyle=:T_2}.
        \label{manyto2}
    \end{eqnarray}
    Let us first bound the expectation $T_1$. By the many-to-one lemma,
    \begin{equation*}
        T_1=\int_0^Lp_u(x,y)e^{2 \mu (y-L)}w_1(y)^2dy.
    \end{equation*}
    Using Equation \eqref{ub:pt2} along with Lemma \ref{ubv1}, we get that \begin{align*}
        T_1 & \leqslant Ce^{-2\beta L}e^{\mu (x-L)}w_1(x)\int_0^L e^{\mu(y-L)}v_1^3(y)dy    \\
            & \leqslant  Ce^{-2\beta L}e^{\mu (x-L)}w_1(x)\int_0^Le^{-(\mu-3\beta)(L-y)}dy,
    \end{align*}
    as long as $u\geqslant c_{\ref{th1}}L$.    {Using that $\mu>3\beta$ under \eqref{hwp}}, we see that the last integral is bounded by a constant that only depends on $\rho$. Hence, for $L$ large enough, we have
    \begin{equation}\label{ub:T1}
        T_1\leqslant Ce^{-2\beta L}e^{\mu(x-L)}w_1(x).
    \end{equation}
    Let us now bound the double integral $T_2$. First, recall from Equation (\ref{rk:green}) that the integral of the density $p_s$ with respect to $s$ can be bounded thanks to the Green function
    \begin{equation*}
        \int_0^{u} p_s(x,y)ds \leqslant eG_{\frac{1}{u}}(x,y).
    \end{equation*}
    Fubini's theorem then gives
    \begin{align}
        T_2 & = \int_0^L\left(\int_0^{u} p_s(x,y) ds\right)e^{2\mu(y-L)}w_1(y)^2dy \nonumber
        \\&\leqslant Ce^{\mu(x-L)}\left(\psi_{\linf+\frac{1}{u}}(x)A(x)+\varphi_{\linf+\frac{1}{u}}(x)B(x)\right) \label{int:Green}
    \end{align}
    with
    \begin{equation*}
        A(x):= \left(\omega_{\linf+\frac{1}{u}}\right)^{-1}\int_0^x e^{\mu(y-L)}w_1(y)^2 \varphi_{\linf+\frac{1}{u}}(y)dy,
    \end{equation*}
    and
    \begin{equation*}
        B(x):= \left(\omega_{\linf+\frac{1}{u}}\right)^{-1}\int_x^Le^{\mu(y-L)}w_1(y)^2 \psi_{\linf+\frac{1}{u}}(y)dy.
    \end{equation*}
    We recall from Lemma  \ref{lem:alphat} (applied to $\xi=\frac{1}{u}$) and Lemma \ref{ubv1} that there exist some constants (that only depend on $\rho$) such that for $L$ large enough,
    \begin{align}
        \varphi_{\linf+\frac{1}{u}}(x)               & \leqslant C(1\wedge x)\left(\frac{1}{u}e^{\beta x}+e^{-\beta x}\right), \label{estpsia3} \\
        \psi_{\linf+\frac{1}{u}}(x)                  & \leqslant C(1\wedge(L-x))e^{\beta(L-x)},
        \label{estphia3}                                                                                                                        \\
        \left(\omega_{\linf+\frac{1}{u}}\right)^{-1} & \leqslant Cue^{-\beta L},\label{eq:wal}                                                  \\
        w_1(x)                                       & \leqslant Ce^{\beta(L-x)}.\label{eq:w1}
    \end{align}
    Equations (\ref{estpsia3}), (\ref{eq:wal}) and (\ref{eq:w1}) yield
    \begin{align}
        A(x) & \leqslant Ce^{-(\mu-\beta) L}\left(\int_0^x e^{(\mu-\beta)y}dy + u\int_0^x e^{(\mu-3\beta)y}dy\right)\nonumber \\
             & \leqslant C(1\wedge x)e^{-(\mu-\beta) L}\left(e^{(\mu-\beta)x}+u e^{(\mu-3\beta)x}\right)\nonumber             \\
             & \leqslant C(1\wedge x)e^{-(\mu-\beta) L}\left(e^{(\mu-\beta)x}+u e^{(\mu-3\beta)L}\right)\nonumber             \\
             & \leqslant C(1\wedge x)\left(e^{-(\mu-\beta)(L-x)}+u e^{-2\beta L}\right)
        \label{est:A}\end{align} since $\mu>3\beta$ under \eqref{hwp} (see \eqref{rem:alpha}).
    Similarly, Equations (\ref{estphia3}), (\ref{eq:wal}) and (\ref{eq:w1}) give that
    \begin{align}
        B(x) & \leqslant Cue^{-(\mu-2\beta)L}\int_x^Le^{(\mu-3\beta)y}dy\leqslant C(1\wedge(L-x))u e^{-(\mu-2\beta)L}e^{(\mu-3\beta)L}\nonumber \\
             & \leqslant C(1\wedge(L-x))u e^{-\beta L} \label{est:B}.
    \end{align}
    Therefore, combining Equations (\ref{estphia3}) and (\ref{est:A}) and using that $\mu>3\beta$ under \eqref{hwp}, we get that
    \begin{equation}
        \psi_{\linf+\frac{1}{u}}(x)A(x)\leqslant C(1\wedge x\wedge(L-x))\left(1+u e^{-2\beta L}e^{\beta(L-x)}\right)\label{est:A:bis},
    \end{equation}
    and it follows from Equations (\ref{estpsia3}) and (\ref{est:B}) that
    \begin{equation}
        \varphi_{\linf+\frac{1}{u}}(x)B(x)\leqslant C(1\wedge x\wedge(L-x))\left(1+u e^{-2\beta L}e^{\beta(L-x)}\right)\label{est:B:bis}.
    \end{equation}
    Yet we know from Lemma \ref{ubv1} that for $L$ large enough
    \begin{equation}
        (1\wedge x\wedge(L-x))e^{\beta(L-x)}\leqslant Cw_1(x).\label{eq:73}
    \end{equation}
    Finally, combining Equations (\ref{int:Green}), (\ref{est:A:bis}), (\ref{est:B:bis}) and (\ref{eq:73}), we get that for $L$ large enough,
    \begin{equation}
        T_2\leqslant Ce^{\mu(x-L)}\left((1\wedge x)+u e^{-2\beta L}w_1(x)\right).\label{est:T2}
    \end{equation}
    Equations (\ref{ub:T1}) and (\ref{est:T2}) yield the lemma.
\end{proof}

\section{The particles hitting the right-boundary} \label{sec:R}

Recall that we are considering a BBM with branching rate $r(x)$, drift $-\mu$ and killed upon reaching $0$.

We are now interested in the contribution of the particles that reach the level $L$. This will be the object of the two following sections. In Section \ref{sec:R}, we control the number of particles reaching this level for the first time or, equivalently, the number of particles killed at the right boundary in the BBM with branching rate $r(x)$, drift $-\mu$ and killed upon exiting $(0,L)$. In Section \ref{sec:W}, we estimate the contribution of these particles to $Z_t$.

For $0\leqslant s< t$, let $R([s,t])$ denote the number of particles absorbed at $L$ between times $s$ and $t$ for the BBM in the interval $[0,L]$. As stated in \cite[Lemma 5.7]{Maillard:2020aa},  the first and second moments  of $R([s,t])$ can be calculated from the density $p_t$. More precisely, if we denote by $w_\tau(x,y)$ the density of a Brownian motion killed upon exiting $[0,L]$ at time $\tau$ and  by $H_0$ and $H_L$ the hitting times of the boundaries $0$ and $L$, then
\begin{equation}
    \mathbb{P}_x\left(H_L\in d\tau, H_L<H_0\right)=-\frac{1}{2}\partial_y w_\tau(x,y)\big |_{y=L}d\tau. \label{eq:lawR}
\end{equation}
In words, this means that the density at time $\tau$ of the hitting time of the right boundary is equal to the heat flow of the density $w_\tau$ out of the boundary $L$ at time $\tau$. Combining this with the many-to-one lemma (see \cite{Maillard:2016uw} for a general version with stopping lines) shows that
\begin{equation}
    \mathbb{E}_x[R([s,t])]=-\frac{1}{2}\int_s^t \frac{\partial}{\partial y}p_\tau(x,y)|_{y=L}d\tau=-\frac{1}{2}e^{\mu(x-L)}\int_{s}^te^{-\linf \tau}\frac{\partial}{\partial y}{ \tilde p}_\tau(x,y)|_{y=L}d\tau.\label{E:R}
\end{equation}
Standard second moment calculations (see \cite[Theorem 4.15]{Ikeda:1969tj}) then yield that
\begin{equation}
    \mathbb{E}_x\left[R([0,u])^2\right]=\mathbb{E}_x\left[R([0,u])\right]+2\int_{s=0}^u\int_{y=0}^L r(y)p_s(x,y)\mathbb{E}_y\left[R([s,u])\right]^2dy\,ds. \label{E:R2}
\end{equation}

In Section \ref{sec:fmR}, we estimate \eqref{E:R} under assumption (\ref{hpushed}).  We first consider the case  $s>c_{\ref{th1}}L$ (see Lemma \ref{lemma:infs}) and apply a similar argument  to that used to control the heat kernel $p_t$ (see Lemma \ref{th1}).
We then bound the expected number of particles absorbed at $L$ between times 0 and $c_{\ref{th1}}L$ (see Lemma \ref{lemma:aL}) using the Green function. We  combine these two estimates in Lemma \ref{th:r}.
In Section \ref{sec:smR}, we establish an upper bound on \eqref{E:R2} under \eqref{hwp}. The idea is similar to that used to bound the second moment of $Z'$.

\subsection{First moment estimates} \label{sec:fmR}

For any measurable subset $S\subset [0,+\infty)$, define
\begin{equation*}
    I(x,S)=-\frac{1}{2}\int_{S}e^{-\linf s}\frac{\partial}{\partial y}{ \tilde p_s(x,y)}|_{y=L}ds.
\end{equation*}
We also define
\begin{equation}
    \ell(S)=\int_Se^{(\lambda_1-\linf)s}ds.
    \label{def:ls}
\end{equation}
We denote by $\text{Leb}(S)$ the Lebesgue measure of the set $S$.
Since $\lambda_1$ is increasing with respect to $L$, we have
\begin{equation}
    \ell(S)\leqslant \text{Leb}(S) \label{r:ell}.
\end{equation}
  {We recall that $C$ denotes a positive constant whose value  may change from line to line. In Sections \ref{sec:fmR} and \ref{sec:smR}, these constants only depend on $\rho$.}
\begin{lemma} \label{th:r}Assume that (\ref{hpushed}) holds.
      {There exists a constant $C>0$ such that for all $L$ large enough and $0<s<t$, we have }
    \begin{equation*}
        \left|\mathbb{E}[R([s,t])]-\ell([s,t])g(L)Z_0'\right|\leqslant C(Y_0+g(L)Z_0'),
    \end{equation*}
    where $\ell$ is defined in \eqref{def:ls} and  $g(L)= \sqrt{2\lambda_1}/(2w_1(1)^{  {2}}\|v_1\|^2).$
\end{lemma}
\begin{rem}\label{rem:g}
    According to Lemma \ref{lem:l2norm}  and Remark \ref{r:vw},  as $L\to\infty$,
    \begin{equation*}
        g(L)=\frac{1}{2}\left(\frac{\beta}{\lim\limits_{L\rightarrow\infty}\|v_1\|^2}\right)e^{-2\beta (L-1)}+o(e^{-2\beta L}).
    \end{equation*}
\end{rem}
\noindent As outlined above, the proof of this result is divided into two parts.

\begin{lemma}\label{lemma:infs}
    Assume that (\ref{hpushed}) holds.   {For all $L$ large enough and $t>s>c_{\ref{th1}}L$}
    \begin{eqnarray*}
        \left|I(x,[s,t])-g(L)\ell([s,t])w_1(x)\right|
        \leqslant  e^{-\beta L}g(L)w_1(x),   {\quad \forall x\in[0,L]},
    \end{eqnarray*}
    where $g(L)$ is  as in Lemma \ref{th:r}.
\end{lemma}

\begin{proof} The proof  is similar to the proof of Lemma \ref{th1} and mainly relies on the bounds established in Lemma \ref{lem:est:evk}.
      {First, remark that $v_1'(L)=-\sqrt{2\lambda_1}/{w_1(1)}$ so that $$-\frac{1}{2}(v_1'(L)v_1)/\|v_1\|^2=w_1(1)g(L)v_1=g(L)w_1.$$}
    Since the sum is uniformly convergent for $s>1$, we have
    \begin{multline*}
        I(x,[s,t])-g(L)\ell([s,t])w_1(x)\\=\underbrace{\sum_{k=2}^K -\frac{1}{2}\frac{v_k(x)v_k'(L)}{\|v_k\|^2}\int_{s}^te^{(\lambda_k-\linf)u}du}_{\textstyle{=:U_1}}
        +\underbrace{\sum_{i=K+1}^\infty -\frac{1}{2}\frac{v_k(x)v_k'(L)}{\|v_k\|^2}\int_{s}^te^{(\lambda_k-\linf)u}du}_{\textstyle=:U_2}.
    \end{multline*}
    Note that for all $k\geqslant 2$
    \begin{equation}\int_s^te^{(\lambda_k-\linf)u}du\leq  \frac{e^{(\lambda_k-\linf)s}}{\linf-\lambda_k},
        \label{dis:e}
    \end{equation}
    and recall from Lemma \ref{eigenvloc}, \ref{ev:sc} and \ref{eigenvlocasymp} that
    \begin{equation}\label{ub:linflk}
        \linf-\lambda_{k}\geqslant \begin{cases}\frac{5}{8}\pi^2 & \text{if} \quad 2\leqslant k \leqslant K \\
             \linf            & \text{if} \quad  k\geqslant K+1.\end{cases}
    \end{equation}
    We first bound $U_1$. Lemma \ref{lem:est:evk} (i)-(ii) implies that for $L$ large enough
    \begin{equation}
          {\frac{\left|v_k(x)\right|}{\|v_k\|^2}\leqslant  C_1^{-1/2}C_3 e^{\beta L}v_1(x) \quad  \quad \forall \ k\in \llbracket 2, K\rrbracket, \ \forall x\in[0,L].}\label{eq1:vkk}
    \end{equation}
      {Moreover, a direct calculation shows that} $v_k'(L)$ tends to $0$ as $L\to \infty$ for all $k\in\llbracket 2, K\rrbracket$.
    Putting this together with \eqref{ub:linflk}, \eqref{eq1:vkk}, Remark \ref{r:vw} and Remark \ref{rem:g} and summing over $k\in \llbracket 2, K\rrbracket$, we get that for $L$ large enough
    \begin{equation*}
        \left| U_1\right|\leqslant C e^{2\beta L-\frac{5}{8}\pi^2 s}g(L)w_1(x), \quad \forall x\in[0,L].
        \label{dis:t1}
    \end{equation*}
    It now remains to bound $U_2$. Let $k>K$. We know from  Lemma \ref{lem:est:evk} (ii) that  for $L$ large enough,\begin{equation}
        \frac{|v_k(x)|}{\|v_k\|}\leqslant C_4\sqrt{\rho-1-2\lambda_k}\,e^{{\beta}L}v_1(x), \quad \forall x\in[0,L].
        \label{eq:fR1}
    \end{equation}
    On the other hand, we have $$|v_k'(L)|=(\sqrt{-2\lambda_k})/|\sin(\sqrt{-2\lambda_k}(L-1)|\leqslant (\sqrt{\rho-1-2\lambda_k})/|\sin(\sqrt{-2\lambda_k}(L-1))|.$$ Putting this together with Lemma \ref{lem:est:evk} (i), we see that
    \begin{equation*}
        \frac{|v_k'(L)|}{\|v_k\|}\leqslant C_2^{-1}\sqrt{\rho-1-2\lambda_k}.
    \end{equation*}
    Combining this inequality with
    (\ref{eq:fR1}), using Remarks \ref{r:vw} and \ref{rem:g} and summing over $k>K$, we get that for $L$ large enough
    \begin{equation}
        \left|U_2\right|\leqslant Ce^{2{\beta} L}\left[\sum_{k=K+1}^\infty(\rho-1-2\lambda_k)\frac{e^{(\lambda_k-\linf) s }}{\linf-\lambda_k}\right]g(L)w_1(x), \quad \forall x\in[0,L]. \label{eq;fR3}
    \end{equation}
    We then recall from Equation \eqref{ub:S3} that the sum
    $\sum_{k>K} (\rho-1-2\lambda_{k})e^{(\lambda_k-\lambda_{K+1}) s }$
    is bounded by $CL$ for all $s >1$. Using Equations (\ref{ub:linflk}) and \eqref{eq;fR3}, we see that for  $L$ large enough
    \begin{eqnarray*}
        \left|U_2\right|
        \leqslant CLe^{2{\beta}L-\linf  s }g(L)w_1(x), \quad \forall x\in[0,L].
        \label{dist:t2}
    \end{eqnarray*}
    The result then follows by comparing $U_1$ and $U_2$ with the quantities $S_1$ and $S_2$ in the proof of Lemma \ref{th1}.
\end{proof}

\begin{lemma} \label{lemma:aL}Assume that (\ref{hpushed}) holds. There exists a constant $C>0$ (that only depends on $\rho$) such that for $L$ large enough,
    \begin{equation*}
        I(x,[0,c_{\ref{th1}}L])\leqslant C(1\wedge x), \quad \forall x\in[0,L].
    \end{equation*}
\end{lemma}
\begin{proof} Let $\xi=\frac{(c_{\ref{th1}})^{-1}}{L}$. Using a bound similar to (\ref{rk:green}), we see that
    \begin{equation*}
        I(x,[0,c_{\ref{th1}}L])\leqslant e \int_0^\infty  e^{-(\linf+\xi) s}\left (-\frac{\partial}{\partial y}{ \tilde p_s(x,y)}|_{y=L}\right)ds.
    \end{equation*} Interchanging the partial derivative and the integral in the definition of the Green function, we get
    \begin{equation*}
        \frac{\partial}{\partial y}G_{\xi}(x,y)|_{y=L}= e^{\mu(x-L)}\int_0^\infty  e^{-(\linf+\xi) s}\left (-\frac{\partial}{\partial y}{ \tilde p_s(x,y)}|_{y=L}\right)ds,
    \end{equation*} so that
    \begin{equation*}
        I(x,[0,c_{\ref{th1}}L]) \leqslant Ce^{\mu(L-x)}\left(-\frac{\partial}{\partial y} G_{\xi}(x,y)|_{y=L}\right), \quad \forall x\in[0,L].
    \end{equation*} We then deduce from  (\ref{greenfun}) that for all $x\in[0,L]$
    \begin{equation*}
        \frac{\partial}{\partial y}G_{\xi}(x,y)_{|_{y=L}}=e^{\mu(x-L)}(\omega_{\linf+\xi})^{-1}\psi_{\linf+\xi}'(L)\varphi_{\linf+\xi}(x).
    \end{equation*}
    The definition of $\psi_{\linf+\xi}$ and Lemma \ref{lemma:wro} applied to $h=\frac{1}{c_{\ref{th1}}}$ implies that, for $L$ large enough,
    \begin{align*}
        -\psi_{\linf+\xi}'(L)                & =\sqrt{2{\linf+\xi}}\leqslant C,                                                         \\
        \left(\omega_{\linf+\xi}\right)^{-1} & \leqslant CLe^{-{\beta}L},                                                               \\
        \varphi_{\linf+\xi}(x)               & \leqslant C(1\wedge x)\left(\frac{1}{L}e^{\beta x}+e^{-\beta x}\right), \quad x\in[0,L].
    \end{align*}
    Putting all of this together, we see that for $L$ large enough
    \begin{equation*}
        I(x,[0,c_{\ref{th1}}L])\leqslant C(1\wedge x)\left(e^{\beta (x-L)}+Le^{-\beta L}\right)\leqslant C(1\wedge x).
    \end{equation*}
\end{proof}

\begin{lemma} \label{th:I} Assume that (\ref{hpushed}) holds. There exists a constant $C>0$ (that only depends on $\rho$) such that for all $0\leqslant s<t$ and $L$ large enough, we have \begin{equation*}
        |I(x,[s,t])-g(L)\ell([s,t])w_1(x)|\leqslant C((1\wedge x)+ g(L)w_1(x)),\quad  \forall x\in[0,L].
    \end{equation*}
\end{lemma}
\begin{proof}   {Note that the result is a direct consequence of Lemma \ref{lemma:infs} when $s\geq c_{\ref{th1}}$L.}

    We now assume that $s< c_{\ref{th1}}L$. By definition of $I$, $I(x,[s,t])=I(x,[s,c_{\ref{th1}}L])+I(x,[c_{\ref{th1}}L,t]).$ Thus the triangle inequality yields
    \begin{align}
                  & I(x,[s,t])-g(L)\ell([s,t])w_1(x)|\nonumber                                                                  \\
        \nonumber & \leqslant |I(x,[s,c_{\ref{th1}}L])|  +|I(x,[c_{\ref{th1}}L,t])-g(L)\ell([c_{\ref{th1}}L,t])w_1(x)|\nonumber \\
                  & \quad +g(L)w_1(x)|\ell([c_{\ref{th1}}L,t])-\ell([s,t])|,\label{eq:ITI}
    \end{align}
    for all $x\in[0,L]$.
    Using Lemma \ref{lemma:infs}, we get that the second term on the RHS of \eqref{eq:ITI} is bounded by
    \begin{equation*}
        |I(x,[c_{\ref{th1}}L,t])-g(L)\ell([c_{\ref{th1}}L,t])w_1(x)|\leqslant Cg(L)w_1(x), \quad \forall x\in[0,L],
    \end{equation*}
    for $L$ large enough.
    Besides, the first summand on the RHS of \eqref{eq:ITI} is upper bounded by $I(x,[0,c_{\ref{th1}}L])$ and we know thanks to Lemma \ref{lemma:aL} that
    $$|I(x,[0,c_{\ref{th1}}L])|\leqslant C(1\wedge x),$$
    for $L$ large enough.
    Finally, by definition of $\ell$ (see  \eqref{def:ls} and (\ref{r:ell})), the last term on the RHS of \eqref{eq:ITI} can be written as
    \begin{equation*}
        g(L)w_1(x)|\ell([c_{\ref{th1}}L,t])-\ell([s,t])|= \ell([s,c_{\ref{th1}}L])g(L)w_1(x)\leqslant c_{\ref{th1}}Lw_1(x)g(L).
    \end{equation*}
    According to Lemma \ref{ubv1} and Remark \ref{rem:g}, we know that
    $$Lw_1(x)g(L)\leqslant  C(1\wedge x), \quad \forall x\in[0,L],$$
    for $L$ large enough, which concludes the proof of the lemma.
\end{proof}

\begin{proof}[Proof of Lemma \ref{th:r}] The lemma follows directly from Equation (\ref{E:R}) and Lemma \ref{th:I}.
\end{proof}

\subsection{Second moment estimates}\label{sec:smR}
\begin{lemma}\label{lem:sdmR}
    Assume that (\ref{hwp}) holds and let $u:(1,\infty)\rightarrow (0,\infty)$ be a function such that $u(L)\rightarrow \infty$  and $u(L)/L\rightarrow \infty$ as $L\rightarrow \infty$. There exists a constant $C>0$ (that only depends on $\rho$) such that for $L$ large enough
    \begin{equation*}
        \mathbb{E}_{  x}\left[R([0,u])^2\right]-\mathbb{E}_{  x}\left[R([0,u])\right]\leqslant C\left(1+g(L)^2u^2\right)(Y_0+u e^{-2\beta L}Z_0'), \quad \forall x\in[0,L].
    \end{equation*}
\end{lemma}
\begin{proof}
    Since  $r(x)\leqslant \rho/2$ for all $x\in[0,L]$, we see from   \eqref{E:R2}  that
    \begin{align*}
        \mathbb{E}_x\left[R([0,u])^2\right]\leqslant \mathbb{E}_x\left[R([0,u])\right]+\rho\int_{y=0}^L\int_{s=0}^up_s(x,y)\mathbb{E}_y\left[R([s,u])\right]^2ds dy.
    \end{align*}
    Moreover, we know from Lemma \ref{th:r} and Equation (\ref{r:ell}) that for all $s\in[0,u]$ and $L$ large enough,
    \begin{equation*}
        \mathbb{E}_y\left[R([s,u])\right]\leqslant Ce^{\mu(y-L)}((1\wedge y)\,+\,g(L)(1+(u-s))w_1(y)).
    \end{equation*}
    Then, using that $(a+b)^2\leqslant 2a^2+2b^2$ for all $(a,b)\in\mathbb{R}^2$, $(1\wedge x)^2\leqslant 1$ for all $x\in[0,L]$ and that $1+(u-s)^2\leqslant 1+u^2\leqslant 2u^2$ for all $u\geqslant 1$,  we obtain that for $L$ large enough,
    \begin{align}
         & \int_{y=0}^L\int_{s=0}^up_s(x,y)\mathbb{E}_y\left[R([s,u])\right]^2dsdy\label{eq:smR}                                                                                                                                 \\
         & \leqslant C\left(\underbrace{\int_0^Le^{2\mu(y-L)}\int_{s=0}^u p_s(x,y)ds dy}_{\textstyle=:V_1}\,+\,g(L)^2u^2\underbrace{\int_0^Le^{2\mu(y-L)}w_1(y)^2\int_{s=0}^u p_s(x,y)ds dy}_{\textstyle=:V_2}\right). \nonumber
    \end{align}
    As in the proof of Lemma \ref{lem:smZ}, we can bound the integrals $V_1$ and $V_2$ using the estimates on the Green function established in Lemma \ref{lem:alphat}. Note that $V_2$ has already been estimated in the proof of Lemma \ref{lem:smZ} (it corresponds to the integral $T_2$, see Equations \eqref{int:Green} and \eqref{est:T2}): we know that for $L$ large enough,
    \begin{equation}
        V_2\leqslant  Ce^{\mu(x-L)}\left((1\wedge x)+u^{-2\beta L}w_1(x)\right).\label{est:U2}
    \end{equation}
    We then bound $V_1$ using  \eqref{rk:green} and (\ref{greenfun}). We obtain that for $L$ large enough
    \begin{align}
        V_1 & =\int_0^Le^{2\mu(y-L)}\int_{s=0}^u p_s(x,y)dsdy \leqslant C(D(x)+E(x)),
        \label{B1B2}\end{align}
    with
    \begin{equation*}
        D(x)= e^{\mu(x-L)}\left(\omega_{\linf+\frac{1}{u}} \right)^{-1}\psi_{\linf+\frac{1}{u}}(x)\int_0^x e^{\mu(y-L)}\varphi_{\linf+\frac{1}{u}}(y)\,dy,
    \end{equation*}
    and
    \begin{equation*}
        E(x)= e^{\mu(x-L)}\left(\omega_{\linf+\frac{1}{u}}\right)^{-1}\varphi_{\linf+\frac{1}{u}}(x)\int_x^L e^{\mu(y-L)}\psi_{\linf+\frac{1}{u}}(y)\,dy.
    \end{equation*}
    Following the proof of Lemma \ref{lem:smZ}, we use Equations (\ref{estpsia3}), (\ref{estphia3}) and (\ref{eq:wal}) to bound the quantities $D(x)$ and $E(x)$. We obtain that for $L$ large enough
    \begin{align*}
        D(x) & \leqslant  C(1\wedge (L-x))e^{\mu(x-L)} e^{-\beta x}\int_0^x e^{\mu(y-L)}\left(e^{\beta y}+u e^{-\beta y}\right)dy                   \\
             & \leqslant C(1\wedge (L-x))e^{\mu(x-L)} e^{-\beta x}e^{-\mu L}\left(\int_0^x e^{(\mu+\beta)y}dy+ u\int_0^x e^{(\mu-\beta)y }dy\right) \\
             & \leqslant C(x\wedge1\wedge (L-x))e^{\mu(x-L)} e^{-\beta x}e^{-\mu L}\left(e^{(\mu+\beta)x}+ u e^{(\mu-\beta)x}\right)                \\
             & \leqslant  C(x\wedge 1\wedge (L-x))e^{\mu(x-L)}\left(e^{\mu(x-L)}+u e^{(\mu-2\beta)x}e^{-\mu L}\right)                               \\
             & \leqslant  C(x\wedge 1\wedge (L-x))e^{\mu(x-L)}\left(1+u e^{-2\beta L}\right)                                                        \\
             & \leqslant  C(x\wedge 1\wedge (L-x))e^{\mu(x-L)}\left(1+u e^{-2\beta L}e^{\beta (L-x)}\right),
    \end{align*}
    where the two last inequalities come from the fact that $(\mu-2\beta)x\leqslant (\mu-2\beta)L$ (since $\mu>3\beta>2\beta$ under \eqref{hpushed}, see \eqref{rem:alpha}) and $e^{\beta(L-x)}\geq 1$ for all $x\in[0,L]$. Similarly, we get that for $L$ large enough,
    \begin{align*}
        E(x) & \leqslant  C(1\wedge x)e^{\mu(x-L)}e^{-\beta L}\left(e^{\beta x}+u e^{-\beta x}\right)\int_x^L e^{(\mu-\beta)(y-L)}dy  \\
             & \leqslant  C(1\wedge x)e^{\mu(x-L)}\left(e^{\beta x}+u e^{-\beta x}\right)e^{-\beta L}\int_0^{L-x} e^{-(\mu-\beta)z}dz \\
             & \leqslant C(1\wedge x\wedge(L-x))e^{\mu(x-L)}e^{-\beta L}\left(e^{\beta x}+u e^{-\beta x}\right)                       \\
             & \leqslant  C(1\wedge x\wedge(L-x))e^{\mu(x-L)}\left(1+u e^{-2\beta L}e^{\beta(L-x)}\right).
    \end{align*}
    Using Lemma \ref{ubv1} we then see  that for $L$ large enough,
    \begin{eqnarray*}
        D(x)+E(x)\leqslant e^{\mu(x-L)}((1\wedge x)+u e^{-2\beta L}w_1(x)).
    \end{eqnarray*}
    Putting this together with Equations (\ref{eq:smR}), (\ref{B1B2}), and (\ref{est:U2}), we get that for $L$ large enough,
    \begin{equation*}
        \mathbb{E}_x\left[R([0,u])^2\right]\leqslant \mathbb{E}_x\left[R([0,u])\right]+ Ce^{\mu(x-L)}(1+g(L)^2u^2)((1\wedge x)+u e^{-2\beta L}w_1(x)),
    \end{equation*}
    which concludes the proof of the lemma.
\end{proof}

\section{Descendants of a single particle} \label{sec:W}
In this section, we estimate the number of descendants of one particle at $L$. As outlined in Section \ref{sec:skproof}, the proof is based on \cite[Section 4]{Berestycki2010}. We start the process with a single particle at $L$ and stop its descendants when they reach the level $L-y$, for some large constant $y>0$. We denote by $Z_y$  the total number of particles stopped at $L-y$. We will be interested in the large-$y$ behaviour of the random variable $Z_y$.
Equivalently, one can consider a dyadic BBM with branching rate $\frac{1}{2}$ and drift $-\mu$, starting with a single particle at $0$ and absorbed at $-y$. Indeed, for sufficiently large $L$, the level $L-y$ stays above $1$. As a consequence, $Z_y$ has the same distribution as the number of particles killed at $-y$ in this process. Besides, since $\mu>1$ under (\ref{hpushed}), this BBM with absorption almost surely goes extinct \cite{KESTEN19789} so that $Z_y$ is finite almost surely. In addition, it was shown \cite{Neveu:1988aa} that the process $(Z_y)_{y\geqslant 0}$ is a supercritical continuous time branching process. Note that this differs from \cite{Berestycki2010}: in the case $\rho=1$,  the drift $\mu$ is equal to $1$ (which corresponds to our pulled regime) so that $Z_y$ is critical. This is the reason why this case requires  a control on the \textit{derivative martingale} associated to the BBM. In our case ($\mu>1$), it will be sufficient to consider a certain \textit{additive martingale} to construct a travelling wave (see Equation \eqref{K:equation} below).

In Lemma \ref{lem:Zy}, we prove that $e^{-(\mu-\beta)y}Z_y$ converges to a random variable $W$ such that $P(W>x)$ is proportional  $1/x^\alpha$. This result follows from  the uniqueness of the travelling wave solutions of Kolmogorov's equation and Karamata's Tauberian theorem (see Theorem 8.1.6 of \cite{Bingham:1989aa}).

\begin{lemma} \label{lem:Zy} Assume that (\ref{hwp}) holds.
    There exists a random variable $W$ such that almost surely
    \begin{equation}
        \lim_{y\to+\infty}e^{-(\mu-\beta)y}Z_y= W. \label{def:W}
    \end{equation}
    Besides, for all $u\in\mathbb{R}$, we have $\mathbb{E}\left[\exp\left(-e^{- (\mu-\beta)  u}W\right)\right]=\phi(u)$, where $\phi:\mathbb{R}\rightarrow(0,1)$ solves Kolmogorov's equation
    \begin{equation}
        \frac{1}{2}\phi''  {+}\mu\phi'=\frac{1}{2}\phi(1-\phi),\label{K:equation}
    \end{equation}
    with $\lim_{u\to-\infty}\phi(u)=0$ and $\lim_{u\to+\infty}\phi(u)=1$. In addition, there exists $b_{\ref{lem:Zy}}>0$ such that, as ${  q}\to 0$, we have
    \begin{equation}
        \mathbb{E}\left[e^{-{  q} W}\right]=\exp\left(-{  q}+b_{\ref{lem:Zy}}{  q}^\alpha +o({  q}^\alpha)\right), \label{TE:LTW}
    \end{equation}
    where $\alpha$ is given by \eqref{defalpha:rho}.
\end{lemma}
\begin{proof} The first part of the lemma (Equations \eqref{def:W} and \eqref{K:equation}) is a consequence of the uniqueness (up to a multiplicative constant) of the travelling wave solutions of \eqref{K:equation}.

    Following \cite{Neveu:1988aa}, we consider a  dyadic BBM with branching rate $\frac{1}{2}$ and no drift (and no killing). For this specific BBM, which will only be studied in this section, we also denote by $\mathcal{N}_t$ the set of individuals alive at time $t$ and for each particle $v\in \mathcal{N}_t$,  we denote by $X_v(t)$ its position at time $t$. Note that $Z_y$ has the same distribution as the number of \textit{first crossings} of the line $y=\mu t$ studied in \cite{Neveu:1988aa}. Let $\underline{{  q}}=\mu-\beta$ and define
    \begin{equation*}
        W_t(\underline{{  q}})=\sum_{v\in\mathcal{N}_t}e^{-\underline{{  q}}(X_v(t)+\mu t)}.
    \end{equation*}
    This process is a positive martingale \cite{Neveu:1988aa}. Hence, it converges almost surely. We denote by $W(\underline{{  q}})$ its limit. It was shown \cite{Neveu:1988aa} that this convergence also holds in $L^1$ and that
    \begin{equation*}
        \mathbb{E}\left[\exp\left(-e^{-  \underline{q} u}W(\underline{{  q}})\right)\right]=\phi(u),
    \end{equation*}
    where $\phi$   {is as in} \eqref{K:equation}.
    In addition, it was proved (see e.g. \cite[Theorem 8]{Kyprianou:2004aa}) that $e^{-\underline{{  q}}y}Z_y$ is also a martingale that converges almost surely and in $\mathrm{L}^1$ to $W(\underline{  q})$. In particular $\mathbb{E}[W(\underline q)]=1.$

    The second part of the lemma (Equation \eqref{TE:LTW}) concerns the tails of the limiting quantity $W(\underline q)$ and follows from \cite[Theorem 2.2]{Liu:2000aa}   combined with Karamata's Tauberian theorem. Consider the BBM introduced above at discrete times. This defines a branching random walk and one can consider the associated additive martingales $W_n({  q})$ defined as

    \begin{equation*}
        W_n({  q})=\sum_{v\in\mathcal{N}_n}e^{-{  q} X_v(n)-n\varphi({  q})},
    \end{equation*}
    with $$\varphi({  q})=\log\mathbb{E}\left[\sum_{v\in \mathcal{N}_1}e^{-{  q} X_v(1)}\right]=\frac{{  q}^2}{2}+\frac{1}{2},$$  for all ${  q}\in \mathbb{R}$. The additive martingale $W_n({  q})$ is positive so that it converges a.s.~to a limit $W({  q})\geq0$ as $n\to\infty$.
    Theorem 2.2 of \cite{Liu:2000aa} states that, if for some $p>1$,
    \begin{equation}
        \varphi(p{  q})=p\varphi({  q}) \label{cond:p},
    \end{equation}
    and
    \begin{multline}
        E_1:=\mathbb{E}\left[\sum_{v\in \mathcal{N}_1}\left(e^{-{  q} X_v-\varphi({  q})}\right)^p\log^+\left(e^{-{  q} X_v-\varphi({  q})}\right)\right]<\infty,\\ E_2:=\mathbb{E}\left[\left(\sum_{v\in \mathcal{N}_1} e^{-{  q} X_v-\varphi({  q})}\right)^p\right]<\infty, \label{cond:moment}
    \end{multline}
    then there exists $l>0$ such that
    \begin{equation}
        \mathbb{P}\left(W({  q})>x\right)\sim \frac{l}{x^p}, \quad x\to\infty. \label{tail:W}
    \end{equation}
    Note that the condition \eqref{cond:p} and the definition of the function $\varphi$ implies that
    \begin{equation*}
        p=\frac{1}{{  q}^2}.
    \end{equation*}
    For ${  q}= \underline{{  q}}$, we get $p=1/\underline{{  q}}^2$. Remark that $\underline{{  q}}$ is the smallest root of $\frac{{  q}^2}{2}-\mu {  q} +\frac{1}{2}$ and that the second root of this polynomial is given by $\bar{{  q}}=\mu+\beta$. Hence,  $\underline{{  q}}\bar{{  q}}=1$ and $$p=\frac{1}{\underline {  q} ^2}=\frac{\bar{{  q}}}{\underline{{  q}}}=\frac{\mu+\beta}{\mu-\beta}=\alpha\in(1,2).$$
      {Assume for a moment that \eqref{cond:moment} holds for this choice of $p$ and $q$}.
    We then deduce from Equation (\ref{tail:W}) and Karamata's Tauberian theorem \cite[Theorem 8.1.6]{Bingham:1989aa}  that
    \begin{equation*}
        \mathbb{E}\left[e^{-{  q} W(\underline{{  q}})}\right]=1-\, {  q}+b{  q}^\alpha+o({  q}^\alpha),\quad{  q}\to 0,
    \end{equation*}
    with $b=-l\Gamma\left(-(\alpha-1)\right)>0,$ where $\Gamma$ refers to the gamma function. Finally, since  $\alpha \in(1,2)$ under \eqref{hwp}, we have
    \begin{equation*}
        \mathbb{E}\left[e^{-{  q} W}\right]=\exp\left(-\,{  q}+b{  q}^\alpha+o({  q}^\alpha)\right), \quad {  q}\to0.
    \end{equation*}
      {It now remains to prove that \eqref{cond:moment} holds for $q=\underline q$ and $p=\alpha$. To do so, we use our many-to-one (see Lemma \ref{lem:many-to-one}) and many-to-two (see Lemma \ref{lem:many-to-two}) formulae. Let $P_t(x,y)$ be the heat kernel associated to the dyadic BBM with branching rate $\frac{1}{2}$. This density $P$ can be expressed as $P_t(x,y)=e^{\frac{1}{2}t}v_t(x,y)$, where $v_t$ denotes the density of a standard Brownian motion at time $t$.}
    We first bound $E_1$. Note that
    $\log^+(e^{-x})\leq e^{-x}$ for all $x\in\mathbb{R}$. Applying the many-to-one formula to $f(y):=e^{-(p+1)\underline q y- (p+1) \varphi(\underline  q)}$, we get that
    \begin{align*}
        E_1 & \leq \mathbb{E}_0\left[\sum_{v\in \mathcal{N}_1} f(X_v(1))\right]=\int_{-\infty}^{+\infty}f(y)P_1(0,y)dy \\ &\leq e^{\frac{1}{2}-(p+1) \varphi(\underline q)}\int_{-\infty}^{\infty}e^{-(p+1)\underline q y}v_1(0,y)dy
        \leq e^{\frac{1}{2}-(p+1) \varphi(\underline q)+\frac{(p +1)^2\underline q^2}{2}}.
    \end{align*}
    We now move to the second part of \eqref{cond:moment}. Since $p<2$ under \eqref{hwp}, it is sufficient to check that the bound holds for $p=2$. The many-to-two lemma applied to $f(y)=e^{-\underline{q} y-\varphi(\underline{q})}$ entails
    \begin{align*}
        \mathbb{E}_0\left[\left(\sum_{v\in\mathcal{N}_1} f(X_v(1))\right)^2\right]= & \int_{-\infty}^{+\infty} f(y)^2P_1(0,y)dy                                                                   \\
                                                                                    & +\int_0^1\int_{-\infty}^{+\infty}P_s(0,y)\left(\int_{-\infty}^{+\infty}f(z)P_{1-s}(y,z)dz\right)^2 dy \ ds.
    \end{align*}
    A direct calculation (using the explicit form of $P_t$) shows that the first term on the RHS of the above is finite. Using  that $W_t(\underline q)$ is a martingale for the BBM, we get that
    \begin{equation*}
        \int_{-\infty}^{+\infty} f(z)P_{1-s}(y,z)dz= e^{-\varphi(\underline q)s}\int_{-\infty}^{+\infty} e^{-\underline q z -\varphi(\underline q)(1-s)}P_{1-s}(y,z)dz= e^{-\underline q y -\varphi(\underline q)s}.
    \end{equation*}
    Hence
    \begin{align*}
         & \int_0^1\int_{-\infty}^{+\infty}P_s(0,y)\left(\int_{-\infty}^{+\infty}f(z)P_{1-s}(y,z)dz\right)^2 dy \ ds                                                                                        \\
         & = \int_0^1e^{\frac{1}{2}s-2\varphi(\underline q)s}\int_{-\infty}^{+\infty}e^{-2\underline {q} y}v_s(0,y)dy\ ds=\int_0^1 e^{\frac{1}{2}s-2\varphi(\underline q)s+2s^2\underline{q}^2} ds <\infty.
    \end{align*}
    This concludes the proof of the result.
\end{proof}

\section{Convergence to the CSBP: small time steps}\label{sec:smallts}

This section is devoted to the proof of Proposition \ref{prop7} (see Section \ref{sec:notation6} below). Following \cite{Maillard:2020aa}, we prove that after a short time (on the time scale of the CSBP), the Laplace transform of the process $Z_t$ is close to that of an $\alpha$-stable CSBP. As in \cite{Maillard:2020aa} and \cite{Berestycki2010} in the case $\rho=1$, we will decompose the set of particles into two subsets: the particles that reach the level $L$ and those who stay below $L$ at all time. We will then control their respective contributions using the estimates established in Sections \ref{sec:moment:est}, \ref{sec:R} and \ref{sec:W}.

\subsection{Notation and result}\label{sec:notation6}
Before getting to the result, we recall the definition of some quantities introduced in the previous sections and define several new constants that will be used in the remainder of the paper.
From now, we consider the dyadic BBM with absorption at $0$, branching rate $r(x)$ and drift $-\mu$. Recall that $\mathcal{N}_t$ denotes the set of particles alive at time $t$, i.e.~that have not been absorbed at the origin. In this framework,  $\mathcal{N}_t^{L}$ will refer to the set of particles whose ancestors stayed below $L$ until time $t$. Recall that for each particle $v\in\mathcal{N}_t$, $X_v(t)$ denotes its position at time $t$. We also define
\begin{equation}
    N_t=|\mathcal{N}_t|, \quad N'_t=|\mathcal{N}^L_t| \quad \text{and} \quad M(t)=\max\{X_v(t), \; v\in\mathcal{N}_t\}.\label{def:Nt}
\end{equation}
Recall the definitions of the processes $Z_t$ and $Z_t'$ from Section \ref{sec:skproof}:
\begin{align}
    Z_t=\sum_{v\in \mathcal{N}_t}z(X_v(t))\mathbf{1}_{X_v(t)\in[0,L]}, \quad
    Z_t^{'}=\sum_{v\in \mathcal{N}^{L}_t}z(X_v(t)),\label{def:Zt}
\end{align}
where $z$ is the function from \eqref{def:z}. We also recall from Section \ref{sec:moment:est}, the definitions of the processes $Y$ and $\tilde Y,$
\begin{align}
    Y_t=\sum_{v\in \mathcal{N}^{L}_t}(X_v(t)\wedge 1)e^{\mu(X_v(t)-{L})},\quad \tilde Y_t^{}=\sum_{v\in \mathcal{N}^{L}_t} e^{\mu(X_v(t)-{L})}.\label{def:Yt}
\end{align}
We  consider the variable $R([s,t])$ from Section \ref{sec:R}, which counts the number of particles that hit $L$ (for the first time) between times $s$ and $t$.
The notation $\mathbb{P}_{(x,t)}$ and $\mathbb{E}_{(x,t)}$  refer to the probabilities and the expectations for the BBM when we start the process at time $t$ with a single particle at $x$. We denote by $(\mathcal{F}_t,t\geqslant 0)$ the natural filtration of the BBM.

In what follows, we will need to consider a quantity $A$ that goes slowly to infinity as $L$ tends to infinity. In other words, we first let $L\rightarrow \infty$, then $A\rightarrow \infty$ and we will consider the following notation:
\begin{itemize}
    \item $\vep_L$ is a quantity that is bounded in absolute value by a function $h(A,L)$ such that $$\forall A\geqslant 1:\lim_{L\rightarrow \infty}h(A,L)=0,$$
    \item $\vep_{A,L}$ is a quantity that is bounded in absolute value by a function $h(A,L)$ such that $$\lim_{A\rightarrow \infty}\limsup_{L\rightarrow \infty}h(A,L)=0.$$
\end{itemize}
Furthermore, we will consider a function $\bar{\theta}:(0,\infty)\rightarrow (0,\infty)$ satisfying
\begin{equation}
    \bar{\theta}(A)e^{4\beta A}\rightarrow 0, \quad A\to\infty, \label{cond:theta}
\end{equation}
and fix a constant $\Lambda>0$.

We will use the symbol $O(\cdot)$ to denote a quantity that is bounded in absolute value by a constant times the quantity inside the parenthesis. We will also use the letter $C$ to refer to a constant whose value may change from line to line as in the previous sections. In both cases,  the constants may only depend on $\rho$, $\Lambda$ and $\bar{\theta}$. We also assume that the functions $h$ defined above only depend on $\rho$, $\Lambda$ and $\bar{\theta}$.

Fix a time $t>0$ and consider $\theta \in (0,\bar{\theta}(A))$ such that $t\left(\theta e^{2\beta A}\right)^{-1}\in \mathbb{N}$. Set $  {\kappa}\in \mathbb{N}$ such that $t =  \kappa\theta e^{2\beta A}$  and define a  subdivision $(t_k)_{k=1}^{  {\kappa}}$ of the interval $[0,t e^{2\beta L}]$ defined by
\begin{equation}
    t_k=k\theta e^{2\beta (L+A)}\label{disc:t}.
\end{equation}
In words, we consider time steps of length $\theta e^{2\beta A}$ on the time scale of the CSBP, namely $e^{2\beta L}$. We  now recall two asymptotic expansions computed in Sections \ref{sec:spec} and \ref{sec:fmR} that will be needed in the proof of this result.
Let
\begin{equation*}
    a:=\frac{\beta}{\lim\limits_{L\rightarrow\infty}\|v_1\|^2}.
\end{equation*}We know from Corollary \ref{exp:lambda1} that
\begin{equation}
    \lambda_1-\linf=-a e^{-2\beta(L-1)}+o(e^{-2\beta L}),\label{DL}
\end{equation}
and that the quantity $g(L)$ defined in Lemma \ref{th:r} is such that (see Remark \ref{rem:g})
\begin{equation}
    g(L)=2a(1+\vep_L)e^{-2\beta(L-1)}. \label{eq:DLg}
\end{equation}
Finally, we recall from Equation \eqref{r:ell} that for all $0\leqslant s <t$,
\begin{equation}
    \ell([s,t]) \leqslant t-s.\label{rk:l}
\end{equation}

Our goal  in this section is to prove the following result, which is a variation of  \cite[ Proposition 7.1]{Maillard:2020aa} in the case $\rho=1$.
\begin{proposition}\label{prop7}
    Assume that (\ref{hwp}) holds.  Set $b_{\ref{prop7}}= 2^{1-\alpha}\frac{\beta}{\lim\limits_{L\rightarrow\infty}\|v_1\|^2}b_{\ref{lem:Zy}}$. Uniformly in ${  q}\in[0,\Lambda]$, on the event $\{\forall v\in \mathcal{N}_{t_k}, X_v(t_k)\leqslant L\},$
    \begin{multline*}
        \mathbb{E}\left[e^{-{  q} e^{-(\mu-\beta)A} Z^{}_{t_{k+1}}}|\mathcal{F}_{t_k}\right]\\=\exp\left(\left(-{  q}+\theta\left(b_{\ref{prop7}}{  q}^\alpha +\vep_{A,L}\right)\right)e^{-(\mu-\beta)A}Z^{}_{t_k}+O\left(e^{-(\mu-\beta)A}Y^{}_{t_k}\right)\right).
    \end{multline*}

\end{proposition}

\subsection{The particles hitting $L$}

We first control the contribution of the particles that reach $L$. As mentioned in Section \ref{sec:skproof}, we can count the descendants of these particles by stopping them at a  level $L-y$ for some large $y>0$.
\begin{lemma}\label{lem:71}Assume that (\ref{hwp}) holds. Let $y:(1,\infty)\rightarrow (0,\infty)$ be a function such that $y(L)\rightarrow\infty$ and $y(L)/L\to0$ as $L\to \infty$. Uniformly in ${  q}\in[0,\Lambda]$ and $u\in[t_k,t_{k+1}],$ we have
    \begin{equation*}
        \mathbb{E}_{(L-y,u)}
        \left[e^{-{  q} Z^{}_{t_{k+1}}}\right]= \exp\left(-\frac{{  q}}{2}(1+O(\theta e^{2\beta A})+\vep_L)e^{-(\mu-\beta)y}\right).
    \end{equation*}
\end{lemma}
\begin{proof}
    On the event $\{R^{}([u,t_{k+1}])=0\}$, $Z^{}_t=Z_t^{'}$ for all $t\in[u,t_{k+1}]$. Thus, by Markov's inequality, we have
    \begin{eqnarray}
        \left|\mathbb{E}_{(L-y,u)}\left[e^{-{  q} Z^{}_{t_{k+1}}}\right]-\mathbb{E}_{(L-y,u)}\left[e^{-{  q} Z^{'}_{t_{k+1}}}\right]\right|&\leqslant& \mathbb{P}_{(L-y,u)}\left(R^{}([u,t_{k+1}])\geqslant 1\right)\nonumber\\
        &\leqslant& \mathbb{E}_{(L-y,u)}\left[R^{}([u,t_{k+1}])\right]\label{eq:mkRA}.
    \end{eqnarray}
    Yet, according to Lemma \ref{th:r},
    \begin{align}\label{eq:expR}
        \mathbb{E}_{(L-y,u)}\left[R^{}([u,t_{k+1}])\right] & \leqslant C\left(g(L)\ell\left([0,t_{k+1}-u]\right)z(L-y)+e^{-\mu y}\right) \\
                                                           & \leqslant C(\theta e^{2\beta A}z(L-y)+e^{-\mu y}),
        \nonumber
    \end{align}
    where the second inequality comes from \eqref{eq:DLg} and  \eqref{rk:l}.
    Besides, we know from Lemma \ref{lem:fmZ} combined with Equation (\ref{DL}) that
    \begin{align}
        \mathbb{E}_{(L-y,u)}\left[Z^{'}_{t_{k+1}}\right] & =e^{(\lambda_1-\linf)(t_{k+1}-u)}z(L-y)=(1+O(\theta e^{2\beta A}))z(L-y),\label{eq:expZ}
    \end{align}
    and from Lemma \ref{lem:smZ} that
    \begin{align}
        \mathbb{E}_{(L-y,u)}\left[(Z^{'}_{t_{k+1}})^2\right] & \leqslant C((t_{k+1}-u)e^{-2\beta L}z(L-y)+(1\wedge (L-y))e^{-\mu y})\nonumber \\
                                                             & \leqslant C(\theta e^{2\beta A} z(L-y)+e^{-\mu y})\label{eq:varZ}.
    \end{align}
    In addition, Equation \eqref{DL} yields
    \begin{align}\label{eq:zy}
        z(L-y) & =e^{-\mu y}\sinh\left(\sqrt{2\lambda_1}y\right)=e^{-\mu y}\left(\sinh(\beta y)+O(Le^{-\beta L})\right)                                \\
               & =e^{-\mu y}\left(\frac{1}{2}e^{\beta y}+O(ye^{-\beta y})\right)=\frac{1}{2}e^{-(\mu-\beta)y}\left(1+O(ye^{-2\beta y})\right)\nonumber \\
               & =\frac{1}{2}(1+\vep_L)e^{-(\mu-\beta)y},\quad
        \nonumber
    \end{align}
    since $e^{-\beta y}\leqslant ye^{-\beta y}$ and $Le^{-\beta L}\leqslant ye^{-\beta y}$ for $L$ large enough.  Note that the $\vep_L$ depends on the form of the function $y$. Yet, we will only  apply this lemma to a single function $y$ so that we do not need a uniform bound.

    We can now put all these estimates together. Using that
    \begin{equation}
        e^{-\mu y}=e^{-(\mu-\beta) y}e^{-\beta y}= e^{-(\mu-\beta) y}\vep_L \label{eq:int2},
    \end{equation}
    we get from Equations \eqref{eq:expZ}, (\ref{eq:varZ}) and \eqref{eq:zy} that
    \begin{eqnarray*}
        \mathbb{E}_{(L-y,u)}\left[e^{-{  q} Z^{'}_{t_{k+1}}}\right]&=&1-{  q}\mathbb{E}_{(L-y,u)}\left[Z^{'}_{t_{k+1}}\right]+O\left(\mathbb{E}_{(L-y,u)}\left[\left(Z^{'}_{t_{k+1}}\right)^2\right]\right)\\
        &=&1-\frac{{  q}}{2}\left(1+O(\theta e^{2\beta A})+\vep_L\right)e^{-(\mu-\beta)y}.
    \end{eqnarray*}
    Combining this with Equations (\ref{eq:mkRA}), (\ref{eq:expR}) and \eqref{eq:int2}, we obtain
    \begin{equation*}
        \mathbb{E}_{(L-y,u)}\left[e^{-{  q} Z^{}_{t_{k+1}}}\right]=1-\frac{{  q}}{2}\left(1+O(\theta e^{2\beta A})+\vep_L\right)e^{-(\mu-\beta)y}.
    \end{equation*}
    Finally, we use that $e^{-x+O(x^2)}=1-x$ to get that
    \begin{equation*}
        \mathbb{E}_{(L-y,u)}\left[e^{-{  q} Z^{}_{t_{k+1}}}\right]=\exp\left(-\frac{{  q}}{2}(1+O(\theta e^{2\beta A})+\vep_L)e^{-(\mu-\beta)y}\right),\end{equation*}
    which concludes the proof of the lemma.
\end{proof}

\begin{lemma}\label{lemma72}
    Assume that (\ref{hwp}) hold. Uniformly in ${  q}\in[0,\Lambda]$ and in $u\in[t_k,t_{k+1}-L]$,
    \begin{equation*}
        \mathbb{E}_{(L,u)}\left[e^{-{  q} e^{-(\mu-\beta) A} Z^{}_{t_{k+1}}}\right]=\exp\left(\psi_{1,b_{\ref{lem:Zy}}}\left(\frac{{  q}}{2}e^{-(\mu-\beta) A}\right)+e^{-(\mu+\beta) A}\vep_{A,L}\right),
    \end{equation*}
    where $\psi_{1,b_{\ref{lem:Zy}}}({  q})=-{  q}+b_{\ref{lem:Zy}}{  q}^\alpha$ and $b_{\ref{lem:Zy}}$ is the constant from Lemma \ref{lem:Zy}.
\end{lemma}
\begin{proof} Let $y:(1,\infty)\rightarrow (0,\infty)$ be a function such that $y(L)\rightarrow\infty$ and $y(L)/L\to0$ as $L\to \infty$. Starting with a single particle located at $L$ at some time $u\in[t_k,t_{k+1}-L]$, we stop its descendants when they reach the level $L-y$. Denote by $\kappa_y$ the number of particles absorbed at $L-y$ and by $\tau_1\leq\tau_2\leq...\leq\tau_{ \kappa_y}$ the times they hit it. As mentioned in Section \ref{sec:W}, ${ \kappa_y}$ is finite almost surely. Moreover, it is known (see \cite[Theorem 1]{Harris:2007aa}) that there exists a positive constant $d$ (that does not depend on $u$) such that
    \begin{equation*}
        \mathbb{P}_{(L,u)}(\tau_{ \kappa_y}-u>L)\sim d\frac{y}{L^{3/2}}e^{\mu y- \linf L}, \quad L\to\infty.
    \end{equation*}
    Hence, by definition of the function $y$, we get that
    \begin{equation*}
        \mathbb{P}_{(L,u)}(\tau_{ \kappa_y}-u>L)=\vep_L.
    \end{equation*}
    Then,  $\tau_i\in[u,t_{k+1}]\subset[t_k,t_{k+1}]$ for all $i\in\llbracket 1,{ \kappa_y}\rrbracket$, with probability $1-\vep_L$.
    Decomposing $Z_{t_{k+1}}$ into subfamilies according to the ancestors at level $L-y$, we get that
    \begin{eqnarray*}
        \mathbb{E}_{(L,u)}\left[e^{-{  q} e^{-(\mu-\beta)A} Z^{}_{t_{k+1}}}\right]
        &=&\mathbb{E}_{(L,u)}\left[\prod_{i=1}^{ \kappa_y} \mathbb{E}_{(L-y,\tau_i)}\left[e^{-{  q} e^{-(\mu-\beta)A} Z^{}_{t_{k+1}}}\right]\mathbf{1}_{\{\tau_i\in[t_k,t_{k+1}], \forall i\in\llbracket 1,{ \kappa_y}\rrbracket\}}\right]+\vep_L.\\
    \end{eqnarray*}
    Lemma \ref{lem:71} then yields
    \begin{align*}
        \mathbb{E}_{(L,u)} & \left[\prod_{i=1}^{ \kappa_y} \mathbb{E}_{(L-y,\tau_i)}\left[e^{-{  q} e^{-(\mu-\beta)A} Z^{}_{t_{k+1}}}\right]\mathbf{1}_{\{\tau_i\in[t_k,t_{k+1}], \forall i\in\llbracket 1,{ \kappa_y}\rrbracket\}}\right]          \\
                           & =\mathbb{E}_{(L,u)}\left[\exp\left(-\frac{{  q}}{2}(1+O(\theta e^{2\beta A})+\vep_L){ \kappa_y} e^{-(\mu-\beta)y}\right)\mathbf{1}_{\{\tau_i\in[t_k,t_{k+1}], \forall i\in\llbracket 1,{ \kappa_y}\rrbracket\}}\right] \\
                           & =\mathbb{E}_{(L,u)}\left[\exp\left(-\frac{{  q}}{2}e^{-(\mu-\beta)A}(1+O(\theta e^{2\beta A})+\vep_L){ \kappa_y} e^{-(\mu-\beta)y}\right)\right]+\vep_L.
    \end{align*}
    By Lemma \ref{lem:Zy}, the quantity ${ \kappa_y} e^{-(\mu-\beta)y}$ converges in law to a random variable $W$ satisfying \eqref{TE:LTW} as $L\to\infty$. Hence, using that $|e^{-z_1}-e^{-z_2}|<|z_1-z_2|\wedge 1$ for all $z_1,z_2>0$, we get that
    \begin{multline*}
        \mathbb{E}_{(L,u)}\left[\exp\left(-\frac{{  q}}{2}e^{-(\mu-\beta)A}(1+O(\theta e^{2\beta A})+\vep_L){ \kappa_y} e^{-(\mu-\beta)y}\right)\right]\\
        \quad =\mathbb{E}\left[\exp\left(-\frac{{  q}}{2}e^{-(\mu-\beta)A}(1+O(\theta e^{2\beta A})+\vep_L)W\right)\right]+\vep_L.
    \end{multline*}
    On the other hand, remarking that $\alpha(\mu-\beta)=\mu+\beta$, we deduce that
    \begin{align*}
         & \psi_{1,b_{\ref{lem:Zy}}}\left(\frac{{  q}}{2}e^{-(\mu-\beta)A}(1+O(\theta e^{2\beta A})+\vep_L)\right)                                                                                       \\
         & \quad = -\frac{{  q}}{2}e^{-(\mu-\beta)A}(1+O(\theta e^{2\beta A})+\vep_L)+b_{\ref{lem:Zy}}\frac{{  q}^\alpha}{2^\alpha}e^{-(\mu+\beta)A}(1+O(\theta e^{2\beta A})+\vep_L)^\alpha \\
         & \quad = \psi_{1,b_{\ref{lem:Zy}}}\left(\frac{{  q}}{2}e^{-(\mu-\beta)A}\right)+O(\theta e^{-(\mu-3\beta)A})+e^{-(\mu+\beta)A}\vep_{A,L}                                                       \\
         & \quad = \psi_{1,b_{\ref{lem:Zy}}}\left(\frac{{  q}}{2}e^{-(\mu-\beta)A}\right)+e^{-(\mu+\beta)A}\vep_{A,L},
    \end{align*} since $\theta e^{-(\mu-3\beta)A}= \theta e^{4\beta A} e^{-(\mu+\beta)A}$ and $ \theta e^{4\beta A}\rightarrow 0$ as $A\rightarrow \infty$.
    Putting this together with \eqref{TE:LTW}, we see that
    \begin{multline*}
        \mathbb{E}\left[\exp\left(-\frac{{  q}}{2}e^{-(\mu-\beta)A}(1+O(\theta e^{2\beta A})+\vep_L)W\right)\right]\\
        =\exp\left(\psi_{1,b_{\ref{lem:Zy}}}\left(\frac{{  q}}{2}e^{-(\mu-\beta)A}\right)+e^{-(\mu+\beta)A}\vep_{A,L}\right).
    \end{multline*}
    Finally, we get that
    \begin{eqnarray*}
        \mathbb{E}_{(L,u)}\left[e^{-{  q} e^{-(\mu-\beta)\beta A} Z^{}_{t_{k+1}}}\right]&=&\exp\left(\psi_{1,b_{\ref{lem:Zy}}}\left(\frac{{  q}}{2}e^{-(\mu-\beta) A}\right)+e^{-(\mu+\beta)A}\vep_{A,L}\right)+\vep_L\\
        \\ &= & \exp\left(\psi_{1,b_{\ref{lem:Zy}}}\left(\frac{{  q}}{2}e^{-(\mu-\beta) A}\right)+e^{-(\mu+\beta)A}\vep_{A,L}\right),
    \end{eqnarray*}
    which concludes the proof of the lemma.
\end{proof}

\subsection{Proof of Proposition \ref{prop7}}\label{sec:62}

\newcommand\zu{Z^{A,(u)}}

Starting with one particle at $x\leqslant L$ at time $t_k$, we stop the particles when they hit $L$.
We denote by $\mathcal{L}$ the set of particles that hit $L$. For each $v\in \mathcal{L}$, we identify the particle $v$ with the time it hits $L$ and  denote by $Z^{(v)}$ the contribution of its descendants to $Z$.  Writing $Z$ as the sum of these different contributions, we get
\begin{equation*}
    Z^{}_t=Z^{'}_t+\sum_{v\in \mathcal{L}}Z^{(v)}_t.
\end{equation*}
Conditioning on $\mathcal{L}$, we see that
\begin{equation*}
    \mathbb{E}_{(x,t_k)}\left[e^{-{  q} e^{-(\mu-\beta)A}Z^{}_{t_{k+1}}}\right]= \mathbb{E}_{(x,t_k)}\left[e^{-{  q} e^{-(\mu-\beta)A}Z^{'}_{t_{k+1}}}\prod_{v\in\mathcal{L}}\mathbb{E}_{(L,
        u)}\left[e^{-{  q} e^{-(\mu-\beta)A} Z^{(v)}_{t_{k+1}}}\right]\right].\end{equation*}
Since Lemma \ref{lemma72} was only proved for $u\in [t_k,t_{k+1}-L]$, we have to show that only a few particles hit $L$ between times $t_{k+1}-L$ and $t_{k+1}$. Set $s= t_{k+1}-L$. Using Lemma \ref{th:r}, Lemma \ref{lem:fmY}, Equations \eqref{eq:DLg} and \eqref{rk:l},  Markov's inequality and conditioning on $\mathcal{F}_{s}$, we get that
\begin{align}
    \mathbb{P}_{(x,t_k)}\left(|\mathcal{L}\cap[s,t_{k+1}]|\geqslant 1\right) & \leqslant\mathbb{E}_{(x,t_k)}\left[R([s,t_{k+1}])\right]\nonumber                                         \\
                                                                             & \leqslant C\mathbb{E}_{(x,t_k)}\left[(\ell([s,t_{k+1}])+1)g(L)Z^{'}_s+Y_s\right]= \vep_Lz(x).\label{eq:s}
\end{align}
On the other hand, Lemma \ref{lemma72} yields
\begin{align*}
     & \mathbb{E}_{(x,t_k)}\left[e^{-{  q} e^{-(\mu-\beta)A}Z^{'}_{t_{k+1}}}\prod_{v\in\mathcal{L}\cap[t_k,s]}\mathbb{E}_{(L,
    u)}\left[e^{-{  q} e^{-(\mu-\beta)A} Z^{(v)}_{t_{k+1}}}\right]\right]                                                                                                                             \\
     & =\mathbb{E}_{(x,t_k)}\left[\exp\left(-{  q} e^{-(\mu-\beta)A}Z^{'}_{t_{k+1}}\right. \right.                                                                                                    \\
     & \qquad \qquad \qquad \qquad \qquad \left.\left.+R^{}([t_k,s])\left(\psi_{1,b_{\ref{lem:Zy}}}\left(\frac{{  q}}{2}e^{-(\mu-\beta) A}\right)+e^{-({\mu+\beta})A}\vep_{A,L}\right)\right)\right].
\end{align*}
A Taylor expansion combined with Equation (\ref{eq:s}) then gives that
\begin{align}
    \mathbb{E}_{(x,t_k)}\left[e^{-{  q} e^{-(\mu-\beta)A}Z^{}_{t_{k+1}}}\right] & =1-{  q} e^{-(\mu-\beta)A}\mathbb{E}_{(x,t_k)}\left[Z_{t_{k+1}}^{'}\right]\label{eq:TY}                                                                                    \\
                                                                                            & +\left(\psi_{1,b_{\ref{lem:Zy}}}\left(\frac{{  q}}{2}e^{-(\mu-\beta) A}\right)+e^{-({\mu+\beta})A}\vep_{A,L}\right)\mathbb{E}_{(x,t_k)}\left[R^{}([t_k,s])\right]\nonumber \\
                                                                                            & +\;O\left(e^{-2(\mu-\beta)A}\mathbb{E}_{(x,t_k)}\left[\left(Z^{'}_{t_{k+1}}\right)^2+R^{}([t_k,s])^2\right]\right)+\vep_Lz(x).\nonumber
\end{align}
The moments appearing in \eqref{eq:TY} have been bounded in Lemmas \ref{lem:fmZ}, \ref{lem:smZ}, \ref{th:r} and \ref{lem:sdmR}. These lemmas, combined with Equations \eqref{DL}, \eqref{eq:DLg} and \eqref{rk:l} provide the following estimates
\begin{align}
    \mathbb{E}_{(x,t_k)}\left[Z_{t_{k+1}}^{'}\right]                & =e^{(\lambda_1-\linf)(t_{k+1}-t_k)}z(x),\label{tab:1}                                                          \\
    \mathbb{E}_{(x,t_k)}\left[\left(Z_{t_{k+1}}^{'}\right)^2\right] & \leqslant C\left(\theta e^{2\beta A} z(x)+(1\wedge x)e^{\mu(x-L)}\right)\label{tab:2},                         \\
    \mathbb{E}_{(x,t_k)}\left[R^{}([t_k,s])\right]                  & =\ell([0,s-t_k])g(L)z(x)+O\left((1\wedge x)e^{\mu(x-L)}\right),\label{tab:3}                                   \\
    \mathbb{E}_{(x,t_k)}\left[\left(R^{}([t_k,s])\right)^2\right]   & \leqslant C(1+\theta^2e^{4\beta A})\left(\theta e^{2\beta A} z(x)+(1\wedge x)e^{\mu(x-L)}\right)\label{tab:4}.
\end{align}
For the sake of clarity, we will write $Y_{t_k}$ instead of $(1\wedge x)e^{\mu(x-L)}$ in the remainder of the proof.
We know from \eqref{DL} that
\begin{align*}
    \ell([0,s-t_k]) & =\frac{1}{\linf-\lambda_1}\left(1-e^{(\lambda_1-\linf)(s-t_k)}\right)        \\
                    & =(1+\vep_L)a^{-1}e^{2\beta(L-1)}\left(1-e^{(\lambda_1-\linf)(s-t_k)}\right).
\end{align*}
Putting this together with \eqref{eq:DLg} and \eqref{tab:3}, we have
\begin{equation*}
    \mathbb{E}_{(x,t_k)}\left[R^{}([t_k,s])\right]=2(1+\vep_L)\left(1-e^{(\lambda_1-\linf)(s-t_k)}\right) z(x)+O\left(Y_{t_k}\right).
\end{equation*}
Then, recalling that $\psi_{1,b_{\ref{lem:Zy}}}({  q})=-{  q}+b_{\ref{lem:Zy}}{  q}^\alpha$ and that $\alpha(\mu-\beta)=\mu+\beta$, we see that the third term on the RHS of \eqref{eq:TY} is equal to
\begin{align}
    \left(\left(1-e^{(\lambda_1-\linf)(s-t_k)}\right)\left( -e^{-(\mu-\beta)A} {  q}+\frac{b}{2^{\alpha-1}}e^{-(\mu+\beta)A}{  q}^\alpha\right)+e^{-({\mu+\beta})A}\vep_{A,L}\right)z(x)\nonumber \\+O\left(e^{-(\mu-\beta)A}Y_{t_k}\right).\label{eq:t:1}
\end{align}
Let us now bound the last summand on the RHS of \eqref{eq:TY}.
Since $\mu>3\beta$ under \eqref{hwp} (see \eqref{rem:alpha}),
\begin{equation*}
    e^{-2(\mu-\beta)A}e^{2\beta A}=  e^{-(\mu-\beta)A}e^{-(\mu-3\beta)A}=\vep_{A,L}e^{-(\mu-\beta)A}.
\end{equation*}
In addition, note that $\theta^2e^{4\beta L}=\vep_{A,L}.$
Putting this together with (\ref{tab:3}) and (\ref{tab:4}), we get that
\begin{equation}
    e^{-2(\mu-\beta)A}\mathbb{E}_{(x,t_k)}\left[\left(Z^{'}_{t_{k+1}}\right)^2+R^{}([t_k,s])^2\right]=O\left(\theta e^{-(\mu-\beta)A}\vep_{A,L}z(x)+e^{-2(\mu-\beta)A}Y_{t_k}\right) .\label{eq:t:2}
\end{equation}
Equation (\ref{eq:TY}) combined with  Equations (\ref{tab:1}), (\ref{eq:t:1}) and (\ref{eq:t:2}) then yields
\begin{align*}
    \mathbb{E}_{(x,t_k)}\left[e^{-{  q} e^{(\mu-\beta)A}Z_{t_{k+1}}}\right] & =1-{  q} e^{-(\mu-\beta)A}\left(1+e^{(\lambda_1-\linf)(t_{k+1}-t_k)}-e^{(\lambda_1-\linf)(s-t_k)}\right)z(x)      \\
                                                                                        & +\frac{b}{2^{\alpha-1}}{  q}^\alpha e^{-(\mu-\beta)A}\left(1-e^{(\lambda_1-\linf)(s-t_k)}\right)e^{-2\beta A}z(x) \\
                                                                                        & + \theta e^{-(\mu-\beta)A}\vep_{A,L}z(x)+ O\left(e^{-(\mu-\beta)A}Y_k\right),
\end{align*}
where we write $e^{-(\mu+\beta)A}$ as $e^{-(\mu-\beta)A}e^{-2\beta A}$ to obtain the third term and use that $e^{-2(\mu-\beta)A}Y_{t_k}=O\left(e^{-(\mu-\beta)A}Y_{t_k}\right)$ to get the last one.

It now remains to control the exponential factors. Since the exponents are negative,
\begin{equation*}
    \left|e^{(\lambda_1-\linf)(s-t_k)}-e^{(\lambda_1-\linf)(t_{k+1}-t_k)}\right|\leqslant (\linf-\lambda_1)(t_{k+1}-s)=\vep_L,
\end{equation*}
and  \eqref{DL} yields
\begin{equation*}
    e^{(\lambda_1-\linf)(s-t_k)} = \exp\left(-a\theta e^{2\beta A}+\theta e^{2\beta A}\vep_{A,L}\right)=1-a\theta e^{2\beta A}+\theta e^{2\beta A}\vep_{A,L}.
\end{equation*}
Therefore,
\begin{equation}
    \mathbb{E}_{(x,t_k)}\left[e^{-{  q} e^{(\mu-\beta)A}Z^{}_{t_{k+1}}}\right]=1-\left[{  q}  -\frac{ab}{2^{\alpha-1}}\theta{  q}^\alpha + \theta \vep_{A,L}\right]e^{-(\mu-\beta)A}z(x)+O\left(e^{-(\mu-\beta)A}Y_{t_k}\right).\label{eqf61}
\end{equation}
Finally, we use that $e^{-y+O(y^2)}=1-y$ as $y\to0$ to conclude the proof. Hence it suffices to show that $\left(e^{-(\mu-\beta)A}z(x)\right)^2=O(e^{-(\mu-\beta)A}Y_{t_k})$.
Using Lemma \ref{ubv1} and recalling that $\mu>3\beta$ under \eqref{hwp} (see \eqref{rem:alpha}), we see  that
\begin{equation*}
    z(x)^2\leqslant C(1\wedge x)^2e^{2\mu(x-L)}e^{2\beta(L-x)}\leqslant C(1\wedge x)e^{\mu(x-L)}e^{(2\beta -\mu)(L-x)}\leqslant CY_{t_k}.
\end{equation*}
This remark combined with Equation \eqref{eqf61} concludes the proof of Proposition \ref{prop7}.

\section{Convergence to the CSBP}\label{sec:cvcsbp}

This section is devoted to the proof of Theorem \ref{semipushedfr}.
As in \cite{Berestycki2010}, we will first  establish the convergence of the process $Z_t$ (see Theorem \ref{th:csbp} below). The technical arguments used to prove the convergence of $Z$ will be adapted from \cite[Section 8]{Maillard:2020aa}. The proof of Theorem \ref{semipushedfr} will then be deduced following the approach used in \cite{Berestycki2010}.

In this section, we will use all the notation introduced in Section \ref{sec:notation6}. Moreover,  we will need to control the position of the rightmost particle $M(t)$. This quantity could be easily controlled by examining the density of the BBM with absorption at $0$. However, since we only established moment estimates on the BBM in an interval, we will rather consider the density of a BBM killed at 0 and $L+y$, for some large $y>0$. To this extent, we define $Z_{t,y}$, $Z_{t,y}'$, $Y_{t,y}$, $\tilde{Y}_{t,y}$, $R_y$, $w_{1,y}$ in the the same way as $Z_t$, $Z_{t}'$, $Y_t$, $\tilde{Y}_t$, $R$, $w_1$ but for $L-y$ instead of $L$. In what follows, we will write $\Rightarrow$ to refer to the convergence in distribution and $\to_p$ for the convergence in probability.

\subsection{The process $Z_t$}\label{ztcsbp}

\begin{theorem}\label{th:csbp}
    Assume that (\ref{hwp}) holds and suppose that the configuration of particles at time zero satisfies $Z_0\Rightarrow Z$  and   $M(0)-L\rightarrow_p-\infty$ as $L\rightarrow \infty$. Let $b_{\ref{prop7}}$ be the constant from Proposition \ref{prop7}. The finite-dimensional distributions of the processes
    \begin{equation*}
        \left(Z_{e^{2\beta L} t}, t\geqslant 0\right)
    \end{equation*} converges as $L\rightarrow \infty$  to the finite-dimensional distributions of a continuous-state branching process $(\Xi(t),t\geqslant 0)$ with branching mechanism $\Psi({  q}):=b_{\ref{prop7}}{  q}^\alpha$, whose distribution at time zero is the distribution of $Z$.
\end{theorem}

The proof of Theorem \ref{th:csbp} is inspired by \cite[Section 8]{Maillard:2020aa}. Similarly, we discretise time  and use the estimate established in Proposition \ref{prop7}. This allows us to identify the Euler scheme of the branching mechanism $\Psi$. Likewise, we claim that it is sufficient to prove the one-dimensional convergence of the process. Indeed, Theorem \ref{th:csbp} can be deduced from the one-dimensional convergence result and the Markov property of the process if we prove that under the assumptions of Theorem \ref{th:csbp},  for any fixed $t>0$, the two following points hold
\begin{itemize} 
    \item[(1)] $Z(te^{2\beta L})\Rightarrow \Xi_t$, where $\Xi$ is as Theorem \ref{th:csbp},
    \item[(2)]  $L-M(te^{2\beta L})\rightarrow_p \infty$,
\end{itemize}
as $L\to\infty$. Theorem \ref{th:csbp} would follow by induction. Yet, the first point is exactly the one-dimensional convergence of the process $Z$. The second will be proved below using the moment estimates established in Section \ref{sec:moment:est}.

Before we prove this second point, we will need to show that the convergence of the process $Z$ implies the convergence of the processes $Z_y$ under suitable conditions on the initial configuration. This will be the object of the following lemmas.

Note that the assumption $L-M(0)\to_p\infty$ as $L\to\infty$ implies the existence of a sequence $(a_L)$ such that $a_L\to \infty$ and \begin{equation}
    \mathbb{P}(L-M(0)\geqslant a_L)=1-\vep_L \label{p_AL}.
\end{equation}
We denote by $A_L$ the event
\begin{equation*}
    A_L:=\{L-M(0)\geqslant a_L\}.
\end{equation*}
Without loss of generality, one can assume that
\begin{equation}
    a_L\leqslant \sqrt{L}.\label{hyp:aL}
\end{equation}
We first state a technical lemma that compares $w_1$ and $w_{1,-y}$ on $[0,L-a_L]$. The proof of this result can be found in Appendix \ref{proof:lem71}
\begin{lemma}\label{est:wy}
    \begin{equation*}  
        \sup_{y\in[0,\infty)} \quad \sup_{ x\in[0,L-a_L]}\left|e^{-\beta y}\frac{w_{1,-y}(x)}{w_1(x)}-1\right|\to 0, \quad L\to\infty.
    \end{equation*}
\end{lemma}
\begin{lemma}\label{lem:CVZ}
    Let $t>0$ and $y\in\mathbb{R}$. Suppose  that $Z_{te^{2\beta L}}\Rightarrow Z'$ for some random variable $Z'\geqslant 0$ and that  $L-M(te^{2\beta L})\to_p\infty$ as $L\to\infty$. Then, as $L\to \infty$,
    \begin{equation*}
        Z_{te^{2\beta L},-y}\Rightarrow e^{-(\mu-\beta)y}Z'.
    \end{equation*}
\end{lemma}

\begin{proof}
    We deal with the case $t=0$. The proof is similar for fixed $t>0$. Let $f$ be a bounded and continuous test function on $[0,\infty)$. Conditioning on the event $A_L$, we see that
    \begin{eqnarray*}
        \mathbb{E}\left[|f(Z_{0,{-y}})-f(e^{-(\mu-\beta)y}Z_0)|\right]\leqslant 2\|f\|_\infty\mathbb{P}(A_L^c)+ \mathbb{E}\left[\left|f(Z_{0,{-y}})-f(e^{-(\mu-\beta)y}Z_0)\right|\mathbf{1}_{A_L}\right].
    \end{eqnarray*}
    Recall from (\ref{p_AL}) that $\mathbb{P}(A_L^c)=\vep_L$ so that we only need to control the second term on the RHS of the above. Yet, we know from Lemma \ref{est:wy} that
    \begin{equation*}
        \text{ on $A_L$,} \quad e^{(\mu-\beta) y} Z_{0,-y}=(1+\vep_L)Z_0.
    \end{equation*}
    The result follows by continuity of $f$.
\end{proof}

\begin{lemma}\label{lem:ZsqrtL}
    Suppose that $Z_0\Rightarrow Z'$ for some random variable $Z'\geqslant 0$ and that $L-M(0)\to_p\infty$ as $L\to\infty$. Then, as $L\to \infty$
    \begin{equation*}
        Z_{0,-\sqrt{L}}\to_p0.
    \end{equation*}
\end{lemma}
\begin{proof}
    Let $\vep>0$. Conditioning on the event $A_L$, we see that
    \begin{equation*}
        \mathbb{P}\left(Z_{0,-\sqrt{L}}>\vep\right)= \mathbb{P}\left(Z_{0,-\sqrt{L}}\mathbf{1}_{A_L}>\vep\right)+\vep_L.
    \end{equation*}
    According to Lemma \ref{est:wy}, we know that
    \begin{equation*}
        \text{on $A_L$}, \quad Z_{0,-\sqrt{L}}= (1+\vep_L)e^{-(\mu-\beta) \sqrt{L}} Z_0.
    \end{equation*}
    Hence, since $Z_0\mathbf{1}_{A_L}\Rightarrow Z$ and $(1+\vep_L)e^{-\beta \sqrt{L}}\to0$ as $L\to 0$, Slutsky's theorem yields the result.
\end{proof}

We are now ready to prove point (2).
\begin{lemma}\label{lem:max}
    Let $t>0$. Suppose that $Z_0\Rightarrow Z'$ for some random variable $Z'\geqslant 0$ and that $L-M(0)\to_p\infty$ as $L\to\infty$. Then, as $L\to \infty$,
    \begin{equation*}
        L-M(te^{2\beta L})\to_p\infty.
    \end{equation*}
    More precisely, we have
    \begin{equation*}
        \mathbb{P}\left(L-M(te^{2\beta L})\leqslant a_L\right)=\vep_L.
    \end{equation*}
\end{lemma}
\begin{proof}
    Consider the event that no particle reaches level $L+\sqrt{L}$ before time $te^{2\beta L}$ \begin{equation*}
        B_L=\{R_{-\sqrt{L}}([0,te^{2\beta L}])=0\}.
    \end{equation*}
    Using Lemma \ref{lemmatild} and  that $\tilde{Y}\geqslant 0$, we get that  for $L$ large enough
    \begin{equation}
        \mathbb{E}\left[\tilde{Y}_{te^{2\beta L},{-\sqrt{L}}}\mathbf{1}_{B_L}|\mathcal{F}_0\right]\leqslant\mathbb{E}\left[\tilde{Y}_{te^{2\beta L},{-\sqrt{L}}}|\mathcal{F}_0\right]\leqslant Ce^{-\beta (L+\sqrt{L})}Z_{0,-\sqrt{L}}'.\label{eq:EYtildeL}
    \end{equation}
    Lemma \ref{th:r} combined with Equation (\ref{eq:DLg}) implies that for $L$ large enough
    \begin{align}
        \mathbb{E}\left[R_{-\sqrt{L}}([0,te^{2\beta L}])|\mathcal{F}_0\right] & \leqslant C\left(Y_{0,-\sqrt{L}}+g(L+\sqrt{L})te^{2\beta L}Z_{0,-\sqrt{L}}\right)\nonumber \\&\leqslant C\left( Y_{0,-\sqrt{L}}+ e^{-2\beta \sqrt{L}}Z_{0,-\sqrt{L}}\right).\label{eq:ERL}
    \end{align}
    Then, conditioning on the event $A_L$, we see that
    \begin{equation}
        \mathbb{P}\left(L-M(te^{2\beta L})\leqslant a_L|\mathcal{F}_0\right)\leqslant
        \mathbb{P}\left(L-M(te^{2\beta L})\leqslant a_L|\mathcal{F}_0\right)\mathbf{1}_{A_L}+\mathbf{1}_{A_L^c}. \label{eq:lm}
    \end{equation}
    Since $\mathbb{P}(A_L^c)=\vep_L$, we only need to bound the first term on the RHS of \eqref{eq:lm}. Yet, note that
    \begin{equation*}
        \mathbb{P}\left(L-M(te^{2\beta L})\leqslant a_L|\mathcal{F}_0\right)=\mathbb{P}\left(e^{\mu(M(te^{2\beta L})-(L+\sqrt{L}))}\geqslant e^{-\mu (\sqrt{L}+a_L)}|\mathcal{F}_0\right).
    \end{equation*}
    In addition, remark that
    $$ B_L\cap \left\{e^{\mu(M(te^{2\beta L})-(L+\sqrt{L}))}\geqslant e^{-\mu (\sqrt{L}+a_L)}\right\} \subset B_L\cap \left\{\tilde{Y}_{-\sqrt{L},te^{2\beta L}}\geqslant e^{-\mu (\sqrt{L}+a_L)}\right\}.
    $$
    Hence,
    \begin{equation*}
        \mathbb{P}\left(e^{\mu(M(te^{2\beta L})-(L+\sqrt{L})}\mathbf{1}_{B_L}\geqslant e^{-\mu (\sqrt{L}+a_L)}|\mathcal{F}_0\right)\leqslant \mathbb{P}\left(\tilde{Y}_{te^{2\beta L},-\sqrt{L}}\mathbf{1}_{B_L}\geqslant e^{-\mu (\sqrt{L}+a_L)}|\mathcal{F}_0\right).
    \end{equation*}
    Therefore, conditioning on the event $B_L$, we get that
    \begin{align}
         &   \mathbb{P}\left(L-M(te^{2\beta L})\leqslant a_L|\mathcal{F}_0\right)\mathbf{1}_{A_L}\nonumber                                                                                                                           \\&\leqslant \mathbb{P}\left(\tilde{Y}_{te^{2\beta L},-\sqrt{L}}\mathbf{1}_{B_L}\geqslant e^{-\mu (\sqrt{L}+a_L)}|\mathcal{F}_0\right)\mathbf{1}_{A_L}+\mathbb{P}\left(B_L^c|\mathcal{F}_0\right)\mathbf{1}_{A_L}\nonumber\\
         & \leqslant \left(e^{\mu(\sqrt{L}+a_L)}\mathbb{E}\left[\tilde{Y}_{te^{2\beta L},{-\sqrt{L}}}\mathbf{1}_{B_L}|\mathcal{F}_0\right]+\mathbb{E}\left[R_{-\sqrt{L}}([0,te^{2\beta L}])|\mathcal{F}_0\right]\right)\mathbf{1}_{A_L}\nonumber \\
         & \leqslant C\left(\left[e^{\mu(\sqrt{L}+a_L)}e^{-\beta(L+\sqrt{L})}+e^{-2\beta \sqrt{L}}\right]Z_{0,-\sqrt{L}}+Y_{0,-\sqrt{L}}\right)\mathbf{1}_{A_L},\label{eq:77}
    \end{align}
    where the second line is obtained thanks to conditional Markov's inequality and the last one by combining Equations \eqref{eq:EYtildeL} and \eqref{eq:ERL}. Then, we see from \eqref{hyp:aL}  that
    $$e^{\mu(\sqrt{L}+a_L)}e^{-\beta(L+\sqrt{L})}\leqslant e^{(2\mu-\beta)\sqrt{L}-\beta L}\leqslant e^{-CL},$$
    for $L$ large enough.
    We then recall from Lemma \ref{ubv1} that for all $x\in[0,L+\sqrt{L}]$, we have
    \begin{equation*}
        \frac{1\wedge x\wedge (L+\sqrt{L}-x)}{w_{1,-\sqrt{L}}(x)}\leqslant Ce^{-\beta(L+\sqrt{L}-x)},
    \end{equation*}
    so that for all $L$ large enough
    \begin{align*}
        \text{on $A_L$,}\quad \forall u\in \mathcal{N}_0, \quad 1\wedge X_u(0) & = 1\wedge X_u(0)\wedge (L+\sqrt{L}-X_u(0))                                                            \\
                                                                               & \leqslant C e^{-\beta (\sqrt{L}+a_L)}w_{1,-L}(X_u(0))\leqslant C e^{-\beta \sqrt{L}}w_{1,-L}(X_u(0)).
    \end{align*}
    Thus we get that, for sufficiently large $L$,
    \begin{equation}
        \text{on $A_L$,} \quad  Y_{0,-\sqrt{L}}\leqslant C e^{-\beta \sqrt{L}}Z_{0,-\sqrt{L}}.\label{eq:79}
    \end{equation}
    Finally, combining Equations \eqref{eq:lm}, \eqref{eq:77} and \eqref{eq:79}, we get that
    \begin{equation}
        \mathbb{P}\left(L-M(te^{2\beta L})\geqslant a_L|\mathcal{F}_0\right)= \vep_LZ_{0,-\sqrt{L}}\mathbf{1}_{A_L}+\mathbf{1}_{A_L^c} \label{IQ:aL}.
    \end{equation}
    Lemma \ref{lem:ZsqrtL} then yields the result.
\end{proof}

We now move to point (1). As in \cite{Maillard:2020aa}, we claim that
it is  sufficient to assume that $Z_0\rightarrow_pz_0$ as $L\rightarrow \infty$ for some constant $z_0\geq 0$ instead of the one dimensional convergence (1). One can then deduce Theorem \ref{th:csbp} thanks to a conditioning argument.
Hence, we will prove the following: for fixed $t>0$ and  ${  q}>0$,
\begin{equation}
    \underset{L\rightarrow \infty}{\lim}\mathbb{E}\left[e^{-{  q} Z_{te^{2\beta L}}}\right]=e^{-z_0u_t({  q})}, \label{eq:CV:CSBP}
\end{equation}
where $u_t({  q})$ is the function from Equation (\ref{diffeq}) corresponding to the branching mechanism $\Psi({  q})=b_{\ref{prop7}}{  q}^\alpha$.
We recall from Section \ref{sec:notation6} that we divided the interval $[0,te^{2\beta L}]$ into small time steps of length $\theta e^{2\beta(L+A)}$ so that we consider the process at times $t_k=k\theta e^{2\beta (L+A)}$ for  $k\in\llbracket 0,  {\kappa}\rrbracket$.

We also recall  the definition of $a_L$ from Equation (\ref{p_AL}) and define $b_L=a_L-A$. Note that $b_L\to\infty$ as $L\to \infty$ by definition of $A$. For $k\in\llbracket 0,  {\kappa}\rrbracket$, let
\begin{equation}
    G_k=\left\{\forall j\in\llbracket0,k\rrbracket:M(t_j)\leqslant L-A-b_L, \, Y_{t_j,A}\leqslant Z_{t_j,A}/b_L\right\}. \label{def:Gk}
\end{equation}

\begin{lemma}\label{lem:proba:G}
    We have $\mathbb{P}(G_{  {\kappa}})=1-\vep_L$.
\end{lemma}
\begin{proof} On $A_L$, we have $M(t_0)\leqslant L-a_L= L-A-b_L$. Moreover, we see from Lemma \ref{ubv1} that for sufficiently large $L$,
    \begin{equation*}
        \text{on $A_L$,}\quad Y_{t_0,A}\leqslant  e^{-\beta(a_L-A)}Z_{t_0,A}\leqslant e^{-\beta b_L}Z_{t_0,A}\leqslant \frac{1}{b_L}Z_{t_0,A}.
    \end{equation*}
    Thus, $\mathbb{P}(G_0)=1-\vep_L$.  Let $k\in\llbracket 1,  {\kappa}\rrbracket$. We know from Lemma \ref{lem:max} that $$\mathbb{P}\left(L-A-M(t_k)\geqslant b_L\right)=1-\vep_L.$$ Similarly, we can prove that, on this event,  $Y_{t_k,A}\leqslant Z_{t_k,A}/b_L$ for $L$ large enough. We conclude the proof of the lemma using a union bound.
\end{proof}

For fixed ${  q}>0$ and $\delta\in \mathbb{R}$, we define the sequence $({  q}_k^{(\delta)})_{k=0}^{  {\kappa}}$ by
\begin{align*}
    {  q}_{  {\kappa}}^{(\delta )} & ={  q},                                                                           \\
    {  q}_k^{(\delta)}                            & = {  q}_{k+1}^{(\delta)}-\theta(\Psi({  q}_{k+1}^{(\delta)})-\delta).
\end{align*}

\begin{lemma}\label{lem:82}

    \begin{itemize}  Fix $t>0$ and  ${  q}>0$. Suppose that $\theta$ is as in \eqref{disc:t}.
        \item[1.] There exists $\Lambda>0$ such that for $|\delta |$ small enough and for $\theta$ small enough, we have ${  q}_k^{(\delta)}\in[0,\Lambda]$ for all $k\in \llbracket 0,   {\kappa}\rrbracket$ .
        \item[2.] For every $\eta>0$, there exists $\delta >0$ such that for $\theta$  small enough, we have
              \begin{equation*}
                  {  q}_0^{(\delta)},{  q}_0^{(-\delta)}\in[u_t({  q})-\eta,u_t({  q})+\eta].
              \end{equation*}
        \item[3.] For every $\delta >0$, we have for sufficiently large $A$ and $L$, for every $k=0,...,  {\kappa},$
              \begin{eqnarray}
                  \mathbb{E}\left[e^{-{  q}_k^{(\delta)}e^{-(\mu-\beta)A}Z_{t_k,A}}\mathbf{1}_{G_k}\right]-\mathbb{P}\left(G_{  {\kappa}}\setminus G_k\right)\leqslant\mathbb{E}\left[e^{-{  q} e^{-(\mu-\beta)A}Z_{t_{  {\kappa}},A}}\mathbf{1}_{G_{  {\kappa}}}\right],
                  \label{eq:lem82}
              \end{eqnarray}
              and
              \begin{equation}
                  \mathbb{E}\left[e^{-{  q}_k^{(-\delta)}e^{-(\mu-\beta)A}Z_{t_k,A}}\mathbf{1}_{G_k}\right]  \geqslant \mathbb{E}\left[e^{-{  q} e^{-(\mu-\beta)A}Z_{t_{  {\kappa}},A}}\mathbf{1}_{G_{  {\kappa}}}\right].
              \end{equation}
    \end{itemize}
\end{lemma}

\begin{proof} The proof of parts $1$ and $2$ relies on standard results on the Euler scheme and is similar to the proof of Theorem 2.1 in \cite{Maillard:2020aa}. The only difference is that we do not need to modify the function $\Psi$ at zero, since it is a Lipschitz function on any interval $[0,\Lambda]$, $\Lambda>0.$

    Hence, we only prove part $3$ of the lemma.
    Let $\Lambda>0$ be such that the first part of the lemma holds. Let $\delta>0$ and $\tilde \theta$ be small enough so that ${  q}_k^{(\pm\delta)}\in[0,\Lambda]$ for all $k\in \llbracket0,  {\kappa}\rrbracket.$ By Proposition \ref{prop7}, we know that for $L$ and $A$ sufficiently large, for all $\theta<\bar \theta(A)\wedge \tilde \theta$, for all ${  q}'\in[0,\Lambda]$, and for all $k\in \llbracket 0,  {\kappa}\rrbracket$,
    \begin{multline*}
        e^{-{  q}'+\theta(\Psi({  q}')-\delta)e^{-(\mu-\beta)A}Z_{t_k,A}}\mathbf{1}_{G_k}\\\leqslant \mathbb{E}\left[e^{-{  q} 'e^{-(\mu-\beta)A}Z_{t_k,A}}|\mathcal{F}_k\right]\mathbf{1}_{G_k}\\ \leqslant e^{-{  q}'+\theta(\Psi({  q}')+\delta)e^{-(\mu-\beta)A}Z_{t_k,A}}\mathbf{1}_{G_k},
    \end{multline*}
    almost surely.
    Since ${  q}_k^{(\pm\delta)}\in[0,\Lambda]$ for all $k\in \llbracket0,  {\kappa}\rrbracket,$ this also implies that for all  $k\in \llbracket0,  {\kappa}\rrbracket,$
    \begin{align}
        \mathbb{E}\left[e^{-{  q}^{(\delta)}_{k+1}e^{-(\mu-\beta)A}Z_{t_{k+1,A}}}|\mathcal{F}_k\right]\mathbf{1}_{G_k}  & \geqslant e^{-{  q}_k^{(\delta)}e^{-(\mu-\beta)A}Z_{t_k,A}}\mathbf{1}_{G_k},\label{79-}   \\
        \mathbb{E}\left[e^{-{  q}^{(-\delta)}_{k+1}e^{-(\mu-\beta)A}Z_{t_{k+1,A}}}|\mathcal{F}_k\right]\mathbf{1}_{G_k} & \leqslant e^{-{  q}_k^{(-\delta)}e^{-(\mu-\beta)A}Z_{t_k,A}}\mathbf{1}_{G_k}. \label{79+}
    \end{align}
    The third part of the lemma follows by induction as in \cite{Maillard:2020aa}: for $k=  {\kappa}$, Equation \eqref{eq:lem82} holds. Let $k\in\llbracket 0,  {\kappa}-1\rrbracket$ and assume that \eqref{eq:lem82} holds for $k+1$. By definition of the sequence $(G_k)$, we have  $G_k\subset G_{k+1}$,  so that the induction hypothesis implies that
    \begin{align*}
        \mathbb{E}\left[e^{-{  q}_k^{(\delta)}e^{-(\mu-\beta)A}Z_{t_{k+1},A}}\mathbf{1}_{G_k}\right]-\mathbb{P}\left(G_{  {\kappa}}\setminus G_k\right) & \leqslant\mathbb{E}\left[e^{-{  q} e^{-(\mu-\beta)A}Z_{t_{  {\kappa}},A}}\mathbf{1}_{G_{  {\kappa}}}\right], \\
        \mathbb{E}\left[e^{-{  q}_k^{(-\delta)}e^{-(\mu-\beta)A}Z_{t_{k+1},A}}\mathbf{1}_{G_k}\right]                                                                  & \geqslant\mathbb{E}\left[e^{-{  q} e^{-(\mu-\beta)A}Z_{t_{  {\kappa}},A}}\mathbf{1}_{G_{  {\kappa}}}\right].
    \end{align*}
    These equations combined with Equations \eqref{79-} and \eqref{79+} concludes the proof of the third point.
\end{proof}

We now deduce (\ref{eq:CV:CSBP}) from Lemmas \ref{lem:proba:G} and \ref{lem:82}.

We see from  Lemma \ref{est:wy} (applied to $L-A$ instead of $L$) that
\begin{equation*}
    \text{on $G_{  {\kappa}}$,}\quad Z_{t_{  {\kappa}},A}=(1+\vep_L)e^{(\mu-\beta)A}Z_{t_{  {\kappa}}}.
\end{equation*}
Combining this with Lemma  \ref{lem:proba:G}, we obtain
\begin{equation}
    \mathbb{E}\left[e^{-{  q} e^{-(\mu-\beta)A} Z_{t_{  {\kappa}},A}}\mathbf{1}_{G_{  {\kappa}}}\right]=\mathbb{E}\left[e^{-{  q} Z_{t_{  {\kappa}}}}\mathbf{1}_{G_{  {\kappa}}}\right]+\vep_L=\mathbb{E}\left[e^{-{  q} Z_{t_{  {\kappa}}}}\right]+\vep_L. \label{eq:715}
\end{equation}
Let $\eta>0$ and choose $\delta>0$ such that the second part of Lemma \ref{lem:82} holds. Therefore, the third part of the lemma and Equation \eqref{eq:715} imply that for $L$ and $A$ large enough,
\begin{equation}\label{eq:716}
    \mathbb{E}\left[e^{-(u_t({  q})+\eta)e^{-(\mu-\beta)A}Z_{0,A}}\right]-\vep_L\leqslant\mathbb{E}\left[e^{-{  q} Z_{t_{  {\kappa}}}}\right]\leqslant \mathbb{E}\left[e^{-(u_t({  q})-\eta)e^{-(\mu-\beta)A}Z_{0,A}}\right]+\vep_L.
\end{equation}
Since $\mathbb{P}(G_{  {\kappa}})=1-\vep_L$, we know from Lemma \ref{lem:CVZ} (applied to $L-A$ instead of $L$) that
\begin{equation*}
    Z_{0,A}\Rightarrow e^{(\mu-\beta)A}z_0, \quad L\to\infty.
\end{equation*}
Hence, letting $L\to\infty$ in \eqref{eq:716}, we get that
\begin{equation*}
    e^{-(u_t({  q})+\eta)z_0}\leqslant \liminf\limits_{L\rightarrow\infty}\mathbb{E}\left[e^{-{  q} Z_{te^{2\beta L}}}\right]\leqslant \limsup\limits_{L\rightarrow\infty}\mathbb{E}\left[e^{-{  q} Z_{te^{2\beta L}}}\right]\leqslant e^{-(u_t({  q})-\eta)z_0}.
\end{equation*}
Letting $\eta\to0$ then concludes the proof of \eqref{eq:CV:CSBP}.

\subsection{The number of particles $N_t$} \label{nbrpart:CSBP}
In this section, we conclude the proof of Theorem \ref{semipushedfr} by deducing the result on the number of particles $N_t$ from the convergence of the process $Z_t$ established in Theorem \ref{th:csbp}.  In Theorem \ref{ThNt}, we state a version of Theorem \ref{semipushedfr}, under more general assumptions on the initial configuration.

First, recall  the definitions of the processes $Z_t$, $Z_t'$, $N_t$ and $N_t'$ from Section \ref{sec:smallts} (see Equations \eqref{def:Nt} and \eqref{def:Zt}). In addition, define the normalising constant
\begin{equation}
    \label{nc}
    d_{\infty}=\frac{1}{2}\left(\lim\limits_{L\rightarrow\infty}\|v_1\|\right)^{-2}\left(\lim\limits_{L\rightarrow\infty}\int_0^Le^{-\mu y} v_1(y)dy\right).
\end{equation}
Recall from Lemma \ref{lem:l2norm} that the $\mathrm{L}^2$-norm of the eigenvector $v_1$ converges to a positive limit as $L$ goes to $\infty$. Besides, one can prove that $L\mapsto \int_0^Le^{-\mu y} v_1(y)dy$ also converges to a positive limit as $L\to \infty$ using the dominated convergence theorem combined with Lemma \ref{ubv1}. Thus the constant $d_{\infty}$ is well defined.

\begin{theorem}\label{ThNt} Assume that \eqref{hwp} holds and let $b_{\ref{prop7}}$ be as in Theorem \ref{th:csbp}.  In addition, suppose that the configuration of particles at time zero satisfies $Z_0\to_p z_0$, for some $z_0>0$,  and  that $L-M(0)\rightarrow_p\infty$ as $L\rightarrow \infty$. Then, the finite-dimensional distributions of the processes
    \begin{equation*}
        \left(\frac{1}{d_\infty e^{(\mu-\beta)L}z_0}N_{e^{2\beta L} t}, t\geqslant 0\right)
    \end{equation*} converge as $L\rightarrow \infty$  to the finite-dimensional distributions of a continuous-state branching process $(\Xi(t),t\geqslant 0)$ with branching mechanism $\Psi({  q}):=b_{\ref{prop7}}{  q}^\alpha$ starting from $1$.
\end{theorem}
\begin{rem} 
    As in the statement of Theorem \ref{th:csbp}, one could assume that $Z_0\Rightarrow Z$ for some random variable $Z\geq 0$. In this case, one could show (using a conditioning argument) that the finite-dimensional distributions of the processes $((d_\infty e^{(\mu-\beta)L})^{-1}N_{e^{2\beta L}t}, \ t\geq 0)$ converge to the finite-dimensional distributions of a CSBP with branching mechanism $\Psi({  q}):=b_{\ref{prop7}}{  q}^\alpha$, whose distribution at time $0$ is the distribution of $Z$.
\end{rem}

Note that Theorem \ref{semipushedfr} can be deduced from Theorem \ref{ThNt} by computing $Z_0$ when the system starts with $N$ particles located at $1$. In this case,  Remark \ref{r:vw} entails
\begin{equation*}
    Z_0=Ne^{\mu(1-L)}w_1(1)=\frac{1}{2}N(1+\vep_L)e^{-(\mu-\beta)L}e^{\mu-\beta}.
\end{equation*}
Thus we set
$$L=\frac{1}{\mu-\beta}\log(N),$$
so that
\begin{equation}
    N=e^{(\mu-\beta)L} \quad  \text{and}\quad  N^{\alpha-1}=e^{2\beta L}.\label{eq:NL}
\end{equation}
Then, Theorem \ref{semipushedfr}  follows directly from Theorem \ref{ThNt} and the normalising constant $\sigma(\rho)$ can be expressed as a function $d_\infty$,
\begin{equation}
    \sigma(\rho)=\frac{2}{d_\infty e^{\mu-\beta}}.\label{def:sigma}
\end{equation}

As outlined in Section \ref{sec:skproof},  Theorem \ref{ThNt} can be deduced from the convergence of the process $Z_t$ because $N_t$ is \textit{roughly proportional} to $Z_t$: there exists a constant $C>0$ such that the number of particles $N_t$ can be approximated by
\begin{equation*}
    N_t\approx Ce^{(\mu-\beta)L}Z_t,
\end{equation*}
for $t$ and $L$ large enough. Actually, we have $C=d_\infty$ and a rigorous statement of this claim is given in the following lemma. This result is analogous to the one proved in \cite[Section 6.3]{Berestycki2010} in the case $\rho=1$ and the strategy of the proof is similar.

\begin{lemma}\label{lemf}
    Assume that (\ref{hwp}) holds and suppose the configuration of particles at time zero satisfies $Z_0\to_p z_0$, for some $z_0>0$,  and   $L-M(0)\rightarrow_p\infty$ as $L\rightarrow \infty$. Let $t>0$. Then, we have
    \begin{equation}
        \left| e^{-(\mu-\beta)L}N_{e^{2\beta L}t} - d_\infty Z_{e^{2\beta L}t}\right|\to_p 0, \quad L\to\infty.
    \end{equation}
\end{lemma}
Theorem \ref{ThNt} then follows  from Theorem \ref{th:csbp} and Lemma \ref{lemf}.

We now consider a fixed time $t$ and denote by $u$ the corresponding time on the time scale of the CSBP
\begin{equation*}
    u=te^{2\beta L}.
\end{equation*}
We also consider a small parameter $0<\delta<1$ and note that, according to  Lemma \ref{eigenvlocasymp} and Lemma  \ref{th1}, we have
\begin{equation}
    p_{\delta u}(x,y)= (1+\vep_L)(1+O(\delta))\|v_1\|^{-2}e^{\mu(x-y)}v_1(x)v_1(y).\label{rk:lem:p}
\end{equation}
The proof of Lemma \ref{lemf} is divided into three parts:
\begin{enumerate}
    \item  We first prove that $N'_u$ is well approximated by $d_\infty e^{(\mu-\beta)L}Z'_{(1-\delta)u}$ for large values of $L$, by controlling the first and second moments of $N'$ (see Lemma \ref{lem1:N} and Lemma \ref{lem2:N}).
    \item Then, we show that $Z'$ does not vary much between times $(1-\delta)u$ and $u$ for $\delta$ small enough with a similar argument.
    \item Finally, we recall why $Z_u'$ (resp. $N'_u$) is a good approximation of $Z_u$ (resp. $N_u$)  for $L$ large enough.
\end{enumerate}
We now move to the moment estimates. As in Section \ref{sec:moment:est}, we estimate the first moment of $N'$ under (\ref{hpushed}) and bound its second moment  under (\ref{hwp}).

\begin{lemma}[First moment of $N'$]\label{lem1:N} Assume that (\ref{hpushed}) holds. Then,

    \begin{equation*}
        \mathbb{E}[N'_{u}|\mathcal{F}_{(1-\delta)s}]=d_\infty(1+\vep_L)(1+O(\delta))e^{(\mu-\beta)L}Z'_{(1-\delta)u}.
    \end{equation*}
\end{lemma}
\begin{proof}
    The many-to-one lemma (see Lemma \ref{lem:many-to-one}) yields
    \begin{equation*}
        \mathbb{E}_{(x,(1-\delta)u)}[N'_{u}]=\mathbb{E}_{(x,(1-\delta)u)}\left[\sum_{v\in\mathcal{N}^L_{u}}1\right]=\int_0^L p_{\delta u}(x,y)dy.
    \end{equation*}
    Then, we see from  \eqref{rk:lem:p} that
    \begin{equation*}
        \mathbb{E}_{(x,(1-\delta)u)}[N'_{u}]= (1+\vep_L)(1+O(\delta))\|v_1\|^{-2}e^{\mu x}v_1(x)\int_0^L e^{-\mu y}v_1(y)dy.
    \end{equation*}
    Moreover, Remark \ref{r:vw} implies that
    \begin{equation}
        e^{\mu x}v_1(x)=\frac{1}{2}(1+\vep_L) e^{(\mu-\beta)L} e^{\mu(x-L)}w_1(x)\label{rem:21}.
    \end{equation}
    Finally, by definition of $d_\infty$, we get that
    \begin{equation*}
        \mathbb{E}_{(x,(1-\delta)u)}[N'_{u}]=d_\infty(1+\vep_L)(1+O(\delta))e^{(\mu-\beta)L}e^{\mu(x-L)}w_1(x),
    \end{equation*}
    which concludes the proof of the lemma.
\end{proof}

\begin{lemma}[Second moment of $N'$]\label{lem2:N} Assume that (\ref{hwp}) holds. For $L$ large enough, we have
    \begin{equation*}
         \mathbb{E}_{(x,(1-\delta)u)}\left[(N'_{u})^2\right]\leqslant C e^{2(\mu-\beta)L}\left(Y_{(1-\delta)u}+\delta Z'_{(1-\delta)u}\right), \quad \forall x\in[0,L].
    \end{equation*}
\end{lemma}

\begin{proof}
    Lemma \ref{lem:many-to-two} entails
    \begin{equation*}
        \mathbb{E}_{(x,(1-\delta)u)}\left[(N'_{u})^2\right]=\mathbb{E}_{(x,(1-\delta)u)}\left[N'_{u}\right]+ \underbrace{\int_0^{\delta u}\int_0^Lp_s(x,y)2r(y)\mathbb{E}_{y}\left[N_{\delta u-s}'\right]^2dyds}_{\textstyle=: \tilde U}.
    \end{equation*}
    Applying the many-to-one lemma and interchanging the integrals, the quantity $ \tilde U$ can be written as
    \begin{equation*}
        \tilde U=\int_0^L2r(y)\int_0^{\delta u}p_s(x,y)\left(\int_0^Lp_{\delta u-s}(y,z) dz\right)^2 dsdy.
    \end{equation*}
    We then divide the second integral into three parts.
    We first recall from Equation \eqref{ub:pt2} that for $\delta u-s>c_{\ref{th1}} L$ and $L$ large enough,
    \begin{equation} 
        p_{\delta u-s}(y,z)\leqslant C e^{\mu(y-z)}v_1(z)v_1(y). \label{eq:ps}
    \end{equation}
    Therefore, combining \eqref{rem:21} and \eqref{eq:ps} with Lemma \ref{ubv1}, we get that
    \begin{align*}
        \tilde U_1 & :=\int_0^L2r(y)\int_{0}^{\delta u-c_{\ref{th1}}L}p_s(x,y)\left(\int_0^Lp_{\delta u-s}(y,z) dz\right)^2 \,ds\,dy                                          \\
                   &  \leqslant C \int _0^L \int_{0}^{\delta u-c_{\ref{th1}}L}p_s(x,y)\left(e^{2y}v_1(y)\right)^2\left(\int_0^Le^{(\mu-\beta)z}dz\right)^2ds \ dy \\
                   & \leqslant  C e^{2(\mu-\beta)L}\int_0^Le^{2\mu(y-L)}w_1^2(y)\int_0^{\delta u-c_{\ref{th1}} L}p_s(x,y)\,ds\,dy                                             \\
                   & \leqslant  Ce^{2(\mu-\beta)L}\int_0^Le^{2\mu(y-L)}w_1^2(y)\int_0^{\delta u}p_s(x,y)\,ds\,dy                                                              \\
                   & \leqslant  Ce^{2(\mu-\beta)L}\left(e^{\mu(x-L)}\left((1\wedge x)+\delta u e^{-2\beta L}w_1(x)\right)\right).
    \end{align*}
    The last line is obtained by applying a similar argument to that developed in the proof of Lemma \ref{lem:smZ}. More precisely, we apply Equations (\ref{int:Green}) and (\ref{est:T2}) to $u(L)=\delta e^{2\beta L}$. Hence, we get that for $L$ large enough
    \begin{equation}\label{est:A1:N}
        \tilde U_1\leqslant C e^{2(\mu-\beta)L}e^{\mu(x-L)}\left((1\wedge x)+\delta w_1(x)\right).\end{equation}

    For small values of $\delta u -s$, we bound the integral $\int_0^Lp_{\delta u-s}(y,z) dz$ by the expected number of particles in a branching Brownian motion with no killing and constant branching rate $\rho/2$, that is
    \begin{equation}
        \label{ub:bbm}
        \int_0^Lp_{\delta u -s}(y,z)\leqslant e^{\frac{\rho}{2}(\delta u -s)}.
    \end{equation}
    We use this upper bound  for $\delta u-s\in[0,cL]$ for some constant $c>0$ to determine. Consider
    \begin{equation*}
        \tilde U_2:= \int_0^L 2r(y)\int_{\delta u}^{\delta u -c  L}p_s(x,y)\left(\int_0^Lp_{\delta u-s}(y,z) dz\right)^2\;ds \,dy.
    \end{equation*}
    Note that, for sufficiently large $L$, Equation (\ref{rk:lem:p}) holds for $s$ instead of $\delta u$ for all $s\in[\delta u, \delta u-cL]$. Hence, Equation \eqref{ub:bbm} combined with \eqref{rk:lem:p} and \eqref{rem:21} implies that
    \begin{align*}
        \tilde U_2 & \leqslant C\int_0^L \int_{\delta u-c L}^{\delta u} p_s(x,y) e^{\rho (\delta u-s)}ds\, dy                                                      \\
                   & \leqslant  Ce^{(\mu-\beta)L}e^{\mu(x-L)}w_1(x)\int_0^L e^{-\mu y}v_1(y)\left(\int_{\delta u -c L}^{\delta u}e^{\rho (\delta u-s)}ds \right)dy \\
                   & \leqslant  Ce^{(\mu-\beta)L}e^{\mu(x-L)}w_1(x)e^{\rho c L}\int_0^L e^{-\mu y}v_1(y)dy                                                         \\
                   & \leqslant  C e^{(\mu-\beta)L}e^{\mu(x-L)}w_1(x)e^{\rho c L}.
    \end{align*}
    We then fix $c<\frac {\mu-\beta}{\rho}$ so that, for $L$ large enough
    \begin{equation}
        \tilde U_2\leqslant C e^{2(\mu-\beta)L}\left(e^{(c\rho -(\mu-\beta))L}\right)e^{\mu(x-L)}w_1(x)\leqslant C e^{2(\mu-\beta)L}\delta e^{\mu(x-L)}w_1(x). \label{est:A2:N}
    \end{equation}
    We now control the remaining part of the time integral, i.e.~for $s\in [\delta u-c_{\ref{th1}}L, \delta u -cL]$.
    To this end, we make use of Lemma \ref{lem:212}: we get that for $\delta u-s>1$,\begin{equation*}
        p_{\delta u-s}(y,z)\leqslant Ce^{\mu(y-z)}\left(v_1(y)v_1(z)+Le^{-\linf (\delta u-s) }\right),
    \end{equation*}
    so that
    \begin{equation}\label{ub:int2}
        \left(\int_0^L p_{\delta u-s}(y,z) dz\right)^2\leqslant Ce^{2\mu y}\left(v_1(y)^2+L^2e^{-2\linf (\delta u-s)}\right).
    \end{equation}
    (recall that  $\int_0^L e^{-\mu z}v_1(z)dz$ converges to a positive limit).
    In addition, note that for sufficiently large $L$, Equation (\ref{rk:lem:p}) holds for $s$ instead of $\delta u$ for all $s\in\left[\delta u -c_{\ref{th1}}L, \delta u-cL\right]$. Therefore,
    \begin{eqnarray*}
        \tilde U_3&:=&\int_0^L2r(y)\int_{\delta u-c_{\ref{th1}}L}^{\delta u-c L}p_s(x,y)\left(\int_0^Lp_{\delta u -s}(y,z)dz\right)^2 ds\;dy\\
        &\leqslant &   \tilde U_{3,1}+  \tilde U_{3,2},
    \end{eqnarray*}
    with
    \begin{equation*}
        \tilde U_{3,1}\leqslant Ce^{\mu x}v_1(x) \int_0^L\int_{\delta u-c_{\ref{th1}}L}^{\delta u-c L} e^{\mu y}v_1^3 (y) ds \,dy,
    \end{equation*}
    and
    \begin{equation*}
        \tilde U_{3,2}\leqslant C L^2 e^{\mu x}v_1(x) \int_0^L\int_{\delta u-c_{\ref{th1}}L}^{\delta u-c L} e^{\mu y}v_1(y)e^{-2\linf (\delta u-s)}ds \,dy.
    \end{equation*}
    Recalling that $\mu>3\beta$ under \eqref{hwp} (see \eqref{rem:alpha}) and using Lemma \ref{ubv1} and Equation \eqref{rem:21}, we get that
    \begin{align}
        \tilde U_{3,1} & \leqslant CLe^{\mu x}v_1(x) \left(\int_0^L e^{\mu y}v_1^3 (y)  dy \right)\leqslant   CL e^{(\mu-\beta)L}e^{\mu(x-L)}w_1(x) e^{(\mu-3\beta)L}\nonumber  \\
                       & \leqslant  C \left(Le^{-2\beta L}\right) e^{2(\mu-\beta)L}e^{\mu(x-L)}w_1(x)\leqslant  C \delta e^{2(\mu-\beta)L}e^{\mu(x-L)}w_1(x) \label{est:A31:N},
    \end{align}
    for $L$ large enough.
    On the other hand,  Lemma \ref{ubv1} and Equation \eqref{rem:21} yield
    \begin{align}
        \tilde U_{3,2} & \leqslant CL^2 e^{\mu x}v_1(x)\left(\int_0^L e^{\mu y} v_1(y)dy\right)\left(\int_{\delta u-c_{\ref{th1}}L}^{\delta u-c L} e^{-2\linf (\delta u-s)}ds \right)\nonumber \\
                       & \leqslant  C e^{2(\mu-\beta)L}e^{\mu(x-L)}w_1(x) \left(L^2e^{-2\linf c L}\right)\leqslant C\delta e^{2(\mu-\beta)L}e^{\mu(x-L)}w_1(x),\label{est:A32:N}
    \end{align}
    for $L$ large enough.
    Finally, we obtain the lemma by combining Equations \eqref{est:A1:N}, \eqref{est:A2:N},   \eqref{est:A31:N} and \eqref{est:A32:N}.
\end{proof}

\begin{proof}[Proof of Lemma \ref{lemf}] Let $\gamma>0$. Let us  prove that for $L$ large enough, we have
    \begin{equation*}
        \mathbb{P}\left(|e^{-(\mu-\beta)L}N_u-d_\infty Z_u|>\gamma \right)=\mathbb{P}\left(|N_u-d_\infty e^{(\mu-\beta)L} Z_u|>\gamma e^{(\mu-\beta)L}\right)<\gamma.
    \end{equation*}
    As explained above, we use that for $u>0$
    \begin{eqnarray*}
        |N_u-d_\infty e^{(\mu-\beta)L}Z_u|&\leqslant &|N_u-N_u'|+|N_u'-d_\infty e^{(\mu-\beta)L}Z'_{(1-\delta )u}|\\&&+d_\infty e^{(\mu-\beta)L}|Z'_{(1-\delta )u}-Z_{u}'|+d_\infty e^{(\mu-\beta)L}|Z_{u}'-Z_{u}|,
    \end{eqnarray*}
    and that each quantity on the RHS is small as long as $L$ is large enough and $\delta$ is small enough.

    First, we choose $\delta >0$ of the form
    \begin{equation*}
        \delta =\theta e^{2\beta A},
    \end{equation*}
    where $A$ and $\theta$ are defined in the beginning of Section \ref{sec:smallts}. As in Section \ref{sec:smallts}, we will first let $L$ to $\infty$, then, $\delta$ to $0$ (or equivalently $A\to\infty$). We also recall the definitions of the subdivision $(t_k)_{k=0}^{  {\kappa}}$ from Equation \eqref{disc:t}  and of the events $(G_k)$ defined in \eqref{def:Gk}. Note that with this notation, we have  $u=t_{  {\kappa}}$ and $(1-\delta )u=t_{  {\kappa}-1}$.

    Since the variances of $N_u'$ and $Z_u'$ are both bounded by a quantity that depends on $Z'_{(1-\delta) u}$ (see Lemma \ref{lem:smZ} and Lemma \ref{lem2:N}), we first control $Z'_{(1-\delta) u}$ on $G_{  {\kappa}}$. Recall from Lemma \ref{lem:fmZ} that $Z'_{t_j}$ is a supermartingale. Thus, one can prove (for example using Doob's martingale inequality) that
    \begin{equation*}
        \mathbb{P}\left(\max_{0\leqslant j\leqslant   {\kappa}} Z_{t_j}'>B\gamma^{-1}  \right)\leqslant \gamma \frac{\mathbb{E}[Z_0]}{B}.
    \end{equation*}
    Let $E_{\gamma}^{  {\kappa}}=G_{  {\kappa}}\cap \left\{\max_{0\leqslant j\leqslant   {\kappa}} Z_{t_j}'\leqslant B\gamma^{-1} \right\}. $ Remark that one can choose $B$ large enough so that $\mathbb{P}(E_{\gamma}^{  {\kappa}})\geqslant 1-\gamma/4$. From now on, we consider the event $E_{\gamma}^{  {\kappa}}$ corresponding to this choice of $B$.   {Remarking that $E_\gamma^\kappa\subset G_\kappa$, we see that for $L$ large enough, }
    \begin{equation}\label{eq:YZE}
        \text{  on $E_\gamma^{  {\kappa}}$,} \quad Y_{(1-\delta)u}\leqslant \delta Z_{(1-\delta)u}.
    \end{equation}

    We now bound the quantities $|N_u'-d_\infty e^{(\mu-\beta)L}Z'_{(1-\delta )u}|$ and $d_\infty e^{(\mu-\beta)L}|Z'_{(1-\delta )u}-Z_{u}'|$ with high probability.   {In both cases, we will use our first and second moment estimates combined with Chebyshev's inequality.} First, we recall from Lemma \ref{lem:fmZ} that
    \begin{equation*}
        \mathbb{E}\left[Z'_{u}|\mathcal{F}_{(1-\delta)u}\right]= (1+O(\delta))Z'_{(1-\delta)u}.  \end{equation*}
      {Note that,  conditional on $\mathcal{F}_{(1-\delta)u},$  the particles alive at time $(1-\delta)u$ evolve independently between times $(1-\delta)u$ and $u$. Hence the conditional variance $\text{Var}\left[Z'_{u}|\mathcal{F}_{(1-\delta)u}\right]$ is equal to the sum of the conditional variances of the contribution to $Z'_u$ from the particles alive at time $(1-\delta)u$.}
    Lemma \ref{lem:smZ} then entails
    \begin{equation*} 
        \text{Var}\left[Z'_{u}|\mathcal{F}_{(1-\delta)u}\right]\leqslant C\left(\delta Z'_{(1-\delta)u}+Y_{(1-\delta)u}\right),\end{equation*}
    for $\delta>0$ and $L$ large enough.
    Therefore,  Chebyshev's inequality,   {together with \eqref{eq:YZE}},  implies that for $L$ large enough and $\delta$ small enough,
    \begin{equation*}
        \text{on $E_\gamma^{  {\kappa}}$,} \quad
        \mathbb{P}\left( \left|Z'_{u}-\mathbb{E}\left[Z'_{u}\big|\mathcal{F}_{(1-\delta)u}\right]\right|>\frac{\gamma}{2}\big|\mathcal{F}_{(1-\delta)u}\right)\leqslant C\delta B\gamma^{-2}\leqslant \frac{\gamma}{2}.
    \end{equation*}
    Moreover, we know that, on $E_\gamma^{  {\kappa}}$, $\delta Z'_{(1-\delta)u}\leqslant \delta B\gamma^{-1}\leqslant  \frac{\gamma}{2}$  for $L$ large enough and $\delta$ small enough. Hence, we obtain that for sufficiently large $L$ and sufficiently small $\delta$, we have
    \begin{eqnarray}
        \mathbb{P}\left( \left|Z'_{u}-Z'_{(1-\delta)u}\right|>\gamma\right)&\leqslant&\mathbb{P}\left(\left|Z'_{u}-Z'_{(1-\delta)u}\right|>\gamma ,E_\gamma^{  {\kappa}}\right)+\mathbb{P}((E^{  {\kappa}}_\gamma)^c)\nonumber \\
        &\leqslant& \frac{\gamma}{2}+\mathbb{P}((E^{  {\kappa}}_\gamma)^c)\leqslant \gamma. \label{eq:condE} \end{eqnarray}
    Similarly, we get thanks to Lemma \ref{lem1:N} and Lemma \ref{lem2:N} that for $L$ large enough and $\delta$ small enough,
    \begin{eqnarray}\label{condEN}
        \mathbb{P}\left( \left|N'_{u}-d_\infty e^{(\mu-\beta)L}Z'_{(1-\delta)u}\right|>\gamma e^{(\mu-\beta)L}\right)&\leqslant&\gamma.\end{eqnarray}

    Recall that, if the process starts with all its particles to the left of $L$ at time $(1-\delta)u$, then $Z'_{(1-\delta) u}=Z_{(1-\delta) u}$ and   $Z'_u=Z_u$ with high probability.  This is a consequence of Lemma \ref{th:r} combined with Markov's inequality:  on $E_\gamma^{  {\kappa}}$,
    \begin{equation*}
        \mathbb{P}\left(\left|Z_u-Z_u'\right|>0\big|\mathcal{F}_{(1-\delta)u}\right)\leqslant \mathbb{P}\left(R([0,\delta u])\geqslant 1\big|\mathcal{F}_{(1-\delta)u}\right)\leqslant C\left(\delta Z'_{(1-\delta)u}+Y_{(1-\delta)u}\right)\leqslant \frac{\gamma}{2},
    \end{equation*}
    for $L$ large enough and $\delta$ small enough. Using the same argument as in Equation \eqref{eq:condE}, we get that
    \begin{equation}
        \mathbb{P}\left(\left|Z_u-Z_u'\right|>0\right)\leqslant \gamma.\label{ub:ZZ}
    \end{equation}
    Similarly, we get that for large $L$ and sufficiently small $\delta$, we have
    \begin{equation}
        \mathbb{P}\left(\left|N'_u-N_u\right|>0\right)\leqslant \gamma.\label{ubNN}
    \end{equation}
    Combining \eqref{eq:condE}, \eqref{condEN}, \eqref{ub:ZZ} and \eqref{ubNN}, we get that for  $L$ large enough,\begin{equation*}
        \mathbb{P}\left(|e^{-(\mu-\beta)L}N_u-d_\infty  Z_u|>2\gamma\right)\leqslant 4\gamma,
    \end{equation*}
    which concludes the proof of the lemma.
\end{proof}

\section{The case $\alpha>2$: the fully pushed regime}

In this section, we briefly outline the adjustments required to prove the second  conjecture stated at the end of Section \ref{sec:model}. We leave the details for future work.

In the fully pushed regime, we have $\rho>\rho_2$ so that $\mu<3\beta$ (see Equation  \eqref{eq:alpha}).  We expect (see Equation \eqref{eq:NL}) the genealogy to evolve over the timescale $$N=e^{(\mu-\beta)L}.$$
For $t>0$ fixed, we see from Lemma \ref{lem:fmZ}  that
\begin{equation*}
    \mathbb{E}[Z_{tN}']=(1+\vep_L)Z_0',
\end{equation*}
and from Lemma  \ref{th:r}  that $$\mathbb{E}[R([0,tN])]=\vep_L Z_0',$$
since $\mu<3\beta$. Essentially, this means that for $N$ large enough, no particle reaches the boundary $L$ on the time scale $N$.
On the other hand, the analysis of the additive martingale conducted in Section \ref{sec:W} would provide a similar convergence result.  Equation \eqref{def:W} would hold
for some random variable $W$ satisfying
\begin{equation*}
    \mathbb{E}\left[e^{-{  q} W}\right]=\exp\left(-{  q}+o({  q}^2)\right), \quad {  q} \to0.
\end{equation*}
As a consequence, a finer estimate of the second moment of $Z'$ is required to prove Proposition \ref{prop7} for $\alpha=2$.   The arguments developed in Sections \ref{sec:smallts} and \ref{sec:cvcsbp} would then be similar.

From a biological standpoint, the fact that the particles do not exit $(0,L)$ indicates that the invasion is driven by the particles living in the bulk (i.e. that stay far from $L$).

\begin{appendix}
\addtocontents{toc}{\protect\setcounter{tocdepth}{-1}}
    \section{Proof of Proposition \ref{th:Trho}} \label{proof:15}

    The properties of $  {h} $ can easily be checked and we only prove the expression of $\lambda_c(\rho)$.

    We first show that $\lambda_c(\rho) \ge 0$ for all $\rho\in\mathbb{R}$. Let $\lambda < 0$. Assume that there exists  $u\in \mathcal D_{T_\rho}$ such that $T_\rho u = \lambda u$. Then $u'' = 2\lambda u$ on $[1,\infty)$ so that $u(x) = A \sin(\sqrt{-2\lambda}x) + B\cos(\sqrt{-2\lambda}x)$ on $[1,\infty)$ for some $A,B\in\mathbb{R}$. Then $u$ changes sign on $[1,\infty)$ so that we do not have $u>0$ on $(0,\infty)$. Hence, $\lambda_c(\rho) \ge 0$ for all $\rho\in\mathbb{R}$.

    We now claim that $\lambda_c(\rho) = 0$ for $\rho \le \rho_c$. Since $\lambda_c(\rho)$ is increasing, it is enough to show that $\lambda_c(\rho_c) = 0$. Define
    \[
        u(x) = \begin{cases}
            \sin(\frac{\pi}{2}x) & x\in[0,1] \\
            1                    & x\ge 1.
        \end{cases}
    \]
    Then $u\in C^1((0,\infty))\cap C^2((0,1)\cup(1,\infty))$, $u(0) = 0$ and $u>0$ on $(0,\infty)$. Moreover, $T_{\rho_c}u(x) = 0$ for $x\in (0,1)\cup(1,\infty)$. Hence, $u\in \mathcal D_{T_{\rho_c}}$ and $T_{\rho_c} u = 0$. It follows that $\lambda_c(\rho_c) = 0.$

    Let $\rho > \rho_c$ and $\lambda \in(0,\rho/2)$. Let $u\in \mathcal D_{T_\rho}$ such that $u>0$ on $(0,\infty)$ and $T_\rho u = \lambda u$. Then $u'' = (2\lambda - \rho) u$ on $(0,1)$. Since $\lim_{x\to0}u(x) = 0$, there exists a constant $A\in\mathbb{R}$ such that  $u(x) = A\sin(\sqrt{\rho-2\lambda}x)$ for all $x\in(0,1)$. Since $u>0$, we have $\rho - 2\lambda < \pi^2$ and $A>0$. Suppose (without loss of generality) that $A=1$ so that $$u(x) = \sin(\sqrt{\rho-2\lambda}x), \quad x\in(0,1).$$
    For $x\in(1,\infty)$, we have $u''(x) = 2\lambda u(x)$, so that $$u(x) = A\cosh(\sqrt{2\lambda}(x-1)) + B\sinh(\sqrt{2\lambda}(x-1)), \quad x\in(0,1),$$ for some $A,B\in\mathbb{R}$. Since $u$ and $u'$ are continuous at $1$, we have $A = u(1) = \sin(\sqrt{\rho-2\lambda})$ and $B = u'(1) = \sqrt{\rho-2\lambda}\cos(\sqrt{\rho-2\lambda})$. Furthermore, we have $u > 0$ on $(1,\infty)$ if and only if  $A+B \ge 0$, which holds if and only if
    \begin{equation}\label{eq:ineq}
        \frac{\cos(\sqrt{\rho-2\lambda})}{\sqrt{2\lambda}} \ge -\frac{\sin(\sqrt{\rho-2\lambda})}{\sqrt{\rho-2\lambda}}.
    \end{equation}
    Moreover, all of these conditions are also sufficient: if $\lambda\in J_\rho := (0\vee \frac12(\rho-\pi^2),\rho/2)$ satisfies \eqref{eq:ineq}, then the function $u$ defined by
    \[
        u(x) = \begin{cases}
            \sin(\sqrt{\rho-2\lambda}x)                            & x\in[0,1] \\
            \sin(\sqrt{\rho-2\lambda}) \cosh(\sqrt{2\lambda}(x-1)) &           \\\qquad  \quad + \;\sqrt{\rho-2\lambda}\cos(\sqrt{\rho-2\lambda})\sinh(\sqrt{2\lambda}(x-1))&x\ge 1,
        \end{cases}
    \]
    satisfies $u\in C^1((0,\infty))\cap C^2((0,1)\cap(1,\infty))$, $u>0$ on $(0,1)\cup(1,\infty)$, $\lim_{x\to0} u(x) = 0$ and $T_\rho u(x) = \lambda u(x)$ for $x\in (0,1)\cup (1,\infty)$. Hence, by continuity of $u$, $T_\rho u(1) = \lim_{x\to1} T_\rho u(x)$ exists and is equal to  $\lambda u(1)$. Thus $u\in\mathcal D_{T_\rho}$, $u>0$ on $(0,\infty)$ (by continuity) and $T_\rho u = \lambda u$.

    Let $\Lambda_\rho$ be the set of those $\lambda\in J_\rho$ such that \eqref{eq:ineq} holds. It remains to show that $\inf \Lambda_{\rho}= \lambda_c(\rho)$, where $\lambda_c(\rho)$ is as in the statement of the result. Note that $\sqrt{\rho-2\lambda} \in (0,\pi)$ for $\lambda\in J_\rho$, so that we can rewrite \eqref{eq:ineq} as
    \begin{equation}
        \label{eq:ineq2}
        \sqrt{2\lambda} \ge -\sqrt{\rho-2\lambda}\cot(\sqrt{\rho-2\lambda}).
    \end{equation}
    Now the right-hand side is a decreasing function of $\lambda$ on the interval $J_\rho' := (0\vee \frac12(\rho-\pi^2),\frac12(\rho-\rho_c))$, and the left-hand side an increasing function of $\lambda$. Admit for the moment that $\lambda_c = \lambda_c(\rho)$ as defined in the statement of the theorem yields equality in \eqref{eq:ineq2} and that $\lambda_c \in J_\rho'$. It then follows that $\Lambda_\rho \cap J_\rho' = [\lambda_c,\frac12(\rho-\rho_c)]$, and in particular, $\lambda_c = \inf\Lambda_\rho$.

    It remains to show that $\lambda_c \in J_\rho'$ and that $\lambda_c$ yields equality in \eqref{eq:ineq2}. Note that $  {h} ^{-1}(\rho) \in (\rho_c,\pi^2)$ for all $\rho> \rho_c$. Hence $\lambda_c \in (\frac12(\rho-\pi^2),\frac12(\rho-\rho_c))$. Furthermore, $  {h} ^{-1}(\rho) < \rho$ for all $\rho > \rho_c$, since $  {h} (\rho) > \rho$ for all $\rho>\rho_c$ by the properties of $  {h} $ stated in the theorem. Hence, $\lambda_c > 0$. It follows that $\lambda_c \in J_\rho'$.

    On the interval $J_\rho'$, the right-hand side of \eqref{eq:ineq2} is positive. Thus this equality holds in \eqref{eq:ineq2} if and only if
    \[
        2\lambda = (\rho-2\lambda)\cot(\sqrt{\rho-2\lambda})^2 = (\rho-2\lambda)(\sin(\sqrt{\rho-2\lambda})^{-2} - 1) =   {h} (\rho-2\lambda) - (\rho-2\lambda),
    \]
    that is, if and only if
    \[
          {h} (\rho-2\lambda) = \rho,
    \]
    which is exactly satisfied for $\lambda = \lambda_c$. This concludes the proof.

    \section{Estimates on the eigenvectors. Proof of Lemma \ref{lem:est:evk}}\label{appendix}

    In this section, we give a proof of Lemma \ref{lem:est:evk}. Recall that, under \eqref{hpushed}, we have $K\geqslant 1$ in Lemmas \ref{eigenvloc} and  \ref{ev:sc}. Since the distribution of the eigenvalues is similar in these two lemmas, all the bounds will be calculated for $\rho\neq 1+\left(n-\frac{1}{2}\right)^2\pi^2$, $n\in\mathbb{N}$, but the results can easily be extended to $\left\{\rho>\rho_1\right\}$.

    \begin{lemma}
        Assume that (\ref{hpushed}) holds. There exists $C>0,$ such that for all $L$ large enough
        \begin{equation*}
            \|v_1\|^2\leqslant C.
        \end{equation*}
        \label{lem:norm1}
    \end{lemma}
    \begin{proof}
        This result follows directly from Lemma \ref{lem:l2norm}.
    \end{proof}
    \begin{lemma}
        \label{lem:normk}
        Assume that (\ref{hpushed}) holds. There exists $C>0$ such that for $L$ large enough and $k\in\llbracket1,K\rrbracket$,
        \begin{equation*}
            \|v_k\|^2\geqslant C.
        \end{equation*}
    \end{lemma}
    \begin{proof} The proof is similar to the case $k=1$. Let $k\in\llbracket1,K\rrbracket$. We proved in Lemma \ref{eigenvlocasymp} that $\lambda_k$ converges to a positive limit $ \lambda_k^\infty$ satisfying \eqref{eqlim} and such that
        $\sqrt{\rho-1-\lambda_k^\infty} \in \left(\left(k-\frac{1}{2}\right)\pi,k\pi\right).$
        Moreover, by definition of $v_k$, we have
        \begin{equation*} 
            \|v_k\|^2=\int_0^Lv_k(x)^2dx=\frac{1-\frac{\sin(2\sqrt{\rho-1-2\lambda_k})}{2\sqrt{\rho-1-2\lambda_k}}}{2\sin(\sqrt{\rho-1-2\lambda_k})^2}+\frac{\frac{\sinh(2\sqrt{2\lambda_k}(L-1))}{2\sqrt{2\lambda_k}}-(L-1)}{2\sinh(\sqrt{2\lambda_k}(L-1))^2}.
        \end{equation*}
        Hence, one can explicitly compute the limit of $\|v_k\|$ as in the proof of Lemma \ref{lem:l2norm} and show that this limit is positive. This entails the result.
    \end{proof}
    \begin{lemma}
        Assume that (\ref{hpushed}) holds. There exists $C>0$  such that for $L$ large enough and $k>K$,
        \begin{equation*}
            \|v_k\|^2\geqslant \frac{C}{\sin(\sqrt{-2\lambda_k}(L-1))^2\wedge\sin(\sqrt{\rho-1-2\lambda_k})^2}.
        \end{equation*}
        \label{lem:normK}
    \end{lemma}
    \begin{proof} For $k>K$,
        $$\|v_k\|^2=\frac{1-\frac{\sin(2\sqrt{\rho-1-2\lambda_k})}{2\sqrt{\rho-1-2\lambda_k}}}{2\sin(\sqrt{\rho-1-2\lambda_k})^2}+\frac{(L-1)-\frac{\sin(2\sqrt{-2\lambda_k}(L-1))}{2\sqrt{-2\lambda_k}}}{2\sin(\sqrt{-2\lambda_k}(L-1))^2}.$$ Both terms are non negative. Furthermore,  since $\lambda_k<0$, we have $\sqrt{\rho-1-2\lambda_k}>\sqrt{\rho-1}>\frac{\pi}{2}>0$ so that there exists $C>0$ such that
        \begin{equation*}
            1-\frac{\sin(2\sqrt{\rho-1-2\lambda_k})}{2\sqrt{\rho-1-2\lambda_k}}\geqslant C, \quad \forall k>K.
        \end{equation*}
        Moreover, recall from Lemma \ref{eigenvloc} that $\lambda_k<-a_1$ for all $k>K$. Therefore, we have
        \begin{equation*}
            \left|\frac{\sin(2\sqrt{-2\lambda_k}(L-1))}{2\sqrt{-2\lambda_k}}\right|\leqslant \frac{1}{2\sqrt{ 2 a_1}},
        \end{equation*}
        and
        $$L-1-\frac{\sin(2\sqrt{-2\lambda_k}(L-1))}{2\sqrt{-2\lambda_k}}>L-1-\frac{1}{2\sqrt{2 a_1}}>L-1-\frac{1}{\pi}(L-1)=\left(1-\frac{1}{\pi}\right)(L-1).$$
        As a consequence,
        \begin{equation*}\|v_k\|^2\geqslant C\left(\frac{L-1}{2\sin(\sqrt{-2\lambda_k}(L-1))^2}\vee\frac{1}{2\sin(\sqrt{\rho-1-2\lambda_k})^2}\right),
        \end{equation*}
        for all $k>K$ and $L$ large enough.
    \end{proof}

    \begin{lemma}\label{upperboundratio1}
        Assume that (\ref{hpushed}) holds. There exists $C>0$ such that for $L$ large enough, $x\in[0,L]$ and $k\leqslant K$, we have
        \begin{equation*}
            \left|v_k(x)\right|\leqslant Ce^{\beta L} v_1(x).
        \end{equation*}
    \end{lemma}
    \begin{proof}
        Recall from Lemma \ref{ubv1} that for sufficiently large $L$, we have
        \begin{equation*}
            v_1(x)\geqslant C(x\wedge 1\wedge (L-x))e^{-\beta x}.
        \end{equation*}
        Similarly, one can easily prove the existence of a constant $C>0$ such that
        \begin{equation*}
            |v_k(x)|\leqslant C(x\wedge 1\wedge (L-x))e^{-\sqrt{2\lambda_k^\infty} x}, \quad \forall k\leq K.
        \end{equation*}
        Hence, we see that for all $k\leq K$, we have
        \begin{equation*}
            |v_k(x)|\leqslant  C(x\wedge 1\wedge (L-x))\leqslant e^{\beta x}v_1(x)\leqslant e^{\beta L}v_1(x).
        \end{equation*}
    \end{proof}

    \begin{lemma}
        Assume that (\ref{hpushed}) holds. There exists $C>0$ such that for $L$ large enough, $x\in[0,1]$ and $k\in \mathbb{N}$, we have
        \begin{equation*}
            \left|{v_k(x)}\right|\leqslant C  \left| \frac{\sqrt{\rho-1-2\lambda_k}}{\sin(\sqrt{\rho-1-2\lambda_k})}\right|v_1(x).
        \end{equation*}
    \end{lemma}
    \begin{proof} We know from Lemma \ref{ubv1} that $v_1(x)\geqslant Cx$, for all $x\in[0,1]$. Then, using that $|\sin y|\leqslant y$ for all $y\in\mathbb{R}$, we see that
        \begin{equation*}
            |v_k(x)|\leqslant \frac{\sqrt{\rho-1-2\lambda_k}}{|\sin(\sqrt{\rho-1-2\lambda_k})|}\; x, \quad \forall x\in[0,1].
        \end{equation*}
        This concludes the proof of the lemma.
    \end{proof}

    \begin{lemma}
        Assume that  \eqref{hpushed} holds. There exists $C>0$ such that for $L$ large enough,  $k>K$ and $x\in[1,L]$, we have
        $$  \left|{v_k(x)}\right|\leqslant C \frac{\sqrt{\rho-1-2\lambda_k}}{|\sin(\sqrt{-2\lambda_k}(L-1))|} e^{\beta L} v_1(x).$$
    \end{lemma}
    \begin{proof}
        Recall from Lemma \ref{ubv1} that for sufficiently large $L$, we have
        \begin{equation*}
            v_1(x)\geqslant C(x\wedge 1\wedge (L-x))e^{-\beta x}.
        \end{equation*}
        Then, note that for $x\in[1,L-1]$,
        \begin{equation*}
            |\sin(\sqrt{-2\lambda_k}(L-x))|\leqslant 1\leqslant C\sqrt{\rho-1}\leqslant C\sqrt{\rho-1-2\lambda_k}\leqslant Ce^{\beta x}\sqrt{\rho-1-2\lambda_k}v_1(x).
        \end{equation*}
        On the other hand, for $x\in[L-1,L],$
        \begin{align*}
            |\sin(\sqrt{-2\lambda_k}(L-x))| & \leqslant \sqrt{-2\lambda_k}(L-x)\leqslant \sqrt{\rho-1-2\lambda_k}(L-x) \\
                                            & \leqslant C\sqrt{\rho-1-2\lambda_k}e^{\beta x}v_1(x),
        \end{align*}
        which concludes the proof of the lemma.
    \end{proof}
    \begin{lemma}
        Assume that \eqref{hpushed} holds. There exists $C>0$ such that for $L$ large enough, $k>K$ and $x\in[0,L]$, we have
        $$  \left|{v_k(x)}\right|\leqslant C \|v_k\|.$$
    \end{lemma}
    \begin{proof}
        This is a straightforward consequence of Lemma \ref{lem:normK}.
    \end{proof}

    \section{The Green function}\label{appendix:greenf}
    \begin{lemma}\label{lemma:ubsf}
        Assume that \eqref{hpushed} holds. There exist $C>0$ and $\delta >0$ such that for $L$ sufficiently large, $x\in[0,L]$ and $
            \xi \in(0,\delta)$,
        \begin{align}\psi_{{\linf+\xi}}(x)  & \leqslant C\sinh\left(\sqrt{2{(\linf+\xi)}}(L-x)\right), \label{psiub}                                                                                                                        \\
             \varphi_{\linf+\xi}(x) & \leqslant C (1\wedge x)\left({  \tilde f_1}({\linf+\xi})e^{\sqrt{2(\linf+\xi)}(x-1)}+{  \tilde f_2}({\linf+\xi})e^{-\sqrt{2{(\linf+\xi)}}(x-1)}\right). \label{phiub}
        \end{align}
    \end{lemma}

    \begin{proof}
        According to Equation (\ref{nt:psi}), it is sufficient to prove that (\ref{psiub}) holds in  $[0,1]$. Yet, for $x\in [0,1],$ we have
        \begin{equation}
            \psi_{\linf+\xi}(x)\leqslant|\psi_{\linf+\xi}(x)|\leqslant\left(1+\frac{\sqrt{2({\linf+\xi})}}{\sqrt{\rho-1-2{(\linf+\xi)}}}\right)e^{\sqrt{2{(\linf+\xi)}}(L-1)}. \label{eq:281}
        \end{equation} Since $\sqrt{\rho-1-2\linf}\in(\pi/2,\pi)$, there exists $\delta>0$ such that
        \begin{equation}
            \pi/2<\sqrt{\rho-1-2({\linf+\xi})}<\pi,\label{eq:281:bis}
        \end{equation}
        and
        \begin{equation*}
            \sqrt{2({\linf+\xi})}<\sqrt{\rho-1-\pi^2/4},
        \end{equation*}
        for all $|\xi|<\delta$. Therefore, for all $\xi\in(-\delta,\delta)$, we have
        \begin{equation}\label{eq:2820}\frac{\sqrt{2({\linf+\xi})}}{\sqrt{\rho-1-2({\linf+\xi})}}\leqslant 2\frac{\sqrt{\rho-1-\pi^2/4}}{\pi}.
        \end{equation}
        Moreover, we know that for $x\in[0,1]$ and $\xi\in(0,\delta),$
        \begin{eqnarray*}
            \sinh\left(\sqrt{2({\linf+\xi})}(L-x)\right)&\geqslant&\sinh\left(\sqrt{2({\linf+\xi})}(L-1)\right)\\
            &=&  \frac{1}{2}(1-e^{-2\sqrt{2\linf+\xi}(L-1)})e^{\sqrt{2{(\linf+\xi)}}(L-1)}\\
            &\geqslant& \frac{1}{2}(1-e^{-2\sqrt{2\linf}(L-1)})e^{\sqrt{2{(\linf+\xi)}}(L-1)}\\
            &\geqslant& \frac{1}{4} e^{\sqrt{2{(\linf+\xi)}}(L-1)},
        \end{eqnarray*}
        for $L$ large enough (that does not depend on~$\delta$). This estimate, combined with Equations (\ref{eq:281}) and (\ref{eq:2820}), concludes the proof of  (\ref{psiub}).

        In order to prove (\ref{phiub}), we use  that ${  \tilde f_1}(\linf)=0$, ${  \tilde f_1}'(\linf)>0$ and ${  \tilde f_2}(\linf)>0$. Therefore, without loss of generality, we have ${  \tilde f_1}({\linf+\xi})>0$ and ${  \tilde f_2}({\linf+\xi})>\frac{1}{2}{  \tilde f_2}(\linf)$  for all $\xi\in(0,\delta)$.
        Thus, for $x\in [0,1]$ and  ${\xi}\in(0,\delta)$, we have
        \begin{equation}
            {  \tilde f_1}({\linf+\xi})e^{\sqrt{2({\linf+\xi})}(x-1)}+{  \tilde f_2}({\linf+\xi})e^{-\sqrt{2({\linf+\xi})}(x-1)}\geqslant \frac{{  \tilde f_2}(\linf)}{2}.\label{fg:xi}
        \end{equation}
        Besides, combining Equations (\ref{nt:phi}) and (\ref{eq:281:bis}), we obtain that for $x\in[0,1]$ and ${\xi}\in(0,\delta)$,
        $$\varphi_{\linf+\xi}(x)=\sin\left(\sqrt{\rho-1-2({\linf+\xi})}\,x\right)\leqslant \sqrt{\rho-1-2({\linf+\xi})} x \leqslant \pi x.$$
        This equation, along with (\ref{fg:xi}), implies that (\ref{phiub}) holds in $[0,1]$ for any $C>\frac{2\pi}{{  \tilde f_2}(\linf)}$. Finally, note that for  ${\xi}\in(0,\delta)$,
        $\left(2({\linf+\xi})\right)^{-\frac{1}{2}}<\left(2\linf\right)^{-\frac{1}{2}},$
        so that (\ref{phiub}) holds in $[0,L]$ for any $C>\max\left(\frac{1}{\sqrt{2\linf}},\frac{2\pi}{{  \tilde f_2}(\linf)}\right).$
    \end{proof}

    \begin{proof}[Proof of Lemma \ref{lem:alphat}] Since $$\sqrt{2(\linf+\xi)}-\sqrt{2\linf}=\frac{2\xi}{\sqrt{2\linf}+\sqrt{2(\linf+\xi)}}\sim\frac{1}{\sqrt{2\linf}}\xi,$$ as $L\to\infty$, we know that for $L$ large enough (that does not depend on $x$), we have
        \begin{equation*}
            e^{\sqrt{2(\linf+\xi)}(L-x)}\leqslant e^{\sqrt{2\linf}(L-x)}e^{\frac{2}{\sqrt{2\linf}}\xi (L-x)}.
        \end{equation*}
        Yet, $\xi(L-x)\leqslant \xi L$ uniformly tends to $0$ as $L\to\infty$. Therefore, for $L$ large enough (that does not depend on $x$), we have
        \begin{equation}
            e^{ \sqrt{2\linf}(L-x)}\leqslant e^{ \sqrt{2(\linf+\xi)}(L-x)}\leqslant 2e^{\sqrt{2\linf}(L-x)}\label{exp:1}.
        \end{equation}
        We then use that ${  \tilde f_2}(\linf)>0$, ${  \tilde f_1}(\linf)=0$ and ${  \tilde f_1}'(\linf)>0$ to claim that
        \begin{align}
            0<\frac{1}{2}{  \tilde f_2}(\linf)<      & {  \tilde f_2}(\linf+\xi(L))<2g(\linf),\label{g:xi}        \\
            \frac{1}{2}{  \tilde f_1}'(\linf)\xi(L)< & {  \tilde f_1}(\linf+\xi(L))<2f'(\linf)\xi(L)\label{f:xi},
        \end{align}
        for $L$ large enough.
        Thus, combining the definition of the Wronskian (\ref{def:walpha}) and Equations (\ref{exp:1}), (\ref{g:xi}) and (\ref{f:xi}), we get that for $L$ large enough,
        \begin{equation*}
            \omega_{\linf+\xi}>{  \tilde f_1}(\linf+\xi)e^{\sqrt{2(\linf+\xi)}(L-1)}>Cf'(\linf)\xi(L)e^{\sqrt{2\linf}L}.
        \end{equation*}
        Then, Equation (\ref{exp:1}) applied to $x=1$, divided by Equation (\ref{exp:1}) implies that for $L$ large enough,
        $$\frac{1}{2}e^{\sqrt{2\linf}(x-1)}\leqslant e^{\sqrt{2(\linf+\xi)}(x-1)}\leqslant 2e^{\sqrt{2\linf}(x-1)}.$$ This inequality combined with Equations (\ref{phiub}) from Lemma \ref{lemma:ubsf}, (\ref{f:xi}) and (\ref{g:xi}) yields the expected control on $\varphi_{\linf+\xi}$.

        The estimate on $\psi_{\linf+\xi}$ can easily be  deduced from Equations (\ref{psiub}) from Lemma \ref{lemma:ubsf} and  from Equation (\ref{exp:1}) on $[0,L-1]$. For $x\in[L-1,L],$ we use that
        $$\sinh\left(\sqrt{2{(\linf+\xi)}}(L-x)\right)\leqslant \sinh\left(\sqrt{2{(\linf+\xi)}}\right)(L-x)\leqslant C(L-x),$$ for $L$ large enough. Putting this together with Equation (\ref{psiub}), we get  that for $L$ large enough and $x\in[L-1,L]$,
        \begin{equation*}
            \psi_{\linf+\xi}(x)\leqslant C(L-x) \leqslant C(L-x)e^{\beta(L-x)},
        \end{equation*}
        which concludes the proof of the lemma.
    \end{proof}

    \begin{proof}[Proof of Lemma \ref{lemma:wro}]
        The proof is similar to that of Lemma \ref{lem:alphat}
        except that
        $$\sqrt{2(\linf+\xi)}-\sqrt{2(\linf+\xi)}\sim\frac{1}{\sqrt{2\linf}}\frac{h}{L},$$ as $L\to\infty$, so that for $L$ large enough (that does not depend on $x$), we have
        \begin{equation*}
            e^{\sqrt{2(\linf+\xi)}(L-x)}\leqslant e^{\sqrt{2\linf}(L-x)}e^{\frac{2}{\sqrt{2\linf}}h}.
        \end{equation*}
    \end{proof}

    \section{Proof of Lemma \ref{est:wy}}
    Only for this proof, we write $\lambda_1^L$ instead of $\lambda_1$ to be able to compare $w_{1,-y}$ and $w_1$. For $x\in [1,L-a_L]$, we have
    \begin{align*}
         & e^{-\beta y}\frac{w_{1,-y}(x)}{w_1(x)}=e^{-\beta y}\frac{\sinh\left(\sqrt{2\lambda_1^{L+y}}(L+y-x)\right)}{\sinh\left(\sqrt{2\lambda_1^{L}}(L-x)\right)}                                                                    \\
         & = \underbrace{\frac{\sinh\left(\sqrt{2\lambda_1^{L+y}}(L+y-x)\right)}{\sinh\left(\beta(L+y-x)\right)}}_{\textstyle =:F(x)}\underbrace{\left(e^{-\beta y}\frac{\sinh\left(\beta(L+y-x)\right)}{\sinh\left(\beta(L-x)\right)}
            \right)}_{\textstyle =:G(x)}\underbrace{\frac{\sinh\left(\beta(L-x)\right)}{\sinh\left(\sqrt{2\lambda_1^{L}}(L-x)\right)}
        }_{\textstyle =:H(x)}.
    \end{align*}
    Note that
    \begin{align*}
        G(x)=\frac{1-2e^{-2\beta y}e^{-2\beta(L-x)}}{1-e^{-2\beta(L-x)}} & \in\left[\frac{1-2e^{-2\beta y}e^{-2\beta a_L}}{1-e^{-2\beta L}},\frac{1-2e^{-2\beta y}e^{-2\beta L}}{1-e^{-2\beta a_L}}\right] \\
                                                                         & \subset\left[\frac{1-e^{-2\beta a_L}}{1-e^{-2\beta L}},\frac{ 1}{1-e^{-2\beta a_L}}\right].
    \end{align*}
    Thus $G$ converges uniformly in $x\in[1,L-a_L]$ and $y\geqslant 0$.
    Using the mean value theorem and Lemma \ref{exp:lambda1}, remarking that $x\mapsto xe^{-x}$ is decreasing on $(1,+\infty)$ and recalling that $\lambda_1^L$ increases with $L$, we get that  for $L$ large enough,
    \begin{align*}
        |F(x)-1| & \leqslant \frac{\cosh\left(\beta(L+y-x)\right)}{\sinh\left(\beta(L+y-x)\right)}\left(\beta-\sqrt{2\lambda_1^{L+y}}\right)(L+y-x)          \\
                 & \leqslant C\frac{e^{-\beta(L+y)}(L+y)}{\tanh(\beta a_L) }\leqslant C\frac{Le^{-\beta L}}{\tanh(\beta a_L) } \xrightarrow[L\to\infty]{} 0.
    \end{align*}
    Using similar arguments, one can prove that as $L\to\infty$
    \begin{equation*}
        |H(x)-1|\to 0 \quad \text{uniformly in $x\in[1,L-a_L]$}.
    \end{equation*}
    Note that $H(x)$ does not depend on $y$. To deal with the case $x\in[0,1]$, we first remark that the previous computations implies that \begin{equation*}
        e^{-\beta y}\frac{\sinh\left(\sqrt{2\lambda_1^{L+y}}(L+y-1)\right)}{\sinh(\sqrt{2\lambda_1^L}(L-1)}\to 1
    \end{equation*} as $L\to \infty$, uniformly in $y\in[0,+\infty)$. Then, one can easily prove (for instance, using the mean value theorem) that, as $L\to\infty$,
    \begin{equation*}
        \frac{\sin\left(\sqrt{\rho-1-2\lambda_1^{L+y}}x\right)}{\sin\left(\sqrt{\rho-1-2\lambda_1^{L+y}}\right)}\frac{\sin\left(\sqrt{\rho-1-2\lambda_1^{L}}\right)}{\sin\left(\sqrt{\rho-1-2\lambda_1^{L}}
            \,x\right)}\to 1
    \end{equation*}
    uniformly in $y\in[0,+\infty)$ and in $x\in[0,1]$, which concludes the proof of  Lemma \ref{est:wy}.
    \label{proof:lem71}
\end{appendix}

\section*{Acknowledgments}

I am grateful to Pascal Maillard and Gaël Raoul for initiating this project and for their guidance. I also thank Jean-Marc Bouclet for useful explanations of spectral methods and Cyril Labbé for helpful discussions, which gave me a better understanding of the problem. This work was partially supported by the grant ANR-20-CE92-0010-01. This project has received funding from the European Union’s Horizon 2020 research and innovation program under the Marie Skłodowska-Curie grant agreement No 101034413. This project was initiated at the École Polytechnique, CMAP, whose support is gratefully acknowledged.

\bibliographystyle{myplain} %
\bibliography{AOP1691.bib}       

\end{document}